\documentclass[reqno]{amsart}

\usepackage{graphicx}
 \usepackage{bbm}
  \usepackage{bm}
\usepackage{amsfonts,dsfont,amsmath,enumerate}
\usepackage{amsbsy}
\usepackage{graphicx}
\usepackage{oldgerm}
\usepackage{yhmath}
\usepackage{wrapfig}
\usepackage[dvipsnames]{xcolor}
\definecolor{greencite}{rgb}{0.2,0.6,0.2}
\definecolor{bluformula}{rgb}{0.1,0.2,0.6}
\definecolor{rosso}{rgb}{1,0,0}

\usepackage[colorlinks=true,linkcolor=bluformula,urlcolor=blue,citecolor=greencite]{hyperref}

\theoremstyle{plain}
\newtheorem{theorem}{Theorem}
\newtheorem{theorem*}{Theorem}
\newtheorem{lemma}{Lemma}
\newtheorem{proposition}{Proposition}
\newtheorem{corollary}{Corollary}

\theoremstyle{definition}
\newtheorem{definition}{Definition}

\theoremstyle{plain}
\newtheorem{remark}{Remark}
\newtheorem*{remark*}{Remark}
\numberwithin{equation}{section}

\def\e{{\epsilon}}

\def\ni{\noindent}

\def\mn{\medskip\noindent}

\def\lb{\left|}
\def\rb{\right|}

\def\ov{\overline}
\def \tr {\mathop {\rm tr}}
\def\le{\leq}

\def\lc{\left[}

\def\si3{\sum_{i=1}^3}

\def\rpt{\right.}

\def\O{\Omega}
\def\ov{\overline}
\def\e{\varepsilon}

\def\a{\alpha}
\def\b{\beta}

\def\l{\lambda}
\def\a{\alpha}
\def\b{\beta}

\def\div{{\rm div}}
\def\bfdiv{{\rm \bf div}} 

\def\lp{\left(}
\def\rp{\right)}
\def\la{\left\{}
\def\ra{\right\}}

\def\build#1_#2^#3{\mathrel{\mathop{\kern 0pt #1}\limits_{#2}^{#3}}}

\def\lime{\lim_{\e\to 0}} 

\def\bfa{\pmb{a}} 
\def\bfb{\pmb{b}}

\def\bfe{\pmb{e}}  
\def\bff{\pmb{f}}  
\def\bfg{\pmb{g}}  
\def\bfh{\pmb{h}}

\def\bfn{\pmb{n}}

\def\bfr{\pmb{r}}

\def\bfu{\pmb{u}}  
\def\bfv{\pmb{v}}  
\def\bfw{\pmb{w}}  
\def\bfx{\pmb{x}}   
\def\bfy{\pmb{y}}

\def\bfA{\pmb{A}} 
\def\bfB{\pmb{B}} 
\def\bfC{\pmb{C}} 
\def\bfD{\pmb{D}} 
\def\bfE{\pmb{E}}

\def\bfI{\pmb{I}}

\def\bfR{\pmb{R}} 
 
\def\bfT{\pmb{T}}

\def\A{{\mathcal A}}
\def\B{{\mathcal B}}
\def\C{{\mathcal C}}
\def\D{{\mathcal D}}

\def\F{{\mathcal F}}

\def\H{{\mathcal H}}

\def\L{{\mathcal L}}
\def\M{{\mathcal M}}

\def\P{{\mathcal P}}

\def\NN{\mathbb{N}}

\def\RR{\mathbb{R}}  
\def\SS{\mathbb{S}}

\def\bfgamma{{\pmb{  \gamma}}}

\def\bfsigma{{\pmb{ \sigma}}}

\def\bfpsi{{\pmb{  \psi}}}

\def\bfvarphi{{\pmb{ \varphi}}}

\def\bfGamma{\pmb{  \Gamma}}

\def\bfXi{\pmb{ \Xi}}

\def\bfUpsilon{\pmb{ \Upsilon}}

\def\bfPsi{\pmb{  \Psi}}

  \def\bfnabla{\pmb{  \nabla}}

\def\?{{\bf ???}}

\def\rrrrightharpoonup{\relbar\!\!\relbar\!\!\relbar\!\!\relbar\!\!\relbar\!\!\relbar\!\!\relbar\!\!\relbar\!\!\rightharpoonup}

\def\rrrrightarrow{\relbar\!\!\relbar\!\!\relbar\!\!\relbar\!\!\relbar\!\!\relbar\!\!\relbar\!\!\relbar\!\!\rightarrow}

\def\rpt{\right.}

\def\intb{{\int\!\!\!\!\!\!-}}
\def\disp{\displaystyle}
\def\.{\cdot}

\newcommand\ep{\varepsilon}

\newcommand\pd[2]{\tfrac{\displaystyle\partial #1}{\displaystyle\partial #2}}

\begin{document}

\bibliographystyle{plain}

\title[]{Asymptotic analysis of stratified elastic media in the space of functions with bounded deformation}


\author{Michel Bellieud}
\address{LMGC (Laboratoire de M\'ecanique et de G\'enie Civil de Montpellier)
UMR-CNRS 5508, Universit\'e Montpellier II, Case courier
048, Place Eug\`ene Bataillon, 34095 Montpellier Cedex 5, France}
\email{michel.bellieud@univ-montp2.fr}
\thanks{The first author was supported in part by NSF Grant \#000000.}

\author{Shane Cooper}
\address{ Department of Mathematical Sciences, University of Bath, Claverton Down, Bath, BA2 7AY, UK}
\email[Corresponding author]{s.a.l.cooper@bath.ac.uk}

\subjclass[2000]{}

\date{}



\begin{abstract}
  We consider a heterogeneous  elastic structure which is stratified in some direction.
  We derive the limit problem under the assumption that 
 the Lam\'e coefficients  and their inverses weakly* converge to Radon measures. 
 Our  method  applies also to linear second-order elliptic systems of partial differential equations and in particular, for the case $d=1$, this addresses the previously open problem of determining the asymptotic behaviour in this context for the  general anisotropic heat equation.
\end{abstract}

\maketitle


\section{Introduction}
\label{intro}
 
 The deformation $\bfu$ of an elastic composite subjected to an external force $\bff$ is described by 
 $$
 -  \bfdiv\big(\bfsigma_\e(\bfu) \big) = \bff,
 $$
 where the parameter $\e$ highlights the dependence of the stress $\bfsigma_\e$  on the underlying composite.
 Normally this parameter is small; for example in the case of a composite with a periodic micro-structure, $\e$ is the period. As such, asymptotic analysis can be employed to determine the effective description of the deformation in the limit of the parameter $\e$. There are various approaches to such a study and they all essentially can be  broken down into the following key steps:
 \begin{enumerate}[(a)]
 	\item{To determine the leading-order asymptotic behaviour of $\bfu_\e$ and subsequently characterise the relationship of this limit with that of the stress $\bfsigma_\e(\bfu_\e)$.}
 	\item{To identify the limit or ``effective" problem that the leading-order limit of $\bfu_\e$ satisfies.}
 \end{enumerate}
 Naturally, an additional physically important challenge should be addressed after identifying the limit problem:
 \begin{enumerate}[(c)]
 	\item{To establish error estimates between $\bfu_\e$ and  the solution of the limit problem,}
 \end{enumerate} 
 but this interesting study is outside the scope of the  article, and in full generality the sophistication of the approach needed in establishing c) heavily depends on the regularity of $\bff$.
 
 A classical study which provides the asymptotic description of deformation for uniformly bounded stress has been known for several decades, see for example \cite{OlYoSh}. The relationship between such models and the study of traditional composites has been extensively explored. At the turn of the century, increasingly exotic composite materials which pertain to non-local,  
 {\color{black}   memory \cite{BeGr, FeKh, Sa1}, or higher-order  effects 
   \cite{BeSiam,BeArma, BeBo2},} negative effective {\color{black} density} \cite{ZhPa} \cite{MiWi},  and novel wave effects such as directional propagation and cloaking \cite{AvGrMiRo}, \cite{VS:MoM}, were studied. These works explore particular examples of  `high-contrast' periodic composites where the ratio between the constituent's material parameters is unbounded in $\ep$.  The correct mathematical framework to study such materials requires one to relax the uniformity assumption on $\bfsigma_\e$. 
To highlight the challenges such an undertaking presents, in terms of a) and b), we focus our attention on a particular variational approach  adopted herein: the so-called energy method, first introduced by Luc Tartar \cite{TaP}.

With regards to a) the energy method involves establishing that the sequences of solutions  $(\bfu_\e)$ and stresses $(\bfsigma_\e(\bfu_\e))$  are bounded with respect to some norm (usually the $L^2$-norm in linear problems) and thus converge in a suitable topology.
The correct choice of topology to use is a subtle affair which relies on (usually formal) a priori information about the leading-order asymptotic behaviour of $\bfu_\e$. For example in the classical case where $\bfsigma(\bfu_\e)$ is uniformly bounded, the correct topology is the standard weak topology. Yet, in high-contrast periodic problems where $\bfsigma(\bfu_\ep)$ is unbounded, the correct notion of convergence is two-scale convergence, see  \cite{Al, Ng,Zh1}. Upon identifying the limits of $\bfu_\ep$ and $\bfsigma(\bfu_\e)$ with respect to the appropriate topology, the characterisation of the functional relationship $\lim_{\e} \bfsigma_\e(\bfu_\e) = \mathcal{F}( \lim_{\e}\bfu_\e)$ is a serious one and is related to the issue of determining the limit for products of weak* convergent sequences. While in general this may not be possible, in certain cases this is attainable; in the classical linear elliptic setting this relies on the ``Compensated Compactness Lemma" of Murat-Tatar \cite{Mu,MuTa}; in the high-contrast case, and more recent  partial high-contrasts, this requires establishing generalised Weyl-type decompositions, see \cite{IKVS}. 

With regards to b) once establishing $\bfu = \lim_\ep\bfu_\e$ and identifying $\F$, the energy method aims {\color{black} at    determining} the effective problem for $\bfu$  by passing to the limit in variational problem for $\bfu_\e$ using suitable test fields that asymptotically {\color{black}  behave} like $\bfu_\e$. This requires establishing a ``strong approximability" result about the limit space to which $\bfu$ belongs:  every element of the limit space is the strong limit, with respect to underlying convergence established in a), of elements from the space to which $\bfu_\e$ belongs. In the classical case this result is automatic as the solution $\bfu_\e$ and limit $\bfu$ are elements of the same function space. In the more general high-contrast setting these spaces may vary and such a statement needs to be established.

The general high-contrast theory of elliptic problems is still lacking, particularly in the case of non-periodic composites where some preliminary work has been done in the case of stratified materials, i.e. scalar elliptic partial differential equations with coefficients that depend only on one variable, see 
 \cite{BeStrat}, \cite{BoPi},  \cite{GuHeMo} and references therein. This article further develops the theory in the direction of elliptic systems.
 
 In particular, we study the asymptotic behaviour of stratified isotropic elastic media, occupying a cylindrical domain {\color{black} $\Omega=(0,L)\times \Omega'$} with clamped  boundary (Dirichlet boundary conditions), with  high-contrast  Lam\'{e} coefficients; the stress is of the form
 \begin{equation*}
 \begin{aligned}
 \bfsigma_\ep(\bfu_\e) =
 \lambda_\ep(x_1) {\rm tr}(\bfe(\bfu_\e)) \bfI + 
 2\mu_\e(x_1) \bfe(\bfu_\e),
 \end{aligned}
 \end{equation*}
 for $\lambda_\e$ proportional to  {\color{black} $\mu_\e \in L^\infty(0,L)$} ($\lambda_\ep = l \mu_\e$, $l\ge0$). The shear modulus  $\mu_\e$ and its inverse $\mu^{-1}_\e$ are assumed to be bounded in {\color{black} $L^1(0,L)$}  and converging, as $\ep$ tends to zero, to Radon measures $\nu$ and $m$, respectively, which share no common atoms, see \eqref{mueto}, \eqref{nocommonatom} and Remark \ref{remcommon}. When investigating the asymptotic behaviour of the deformation $\bfu_\ep$ the first crucial thing to note is that the atoms of $\nu$ are points where discontinuities in $\bfu_\e$ will appear. This informs us (at least formally) that the infinitesimal strain $\bfE\bfu$ of the limit function $\bfu = \lim_\e \bfu_\e$ will necessarily be a measure {\color{black} with a  non-vanishing singular part with respect to the Lebesgue measure}, that is $\bfu$ should be a function with bounded deformation ({\color{black} i.e.   belong}  to $BD(\Omega)$, see \eqref{defBD}), and the jump set of $\bfu$ should coincide with the atoms of $\nu$. As such, we expect  $\bfE\bfu$ to be absolutely continuous with respect to  $\nu\otimes\L^2$:
 $$
 \int_B \bfE\bfu = \int_B 
 \frac{\bfE\bfu}{\nu\otimes\L^2}
 \ d \nu\otimes\L^2 \qquad  \text{for Borel sets $B$},
 $$
 where $ \frac{\bfE\bfu}{\nu\otimes\L^2}
 $ denotes the Radon-Nikod\'{y}m derivative of $\bfE\bfu$ with respect to $\nu\otimes\L^2$. On the other hand,  as the medium is stratified 
 we expect deformations  {\color{black} transverse} to the material layers to remain regular, even within  those layers contracting to the atoms of $m$. More precisely,  we determine that under the influence of an external body force $\bff\in L^\infty(\Omega)$,  the deformation $\bfu_\e \in H^1_0(\Omega;\RR^3)$ weakly* converges in $BD(\Omega)$ to some $\bfu$ which belongs to the space
 
 \begin{equation*}
 \hskip-0,2cm  \begin{aligned} 
 &
 BD^{\nu,m}_0(\Omega) := \la  \bfvarphi \in BD(\Omega)\lb
 \begin{aligned}
 &  \bfE\bfvarphi \ll \nu\otimes\L^2, \ \frac{\bfE\bfvarphi}{\nu\otimes\L^2} \in L^2_{\nu\otimes\L^2}(\Omega;\SS^3)
 \\
 & \varphi^\star_\a\in L^2_{m }(0,L; H^1_0(\Omega';\RR^3)) \quad \a\in \{2,3\}  
 \\&\bfvarphi=0 \ \hbox{ on } \ \partial\Omega
 \end{aligned}\rpt\ra,
 \end{aligned} 
 \end{equation*} 
 which we demonstrate is a Hilbert space when equipped with the norm
 $$
 \left|\left|\bfvarphi\right|\right|_{BD^{\nu,m}_0(\Omega)}:= \lp  \int_{\Omega}  |\frac{\bfE\bfvarphi}{\nu\otimes\L^2} |^2 \ d\nu\otimes\L^2\rp^{\frac{1}{2}}+  \lp\int_\Omega |\bfe_{x'}(\bfvarphi^\star)|^2 dm\otimes\L^2\rp^{\frac{1}{2}}.
 $$
 Here $\bfu^\star$ is the precise representative of $\bfu$, see \eqref{defprecise}, which is necessarily introduced as $m$ {\color{black} may have} atoms and $\bfu$ is only defined up to a set of Lesbesgue measure zero.
 
 For the asymptotics of the stress tensor $\bfsigma(\bfu_\ep)$, there are two principle types of behaviour exhibited. Roughly, by the continuity of stresses across interfaces we argue that the components of stress $\bfsigma(\bfu_\ep) \cdot \bfe_1$ must asymptotically behave well,  even at atoms of $\nu$ where $\bfu_\ep$ becomes discontinuous. Similarly, as hinted at above, the components of strain $e_{\e\a\b}(\bfu_\ep)$, for $\a,\b = 2,3$, remain well behaved, even at the atoms of $m$. Indeed, we demonstrate that
{\color{black} $\mu_\e^{-1} \bfsigma_\e(\bfu_\e) \cdot \bfe_1$
  weakly* converges in $\M(\overline{\Omega}; \RR^3)$ (the space of
    $\RR^3$-valued  Radon measures on $\overline{\Omega}$)  to $\Big(l\tr(\frac{\bfE\bfu}{\nu\otimes\L^2}\bfe_1 + 2(\frac{\bfE\bfu}{\nu\otimes\L^2})\bfe_1\Big) \nu\otimes\L^2$}
  while $\mu_\ep e_{\ep\a\b}(\bfu_\e)$  weakly* converges in $\M(\overline{\Omega})$ to $e_{\a\b}(\bfu^\star) m\otimes\L^2 $.

 To determine the problem {\color{black}  which } $\bfu$ solves, we pass to the limit in the variational equality
 \begin{equation*}
 \int_\Omega \bfsigma_\e (\bfu_\ep) \cdot \bfe(\bfvarphi_\e) \ dx = \int_\Omega \bff \cdot \bfvarphi_\e \ dx,
 \end{equation*}
 by constructing a suitable test field $\bfvarphi_\e \in H^1(\Omega;\RR^3)$ which asymptotically behaves like $\bfu_\e$. Here there are two main considerations. First, one must address the apparent issue of  the limit of products of terms $\sigma_{\e1i}(\bfu_\e) $ and $e_{\a\b}(\bfvarphi_\e)$ since a priori results state they converge to elements of different measure spaces, namely
 {\color{black}  $\{ g( \nu\otimes\L^2), \ g\in  L^2_{\nu\otimes\L^2}\}$ and  $\{ h( m\otimes \L^2), h\in L^2_{m\otimes\L^2}\}$ } respectively. To overcome this, we rearrange $\bfsigma(\bfu_\ep) : \bfe(\bfvarphi_\ep)$ in terms of `good' products whose factors mutally converge in the same space, i.e. we determine 
 $$
 \bfsigma_\e (\bfu_\ep) : \bfe(\bfvarphi_\e) = \mu_\e \bfa^{\bot}\bfe(\bfu_\e) : \bfe(\bfvarphi_\e) + \mu_\e \bfa^\Vert\bfe(\bfu_\e) : \bfe(\bfvarphi_\e)
 $$
 where the matrices $\bfa^\bot$, $\bfa^\Vert$ are given by \eqref{defabotaVert}.
 With  $\bfsigma_\e (\bfu_\ep) : \bfe(\bfvarphi_\e)$ written in terms of products that converge in the same sense, we must address the issue of identifying these limits. This requires establishing the ``strong approximability" of the limit space:  each element of $BD^{\nu,m}_0(\Omega)$ has a sequence of fields $\bfvarphi_\ep$, suitably smooth, (belonging to $H^1$) such that $\bfvarphi_\ep$, $\sigma_{\e1i}(\bfvarphi_\e)$ and $e_{\a\b}(\bfvarphi_\e)$ strongly convergence in an appropriate sense (see Definition \ref{defcvtet}) to their analogous limits, this result is Proposition \ref{propvarphie}. The main result of the article is as follows.
 
 \begin{theorem*}
 	The solution sequence $(\bfu_\e)$ is weak* convergent in $BD(\Omega)$ 
 	and its limit $\bfu \in BD^{\nu,m}_0(\Omega)$ is a solution to 
\begin{multline*}
 \int_\O \bfa^\bot\tfrac{\bfE\bfu}{\nu\otimes\L^2} : \tfrac{\bfE\bfvarphi}{\nu\otimes\L^2}   d\nu\otimes\L^2 
 +\int_\O \bfa^\Vert \bfe_{x'}(\bfu^\star):\bfe_{x'}(\bfvarphi^\star)  dm\otimes\L^2 = \int_{\Omega} \bff \cdot \bfv\\ \qquad \forall \bfv \in BD^{\nu,m}_0(\Omega),
\end{multline*}
 Furthermore, such a solution is unique.
 \end{theorem*}
In the particular situation where the Radon measures $\nu$, $m$ have no Cantor part the above weak variational form admits a strong PDE form, see Corollary \ref{cor:strongform}.

 We conclude the introduction by remarking that our method outlined above is not restricted to the case of isotropic linear elasticity. This result is applicable to linear systems of the form 
 $$
 -\div{}{\big( \mu_\e \hbox{\rm\bfC\,}  \nabla \bfu_\e \big)} = \bff,
 $$
 where $\Omega\subset \RR^n$ is a cylindrical domain, $\bfu \in H^1_0(\Omega; \RR^n)$,
 $\hbox{\rm\bfC\,} $ is a symmetric definite positive second order tensor on $\RR^{n+d}$ 
    such that the $n\times n$ matrix $T_{ij} =  \hbox{\rm C\,}_{i1j1}$ is invertible, see Remark \ref{remsystems}. The limit functions here will be of the, more regular, class of functions with bounded variation. For the case $d=1$, this addresses the previously open problem of determining the asymptotic behaviour in this context for the  general anisotropic heat equation. The key to transferability of the method is that under such assumptions on  $\hbox{\rm\bfC\,} $ one can perform a rearrangement of  $\hbox{\rm\bfC\,} \bfnabla \bfu : \bfnabla \bfvarphi$ similar to that for $\bfsigma_\e(\bfu) : \bfe(\bfvarphi)$. It is interesting to note that no such rearrangement can be performed for the case of fully anisotropic elasticity and as such our method does not apply.

\section{Notations}
\label{Secnot}
 
In this article, $\{ \bfe_1, \bfe_2, \bfe_3\}$ stands for the canonical basis of   $\RR^3$.
Points in $\RR^3$ 
and real-valued functions  are represented by symbols beginning with a  lightface lowercase  (example $x, i, \tr \bfA,\ldots$) while vectors and vector-valued functions by symbols
beginning in boldface lowercase (examples:
$\bfu$,
$\bff$,
$\bfdiv  \bfsigma_\e $, \ldots). Matrices and matrix-valued functions are represented  by symbols beginning in boldface uppercase
with the following exceptions:
$\bfnabla\bfu$ (displacement gradient), $\bfe(\bfu)$ (linearized strain tensor).
We denote by $u_i$ or
$(\bfu)_i$  the  components of a vector $\bfu$ and   by $A_{ij}$ or $(\bfA)_{ij}$ those of a matrix $\bfA$ (that is $\bfu=
\sum_{i=1}^3 u_i \bfe_i =\sum_{i=1}^3 (\bfu)_i
\bfe_i$; $\bfA=\sum_{i,j=1}^3 A_{ij} \bfe_i\otimes\bfe_j =\sum_{i,j=1}^3 (\bfA)_{ij} \bfe_i\otimes\bfe_j $, where $\otimes$ stands for  the tensor product). 
For any two vectors $\bfa$, $\bfb$ in $\RR^3$, the symmetric product $\bfa\odot \bfb$ is the symmetric $3\times 3$ matrix defined by $\bfa\odot \bfb:=\frac{1}{2}(\bfa\otimes \bfb+\bfb\otimes\bfa)$. 
 We do not employ  the usual repeated index
convention for summation.    We denote 
by $\bfA\!:\!\bfB=\sum_{i,j=1}^3 A_{ij}B_{ij}$ the   inner
product of two matrices,
  by $\SS^3$ 
   the set of all real symmetric matrices of order
$3$, by $\bfI$   the $3\times 3$ identity matrix.  
 We denote by $\L^n$ the Lebesgue measure in $\RR^n$ and by $\H^k$  the $k$-dimensional Hausdorff measure. 
The letter
$C$ denotes constants  whose precise values  may vary from line to line.
Let $\Omega:=(0,L)\times \O'$ be a  connected cylindrical  open Lipschitz subset    of $\RR^3$. For any  $\bfvarphi\in L^1_{loc}(\Omega;\RR^3)$, 
we denote by $\bfvarphi^\star$ its precise representative, that is 

	\begin{equation}
	\label{defprecise}
	\bfvarphi^\star(x) = \left\{ \hspace{5pt}
	\begin{array}{lcl}
	\displaystyle 
	\lim_{r\rightarrow 0} 
	\intb_{B_r(x)} \bfvarphi(y) \ dy & &  \text{if this limit exists,} \\[5pt] 0 & &  \text{otherwise,}
	\end{array}
	\right.
	\end{equation}

where $B_r(x)$ is  the open ball of radius $r$ centered at $x$, and $\intb_{B_r(x)} \bfvarphi(y) \ dy := \tfrac{1}{\L^3(B_r(x))} \int_{B_r(x)} \bfvarphi(y) \ dy$.
%
%
%
{\color{black}
We also set
\begin{equation}
	\label{defhalfprecise}
	\bfvarphi^\pm(x) = \left\{ \hspace{5pt}
	\begin{array}{lcl}
	\displaystyle 
	\lim_{r\rightarrow 0}
	\intb_{B^\pm_r(x)} \bfvarphi(y) \ dy & &  \text{if this limit exists,} \\[5pt] 0 & &  \text{otherwise,}
	\end{array}
	\right.
	\end{equation}

 where 
  
 \begin{equation}\label{B+-}
\begin{aligned}
&B^+_r(x) : =  B_r(x) \cap \big( (x_1, L) \times \Omega' \big),
\quad B^-_r(x) : =  B_r(x) \cap \big( (0,x_1) \times \Omega' \big).
\end{aligned}
\end{equation}

 The fields  $\bfvarphi^\star$  and  $\bfvarphi^\pm$  
are Borel-measurable and take the same values on the Lebesgue points of $\bfvarphi$, thus 
\begin{equation}\label{fistar=fiae}
\begin{aligned}
&\bfvarphi^\pm=\bfvarphi^\star=\bfvarphi \quad \L^3\hbox{-a.e. in } \ \O.
\end{aligned}
\end{equation}

}
We denote by  $\bfvarphi'$ the element of $L^1_{loc}(\Omega;\RR^3)$ defined by 
\begin{equation}
\begin{aligned} 
 &  \varphi'_1=0, \quad \quad \varphi'_\a=\varphi_\a \quad \forall \a\in \{2,3\},
\end{aligned} 
\label{defpsi'}  
\end{equation} 

\ni and  by $\widetilde\bfvarphi$ the extension of $\bfvarphi$ by $0$ into $\RR^3$.
  If $\varphi_2, \varphi_3$ admit weak  derivatives with respect to $x_2, x_3$, we set 

\begin{equation}
\begin{aligned} &  \bfe_{x'} (\bfvarphi):= \sum_{\a,\b=2}^3 \frac{1}{2}\lp {\partial \varphi_\a \over\partial x_\b }+{\partial \varphi_\b \over\partial x_\a }\rp \bfe_\a\otimes\bfe_\b.
\end{aligned} 
\nonumber
\end{equation}

The symbol      $\bfD\bfvarphi$    represents  the distributional 
 gradient of $\bfvarphi$
 and  $\bfE\bfvarphi:=\tfrac{1}{2}\lp \bfD\bfvarphi+  \bfD\bfvarphi^T\rp$   the  symmetric distributional gradient of $\bfvarphi$. The     space of  functions with bounded deformation on $\Omega$ is defined by
\begin{equation} 
\label{defBD}
   \begin{aligned}
  &  BD(\Omega):=  \la  \bfvarphi\in L^1(\Omega;\RR^3): \  \bfE\bfvarphi  \in \M(\Omega; \SS^3)\ra, 
   \end{aligned}
\end{equation}
\ni where $\M(\Omega; \SS^3)$ stands for  the space of $\SS^3$-valued   Radon measures on $\Omega$ with bounded total variation.
{\color{black} For any $x_1\in (0,L)$,  we set
\begin{equation}
\label{defSigmax1}
\begin{aligned}
  \Sigma_{x_1}:= \{x_1\}\times\O'. 
\end{aligned}
\end{equation} 

  The symbol  $\frac{\l}{\theta}$  represents  the Radon-Nikod\'ym density of  
a  (finite)  vector valued  Radon measure $\l$  on $\Omega$
 with respect to  a positive Radon measure  $\theta$ on $\Omega$.

 \section{Setting  of the problem and    results}\label{secresults} 
 Let $\Omega := (0,L)\times \Omega'$ be a  open bounded cylindrical Lipschitz domain of $\RR^3$. We  are interested in the asymptotic analysis of the solution $\bfu_\e$  to 

\begin{equation}
(\P_\e)\!:\!\la \begin{aligned}
      &\!- \bfdiv(\bfsigma_\e(\bfu_\e)  ) = \bff  \quad  \hbox{ in }    \Omega, \quad 
     \\ & 
   \!  \bfsigma_\ep(\bfu_\e) \!=\!
     \lambda_\ep(x_1) {\rm tr}(\bfe(\bfu_\e)) \bfI \!+ 
     2\mu_\e(x_1) \bfe(\bfu_\e), \quad  
      \bfe(\bfu_\e)\!= \!{1\over 2} (\bfnabla \bfu_\e +\!
 \bfnabla^T
\bfu_\e),
\\& \!\bfu_\e \in H^1_0(\Omega;\RR^3), \quad  \bff \in  L^\infty(\O, \RR^3),
 \end{aligned}
 \rpt
\label{Pe}
\end{equation}
 
when  the Lam\'{e} coefficients  $\l_\e$, $\mu_\e$ only depend on one   variable (say $x_1$) and 
  $\mu_\e, \mu_\e^{-1}$ belong to $L^\infty(0,L)$ and are bounded in $L^1(0,L)$. 
 More precisely, setting

 \begin{equation}
 \label{defmenue}
 \nu_\e:= \mu_\e^{-1} \L^1_{\lfloor [0,L]}; \qquad m_\e:= \mu_\e \L^1_{\lfloor [0,L]},
  \end{equation}
 
we make the following hypotheses:

\begin{equation}
\begin{aligned} 
 &\lambda_\e = l \mu_\e \quad (l \ge 0),  \qquad  \sup_{\e>0} \lp \lb\lb  \mu_\e \rb\rb_{L^1(0,L)}+\lb\lb   \mu_\e^{-1} \rb\rb_{L^1(0,L)}\rp<\infty,
 \\
&
 m_\e  \buildrel \star\over \rightharpoonup  m , 
 \qquad  
 \nu_\e
 \buildrel \star\over \rightharpoonup  \nu  \quad \hbox{ weakly* in } \ \M([0,L]).
  \end{aligned} 
\label{mueto}  
   \end{equation} 
   
We also suppose that  (see Remark \ref{remcommon})
   \begin{equation}
\begin{aligned} 
 &  m(\{t\})\nu(\{t\}) =0 \ \forall t\in [0,L], \quad  m(\{0\})\!=\!m(\{L\})\!=\!\nu(\{0\})\!=\!\nu(\{L\})\!=0.
 \end{aligned} 
\label{nocommonatom}   
   \end{equation} 
   
 Under these assumptions, we prove that  $\bfu_\e$ 
weakly* converges in $BD(\Omega)$ 
 to an element $\bfu$ of the space 

\begin{equation}
 \label{defBDnum0}
\hskip-0,2cm  \begin{aligned} 
 &
BD^{\nu,m}_0(\Omega) := \la  \bfvarphi \in BD(\Omega)\lb
\begin{aligned}
&  \bfE\bfvarphi \ll \nu\otimes\L^2, \  \tfrac{\bfE\bfvarphi}{\nu\otimes\L^2} \in L^2_{\nu\otimes\L^2}(\Omega;\SS^3)
 \\
 & \varphi^\star_\a\in L^2_{m }(0,L; H^1_0(\Omega')) \quad \a\in \{2,3\}  
\\&\bfvarphi=0 \ \hbox{ on } \ \partial\Omega
 \end{aligned}\rpt\ra,
 \end{aligned} 
   \end{equation}  
   
which, endowed with the norm   
   
\begin{equation}
 \label{defnormBDnum0}
\left|\left|\bfvarphi \right|\right|_{BD^{\nu,m}_0(\Omega)}:= \lp  \int_{\Omega}  |\tfrac{\bfE\bfvarphi}{\nu\otimes\L^2}   |^2 \ d\nu\otimes\L^2\rp^{\frac{1}{2}}+  \lp\int_\Omega |\bfe_{x'}( \bfvarphi  ^\star)|^2 dm\otimes\L^2\rp^{\frac{1}{2}},
    \end{equation}  

turns out to be  a Hilbert space. We prove  that   $\bfu$ satisfies  the variational   problem
 
\begin{equation}
\label{Peff}
(\P^{eff}): \la
\begin{aligned}
&
a(\bfu,\bfvarphi ) = \int_\Omega \bff \cdot \bfvarphi   \qquad \forall \bfvarphi  \in BD^{\nu,m}_0(\Omega),
\\& \bfu\in BD^{\nu,m}_0(\Omega),
\end{aligned} \rpt
\end{equation}


 
where 
 $a(\cdot,\cdot)$ is the  non-negative   symmetric bilinear form on  $BD^{\nu,m}_0(\Omega)$ defined  by 
 
 \begin{equation}
\label{defa}
\begin{aligned}
 a(\bfu,\bfvarphi) : =  &\int_\O \bfa^\bot\tfrac{\bfE\bfu}{\nu\otimes\L^2} : \tfrac{\bfE\bfvarphi}{\nu\otimes\L^2}   d\nu\otimes\L^2 
+\int_\O \bfa^\Vert \bfe_{x'}(\bfu^\star):\bfe_{x'}(\bfvarphi^\star)  dm\otimes\L^2, 
 \end{aligned}
\end{equation}

in terms of the fourth order  tensors $\bfa^\bot$ and $\bfa^\Vert$      given by 

\begin{equation}
\label{defabotaVert}
\begin{aligned}
&\bfa^\bot \bfXi :=  \begin{pmatrix}  l\tr\bfXi + 2\Xi_{11} & 2\Xi_{12}&2\Xi_{13}\\ 2\Xi_{12} & \tfrac{l^2}{l+2}\tr\bfXi + \tfrac{2l}{l+2}\Xi_{11}& 0
\\2\Xi_{13} &0& \tfrac{l^2}{l+2}\tr\bfXi + \tfrac{2l}{l+2}\Xi_{11}
\end{pmatrix} ,
\\& \bfa^\Vert \bfGamma :=  \tfrac{2l}{l+2}  \sum_{\b=2}^3\Gamma_{\b\b} \sum_{\a=2}^3 \bfe_\a\otimes\bfe_\a + 2\sum_{\a,\b=2}^3 \Gamma_{\a\b}\bfe_\a\otimes\bfe_\b  .
\end{aligned}
\end{equation}

We prove that $a(.,.)$ is continuous and coercive on $BD^{\nu,m}_0(\O)$, therefore \eqref{Peff} has a unique solution.

\begin{theorem}\label{th}
  The space $BD^{\nu,m}_0(\Omega)$ defined by \eqref{defBDnum0}, endowed with the norm \eqref{defnormBDnum0},   is a Hilbert space. Under the assumptions \eqref{mueto} and \eqref{nocommonatom}, the symmetric bilinear form $a(\cdot,\cdot)$   defined  by \eqref{defa} is coercive and continuous on $BD^{\nu,m}_0(\Omega)$.
The    solution  to (\ref{Pe}) 
weakly* converges in $BD(\Omega)$ 
to the unique solution    to  (\ref{Peff}).
\end{theorem}

 \begin{remark}\label{remen} The problem \eqref{Peff} is equivalent to 
 
\begin{equation}
\label{Peffmin}
\begin{aligned}
& \min_{\bfu\in BD^{\nu,m}_0(\O)} F(\bfu)-\int_\O \bff\cdot\bfu dx
\end{aligned}
\end{equation}

where the functional $F$ is defined on $BD^{\nu,m}_0(\O)$ by 

\begin{equation}
\label{defF}
\begin{aligned}
& F(\bfu):= \int_\O f\lp \tfrac{\bfE\bfu}{\nu\otimes\L^2}\rp d\nu\otimes\L^2 + \int_\O  g(\bfe_{x'}(\bfu^\star) ) d m\otimes\L^2,
\end{aligned}
\end{equation}

in terms of $f,g$ given by 

\begin{equation}
\label{deffg}
\begin{aligned}
& f(\bfXi):=
 \tfrac{1}{2} \bfa^\bot\bfXi:\bfXi, \qquad  g(\bfGamma) ):= \tfrac{1}{2} \bfa^\Vert\bfGamma:\bfGamma.
\end{aligned}
\end{equation}

\end{remark}

\begin{remark}\label{remdec} The symmetric distributional derivative $\bfE\bfvarphi$ of any $\bfvarphi\in BD(\O)$ can be decomposed into an absolutely continuous part $\bfE^a\bfvarphi$ with respect to $\L^3$, 
a jump part $\bfE^j\bfvarphi$ and a Cantor part $\bfE^c\bfvarphi$. The Cantor part $\bfE^c\bfvarphi$ vanishes on any Borel set which is $\sigma$-finite with respect to $\H^2$. The density     $\bfe(\bfvarphi)$ of $\bfE\bfvarphi$  with respect to $\L^3$  is      the approximate symmetric differential of $\bfvarphi$ (see   \cite[Theorem 4.3]{AmCoDa} for more details). 
When  $\bfE\bfvarphi\ll\L^3$,   $\bfe(\bfvarphi)$ is  the weak symmetric gradient of $\bfvarphi$.
 The jump part takes the form $\bfE^j\bfvarphi=\bfE\bfvarphi_{\lfloor  J_\bfvarphi}$, where the ``jump set"  $ J_\bfvarphi$ 
 is   a  countably  $\H^2$-rectifiable subset of $\O$ (i.e.  there exists countably many Lipschitz functions $f_i:\RR^2\to \O$ such that
 $\H^2\lp J_\bfvarphi\setminus \bigcup_{i=0}^{+\infty} f_i(\RR^2)\rp=0$, see \cite[Definition 2.57]{AmFuPa}).
For any countably $\H^2$-rectifiable Borel set $M\subset\O$, the following holds (see  \cite[Chapter II]{Te}, \cite[p.209 (3.2)]{AmCoDa}) 

 \begin{equation}\label{EphilfloorM}
 \bfE\bfvarphi\lfloor_{M} = (\bfvarphi^+_{M}-\bfvarphi^-_{M})\odot\bfn_{M}\H^2_{\lfloor M}, 
      \end{equation}
      
 where $\bfn_{M}(x)$ is a unit normal to $M$ at $x$ and 
$\bfvarphi^\pm_{M}$ is deduced from \eqref{defhalfprecise} by substituting  
 $B_r^\pm(x,\bfn_{M}):=\{y\in B_r(x),  \pm(\bfy-\bfx)\cdot  \bfn_{M}(x) >0\}$  for $B^\pm_r(x)$.

\quad Due to their  absolutely continuity with respect to $\nu\otimes\L^2$, the    symmetric distributional gradient of  the elements of $BD(\O)$  admit  a specific decomposition. 
  The measure  $\nu$ (resp. $m$) can be  split into an absolutely continuous part  $\nu^a$ (resp. $m^a$)  with respect to the Lebesgue measure,
  a   singular part  without atoms  or Cantor part    $\nu^c$ (resp. $m^c$),  and a purely atomic part $\nu^{at}$: 

 \begin{equation}
\begin{aligned}
& \nu\hskip-0,1cm= \hskip-0,1cm \nu^a \hskip-0,1cm +\hskip-0,05cm \nu^c\hskip-0,1cm + \nu^{at},\  \nu^{at}\hskip-0,1cm= \hskip-0,2cm \sum_{t\in \A_\nu}\hskip-0,1cm \nu(\!\{t\}\!)
\delta_{t},  \   \nu^a\hskip-0,1cm=  \hskip-0,1cm\tfrac{\nu}{\L^1} \L^1, \   \A_\nu\hskip-0,1cm:=\hskip-0,1cm\{ t\in\hskip-0,1cm [0,L]; \nu(\{t\})>0\},
\\&  m\hskip-0,1cm =\hskip-0,1cm m^a\hskip-0,1cm +\hskip-0,05cm m^c  \hskip-0,1cm +\hskip-0,2cm \sum_{t\in\A_m}m(\{t\}) \delta_{t},
\  m^a =  \tfrac{m}{\L^1} \L^1,
 \   \A_m\hskip-0,1cm :=\{ t\in [0,L]; m(\{t\})\hskip-0,1cm >0\}.
\end{aligned}
\label{defAnu}
\end{equation} 

We have   $ \nu^a\otimes\L^2\ll \L^3$, it can be shown  that  $\nu^c\otimes\L^2$ vanishes on countably $\H^2$-rectifiable   Borel sets,
the measures $\nu^c\otimes\L^2$ and $\L^3$ are mutually singular and $\nu^{at}\otimes\L^2\ll\H^2_{\lfloor\Sigma_\nu}$, where 

\begin{equation}  \label{defSigma}
\begin{aligned}
\Sigma_\nu:= \bigcup_{t\in\A_\nu}\Sigma_t, \quad \Sigma_m:= \bigcup_{t\in\A_m}\Sigma_t, \quad \Sigma :=\Sigma_\nu\cup \Sigma_m.
\end{aligned}
\end{equation}

Accordingly, the condition $\bfE(\bfvarphi)\ll\nu\otimes \L^2$ satisfied by 
any  element $\bfvarphi$ of $BD^{\nu,m}_0(\O)$ implies 
$\bfE^a\bfvarphi\ll  \nu^a\otimes\L^2$,  $\bfE^c\bfvarphi\ll  \nu^c\otimes\L^2$, $\bfE^j\bfvarphi\ll \H^2_{\lfloor \Sigma_\nu}$. Taking   \eqref{EphilfloorM} into account, we deduce

 \begin{equation}
\begin{aligned}
 \bfE\bfvarphi&
 = \bfe(\bfvarphi) \L^3+ \tfrac{\bfE\bfvarphi}{\nu^c\otimes\L^2} \nu^c\otimes\L^2+ \sum_{t\in \A_\nu} (\bfvarphi^+-\bfvarphi^-)\odot \bfe_1 \H^2_{\lfloor \Sigma_t}.
\end{aligned}
\label{decEfi}\end{equation} 

 The substitution of  \eqref{decEfi} into \eqref{defa}, \eqref{defabotaVert} leads to

 \begin{equation}
\label{defa2}
\begin{aligned}
 a(\bfu,\bfvarphi)\hskip-0,1cm  =\hskip-0,1cm  &  \int_\O\hskip-0,1cm \bfa\bfe(\bfu)\hskip-0,1cm: \hskip-0,1cm\bfe(\bfvarphi )dx+\hskip-0,1cm\hskip-0,1cm
\sum_{t\in \A_\nu}\hskip-0,1cm \nu(\{t\})^{-1}  \hskip-0,1cm\int_{\Sigma_{t}} (\bfu^+ \hskip-0,1cm-\bfu^-)\cdot  
 \bfA  (\bfvarphi^+ \hskip-0,1cm-\bfvarphi^-) d\H^2
 \\&  + \sum_{t\in\A_m}m(\{t\}) \int_{\Sigma_t} \bfa^\Vert\bfe_{x'}(\bfu^\star):\bfe_{x'}(\bfvarphi^\star) d\H^2
 \\& + \int_\O \bfa^\bot \tfrac{\bfE\bfu}{\nu^c\otimes\L^2} : \tfrac{\bfE\bfvarphi}{\nu^c\otimes\L^2}   d\nu^c\otimes\L^2 
+\int_\O  \bfa^\Vert\bfe_{x'}(\bfu^\star):\bfe_{x'}(\bfvarphi^\star)  dm^c \otimes\L^2
 , 
 \end{aligned}
\end{equation}

where the fourth  order tensor  $\bfa$ and the matrix $\bfA$ are  given by

\begin{equation}
\label{deftensoraA}
\begin{aligned}
  \bfa   = \lp\tfrac{\nu}{\L^1}\rp^{-1} \bfa^\bot+  \tfrac{m}{\L^1}\bfa^\Vert,
\qquad \bfA := \begin{pmatrix} l+2&0&0\\0&1&0\\0&0&1\end{pmatrix}.
\end{aligned}
\end{equation}

%

Similarly,   substituting \eqref{decEfi} into \eqref{defF} yields

\begin{equation}
\nonumber
\begin{aligned}
  &F(\bfu):=  \int_\O f\lp\bfe(\bfu)\rp \lp\tfrac{\nu}{\L^1}\rp^{-1} + g(\bfe_{x'}(\bfu ) ) \tfrac{m}{\L^1} dx 
  \\&+\hskip-0,1cm\sum_{t\in \A_\nu}  \hskip-0,1cm\tfrac{\nu(\{t\})^{-1}}{2}\hskip-0,1cm \int_{\Sigma_t}\hskip-0,1cm (\bfu^+-\bfu^-)\hskip-0,1cm\cdot\hskip-0,1cm \bfA(\bfu^+-\bfu^-) d\H^2
 \hskip-0,1cm +\hskip-0,15cm\sum_{t\in \A_m} \hskip-0,1cm \tfrac{m(\{t\})}{2} \hskip-0,1cm\int_{\Sigma_t} \hskip-0,1cmg(\bfe_{x'}(\bfu^\star ) \!)  d\H^2
  \\&+\int_\O f\lp \tfrac{\bfE^c\bfu}{\nu\otimes\L^2}\rp d\nu^c\otimes\L^2 + \int_\O  g(\bfe_{x'}(\bfu^\star) ) d m^c\otimes\L^2.
\end{aligned}
\end{equation}

We   can  write the PDE system associated with \eqref{Peff}, \eqref{defa2} provided  the Cantor parts  $\nu^c$ and $m^c$  vanish and the sets ot atoms $\A_\nu$ and $\A_m$ are  finite:

\begin{corollary} 
\label{cor:strongform}
If $\nu^c=m^c=0$ and    $\A_\nu$, $\A_m$ are  finite, 
the problem \eqref{Peff} is equivalent to 
 
  \begin{equation}
  \la \begin{aligned}
  &- \bfdiv \bfa \bfe (\bfu) = \bff \  \hbox{ in }   \   \Omega\setminus\Sigma,  
  \quad \bfu\in BD^{\nu,m}_0(\O),
         \\[2pt]&    \nu(\{t\})^{ \hskip-0,05cm-1} \hskip-0,1cm\bfA (\bfu^+\hskip-0,1cm-\bfu^-)\hskip-0,05cm=\hskip-0,05cm
    \bfa \bfe (\bfu^-)\bfe_1\hskip-0,05cm=\hskip-0,05cm  \bfa \bfe (\bfu^+)\bfe_1
  \  & &\hbox{ on } \Sigma_{t}, \  \forall t\in \A_\nu,
    \\[2pt]& \bfa \bfe (\bfu^-)\bfe_1\hskip-0,05cm-\hskip-0,05cm  \bfa \bfe (\bfu^+)\bfe_1-m(\{t\})\bfdiv_{x'} \bfa^\Vert \bfe_{x'}(\bfu^\star) =0 & & \hbox{ on   }  \Sigma_t,\   \forall t\in \A_m,
   \end{aligned}\rpt
\label{PeffnoCantor}
\end{equation}
\end{corollary}

where  
$\Sigma$, $\bfa$, $\bfA$ are  given by \eqref{defSigma}, \eqref{deftensoraA}.

\begin{proof} 
Choosing $\bfvarphi\in \D(\O\setminus\Sigma)$ in \eqref{Peff}, taking \eqref{defa2} into account, we get $ \int_{\O\setminus\Sigma}  \bfsigma(\bfu):\bfe (\bfvarphi) dx =\int_\O \bff\cdot\bfvarphi$ and infer, 
by the arbitrary choice of $\bfvarphi$,  that  $- \bfdiv  \bfa \bfe (\bfu) = \bff $ in $ \Omega\setminus\Sigma$. 
Choosing $\bfvarphi\in BD^{\nu,m}_0(\O)$ such that   $\bfvarphi\in C^\infty( U)$  for every connected component $U$ of $\O\setminus \Sigma_\nu$, and integrating 
$ \bfa \bfe (\bfu):\bfe (\bfvarphi)$ by parts  over  each connected component of $\O\setminus\Sigma$, taking the first line of \eqref{PeffnoCantor} into account,   we deduce 

 \begin{equation}
\nonumber
\begin{aligned}
&\sum_{t\in\A_m} \int_{\Sigma_t} \lp \bfa \bfe (\bfu^-)\bfe_1\hskip-0,05cm-\hskip-0,05cm  \bfa \bfe (\bfu^+)\bfe_1\rp\cdot \bfvarphi+
 m(\{t\})\bfa^\Vert\bfe_{x'}(\bfu^\star):\bfe_{x'}(\bfvarphi^\star) d\H^2
\\& +
 \hskip-0,1cm\sum_{t\in \A_\nu}  \hskip-0,1cm   \int_{\Sigma_{t}}   \hskip-0,1cm
  \bfa \bfe (\bfu^-\!)\bfe_1 \hskip-0,1cm\cdot \!\bfvarphi^- \hskip-0,1cm -\!   \bfa \bfe (\bfu^+\!)\bfe_1 \hskip-0,1cm \cdot\! \bfvarphi^+  \hskip-0,1cm
+  \!(\bfu^+ \hskip-0,1cm -\!\bfu^-\!) \hskip-0,1cm\cdot  \hskip-0,1cm 
\nu(\{t\})^{ \hskip-0,07cm-1}  \hskip-0,1cm\bfA  (\bfvarphi^+ \hskip-0,1cm -\!\bfvarphi^-\!) d\H^2  \hskip-0,1cm=\!0,
\end{aligned}
\end{equation}

and obtain the transmission conditions stated in  the second and third lines of \eqref{PeffnoCantor}.
 Conversely, any solution to 
      \eqref{PeffnoCantor} satisfies \eqref{Peff}.
\end{proof} 

\end{remark}

\begin{remark}\label{remcommon}
When  $\nu$ and $m$ do  not satisfy  \eqref{nocommonatom}, 
the effective problem not only  depends on the couple $(\nu,m)$, but also on the choice of the sequence $(\mu_\e)$ satisfying \eqref{mueto}.
By way of illustration, let us  fix  two sequences of positive reals $(r_\e^{(1)})$, $(r_\e^{(2)})$  converging to $0$,  set  $r_\e:= \max\la r_\e^{(1)}, r_\e^{(2)}\ra$, and consider the sequence $(\mu_\e)$ defined by 
$$
\mu_\e :=  \mathds{1}_{
(0,L)\setminus \lp\tfrac{L}{2}-\tfrac{r_\e}{2}, \tfrac{L}{2}+\tfrac{r_\e}{2} \rp}
 +  r_\e^{(1)}  \mathds{1}_{\big(\tfrac{L}{2}-\tfrac{r_\e^{(1)}}{2}, \tfrac{L}{2}+\tfrac{r_\e^{(1)}}{2}\big)}
+  \tfrac{1}{r_\e^{(2)}} \mathds{1}_{\big(\tfrac{L}{2}-\tfrac{r_\e^{(2)} }{2}, \tfrac{L}{2}+\tfrac{r_\e^{(2)}}{2} \big)}.
$$

The convergences   \eqref{mueto}   are satisfied  with 
$ \nu =m= \delta_{\tfrac{L}{2}}+ \L^1$.  By adapting to  the framework  of elasticity  the argument developed in \cite[Chapter 4]{BeThesis} in the context of  the heat equation,
one can prove the following results:

\begin{itemize}

\item If  $r_\e^{(2)}\ll r_\e^{(1)}$, the effective problem takes the form 

$$
\inf_{} \la 
F(\bfvarphi)
 -\int_\O \bff\cdot\bfvarphi dx, \quad \bfvarphi \in H^1(\O\setminus\Sigma_{L/2}), \ \bfvarphi=0 \hbox{ on } \ \partial\O\ra,
 $$
 
\ni  where, setting $\bfsigma(\bfvarphi):= l\tr(\bfe(\bfvarphi))\bfI+2\bfe(\bfvarphi)$, $\bfsigma_{x'}(\bfvarphi'):= l\tr(\bfe_{x'}(\bfvarphi'))\bfI+2\bfe_{x'}(\bfvarphi')$,
 $F$ is the non-local functional defined by (see \eqref{deftensoraA})

 \begin{equation}
\nonumber
\begin{aligned}
& F(\bfvarphi)=    \inf_{\bfv\in H^1_0(\Sigma_{L/2};\RR^3)} \Phi(\bfvarphi,\bfv),
\\& \Phi(\bfvarphi,\bfv),:=\tfrac{1}{2} \int_{\O\setminus\Sigma_{L/2}} \bfsigma(\bfvarphi):\bfe (\bfvarphi) dx  +\tfrac{1}{2}  \int_{\Sigma_{L/2}} \bfsigma_{x'}(\bfv'):\bfe_{x'}(\bfv') d\H^2
\\&\hskip 3cm +
\tfrac{1}{4}  \int_{\Sigma_{L/2}} (\bfv -\bfvarphi^-)\cdot  
\bfA  (\bfv -\bfvarphi^-) +  (\bfv -\bfvarphi^+)\cdot  
\bfA  (\bfv -\bfvarphi^+) 
 d\H^2.
\end{aligned}
\end{equation}

\item If   $r_\e^{(1)}\ll r_\e^{(2)}$, the effective problem is given by

$$
\inf_{} \la 
F(\bfvarphi)
 -\int_\O \bff\cdot\bfvarphi dx, \quad \bfvarphi \in H^1(\O\setminus\Sigma_{L/2}), \ \bfvarphi=0 \hbox{ on } \ \partial\O\ra,
 $$
 
 where

 \begin{equation}
\nonumber
\begin{aligned}
 F(\bfvarphi)=&\tfrac{1}{2} \int_{\O\setminus\Sigma_{L/2}} \bfsigma(\bfvarphi):\bfe (\bfvarphi) dx  + \tfrac{1}{4} \int_{\Sigma_{L/2}} \bfsigma_{x'}(\bfvarphi^-):\bfe_{x'}(\bfvarphi^-) d\H^2
\\&+ \tfrac{1}{4} \int_{\Sigma_{L/2}} \bfsigma_{x'}(\bfvarphi^+):\bfe_{x'}(\bfvarphi^+) d\H^2
  + \tfrac{1}{2}\int_{\Sigma_{L/2}} (\bfvarphi^+ -\bfvarphi^-)\cdot  
\bfA  (\bfvarphi^+ -\bfvarphi^-)   d\H^2.
\end{aligned}
\end{equation}

\end{itemize}

Indeed,  when  \eqref{nocommonatom} is not satisfied, there exists   infinitely many  different limit problems associated to  some sequence $(\mu_\e)$ satisfying \eqref{mueto}.

\end{remark}


 
 
 \ni  
\begin{remark}\label{remsystems}
 Our method  applies to the study 
of   
second-order elliptic systems of partial differential equations 
  of the type 
 \begin{equation}
(\P_\e):  \begin{aligned}
   -  & \bfdiv(\mu_\e\hbox{\rm\bfC\,} \bfnabla \bfu_\e) = \bff  \quad  \hbox{ in }    \Omega,  
        \quad  \bfu_\e \in H^1_0(\Omega ; \RR^n), \   \bff \in  L^\infty(\O;\RR^n)  ,
 \end{aligned}
\label {PeSyst} 
\end{equation}

 where $\Omega:= (0,L)\times \Omega'$ is a cylindrical domain in $\RR^d$  and  $\hbox{\rm\bfC} $
is
 a second order tensor on  $\RR^{n+d}$  satisfying the following  assumptions of symmetry
  and ellipticity:

 \begin{equation}
\label{hypA}
\begin{aligned}
&  \hbox{\rm C}_{ijpq} = \hbox{\rm C}_{pqij} \quad \forall ((i,j), (p,q))\in (\RR^n\times\RR^d)^2,
\\&\hbox{\rm\bfC\,} \bfXi : \bfXi  \ge c \vert \bfXi \vert^2  \quad \forall \bfXi \in \RR^{n\times d} \qquad \hbox{ for some } \ c>0.
 \end{aligned}
\end{equation}
 
We  suppose that

 \begin{equation}
\label{defT}
\begin{aligned}
& \bfT:=  \sum_{i,p=1}^n \hbox{\rm C}_{i1p1}\bfe_i\otimes\bfe_p \quad \hbox{ is invertible}.
 \end{aligned}
\end{equation}

 We denote by $BV(\Omega;\RR^n)$  the space of $\RR^n$-valued functions  on $\Omega$ with bounded variation, that is
  
  \begin{equation} 
\label{defBV}
   \begin{aligned}
  & BV(\Omega;\RR^n):=  \la  \bfvarphi\in L^1(\Omega;\RR^n): \  \bfD\bfvarphi  \in \M(\Omega; \RR^{n+d})\ra.
   \end{aligned}
\end{equation}

   Under these  assumptions,    the solution to \eqref{PeSyst} weakly* converges  in $BV(\Omega;\RR^n)$ 
  to the unique solution to the problem 

 \begin{equation}
 \label{PeffSyst}
 (\P^{eff}): \la
 \begin{aligned}
 &
 a(\bfu,\bfvarphi ) = \int_\Omega \bff\cdot \bfvarphi  , \qquad \forall \ \bfvarphi  \in BV^{\nu,m}_0(\Omega),
 \\& \bfu\in BV^{\nu,m}_0(\Omega),
 \end{aligned} \rpt
 \end{equation}

where $ BV^{\nu,m}_0(\Omega)$ is the Hilbert space defined by 

\begin{equation}
 \label{defBVnum0}
 \hskip-0,2cm  \begin{aligned} 
 &
 BV^{\nu,m}_0(\Omega) := \la  \bfvarphi \in BV(\Omega;\RR^n)\lb
 \begin{aligned}
 & \bfD\bfvarphi \ll \nu\otimes\L^{d-1}, \quad \bfvarphi=0 \ \hbox{ on } \ \partial\Omega
 \\& 
 \tfrac{\bfD\bfvarphi\ \ }{\nu\otimes\L^{d-1}}\in L^2_{\nu\otimes\L^{d-1}}(\Omega;\RR^n)
  \\ & \bfvarphi^\star \in L^2_{m }(0,L; H^1_0(\Omega';\RR^n))
 \end{aligned}\rpt\ra,
 \\& \left|\!\left| \bfvarphi \right|\!\right|_{BV^{\nu,m}_0(\Omega)}\!
 := \!\!\lp  \int_{\Omega} \!\! | \tfrac{\bfD\bfvarphi\ \ }{\nu\otimes\L^{d-1}}  |^2  \!d\nu\!\otimes\!\L^{d-1}\rp^{\!\frac{1}{2}}\!\!\!+ \! \lp\int_\Omega |\bfnabla_{x'}(\bfvarphi ^\star)|^2 \! dm\otimes\L^{d-1}\rp^{\frac{1}{2}}\!\!\!\!\!,
 \end{aligned} 
 \end{equation} 

and, setting
 
 \begin{equation}
 \label{defnablax'Xi}
 \begin{aligned}
 \bfnabla_{x'} \bfvarphi:=  \sum_{i=1}^n\sum_{\a=2}^d \pd{\varphi_i}{x_\a} \bfe_i \otimes \bfe_\a,
 \end{aligned}
 \end{equation} 
 
  $a$ is the continuous   coercive symmetric bilinear form on $BV^{\nu,m}_0(\O)$ given by 
  
   \begin{equation}
\label{defaSyst}
\begin{aligned}
\hskip-0,1cm  a(\bfu,\bfvarphi) \hskip-0,1cm  : = \hskip-0,1cm   &\int_\O\hskip-0,1cm   \bfa^\bot\tfrac{\bfD\bfu}{\nu\otimes\L^2}\hskip-0,1cm   :\hskip-0,1cm   \tfrac{\bfD\bfvarphi}{\nu\otimes\L^2}   d\nu\otimes\L^{d-1}
+\int_\O \bfa^\Vert \bfnabla_{x'}(\bfu^\star)\hskip-0,1cm  :\hskip-0,1cm  \bfnabla_{x'}(\bfvarphi^\star)  dm\otimes\L^{d-1}\hskip-0,1cm  , 
 \end{aligned}
\end{equation}

where 

\begin{equation}
\label{defabotaVertsyst}
\begin{aligned}
& a^\bot_{ijkl}:=  \sum_{p,r=1}^n C_{ijp1} (\bfT^{-1})_{pr}C_{r1kl},
\\& a^\Vert_{ijkl} :=    \sum_{p,r=1}^n ( C_{ijp1}(\bfT^{-1})_{pr} C_{r1kl} + C_{ijkl}) (1-\delta_{j1})(1-\delta_{l1})  .
\end{aligned}
\end{equation}

\begin{proposition}\label{thsystem}
  The normed  space  $BV^{\nu,m}_0(\Omega)$ defined by \eqref{defBVnum0}   is a Hilbert space. Under the assumptions \eqref{mueto}, \eqref{hypA},  \eqref{defT}, the symmetric bilinear form $a(\cdot,\cdot)$   defined  by \eqref{defaSyst} is coercive and continuous on $BV^{\nu,m}_0(\Omega)$, and 
the sequence $(\bfu_\e)$ of the solutions to (\ref{PeSyst}) 
weakly* converges in $BV(\Omega;\RR^n)$ 
to the unique solution $\bfu$  to  (\ref{PeffSyst}).
\end{proposition}

 The proof of Proposition \ref{thsystem} is sketched in Section \ref{secsketchproofthsystem}.
 
\end{remark}

  \begin{remark}\label{remheat}  
  The particular case of the heat equation in a three-dimensional domain corresponds to the choice 
  $(n,d)=(1,3)$ in  \eqref{PeSyst}. Setting  $A_{jq}:=\hbox{\rm C}_{1j1q}$,   we deduce from Proposition \ref{thsystem} that  under  the assumption \eqref{mueto}, 
if  $\bfA$ is definite positive and $A_{11}\not=0$ (see   \eqref{defT}),  the solution $u_\e$ to 


\begin{equation}
(\P_\e):  \begin{aligned}
   -  & \bfdiv(\mu_\e \bfA  \bfnabla u_\e) = f  \   \hbox{ in  }    \Omega,  
        \quad  u_\e \in H^1_0(\Omega ), \quad   f \in  L^\infty(\O),
 \end{aligned}
\label{Peheat}
\end{equation}

weakly converges in $BV(\O;\RR)$  to the unique solution to  

\begin{equation}
\nonumber
\begin{aligned}
& \min_{ u\in BD^{\nu,m}_0(\O)} F( u)-\int_\O  f u dx,
\end{aligned}
\end{equation}
 
 where the functional $F$ is defined on $BV^{\nu,m}_0(\O)$ by 

\begin{equation}
\nonumber
\begin{aligned}
& F(u):= \hskip-0,1cm  \tfrac{1}{2} \int_\O \hskip-0,1cm  \bfA^\bot  \tfrac{\bfD u}{\nu\otimes\L^2} \hskip-0,1cm  \cdot  \hskip-0,1cm   \tfrac{\bfD u}{\nu\otimes\L^2}   d\nu\otimes\L^2 +\tfrac{1}{2} \int_\O \bfA^\Vert \bfnabla_{x'}( u^\star) \cdot  \bfnabla_{x'}( u^\star)  d m\otimes\L^2,
\end{aligned}
\end{equation}

in terms of $\bfA^\bot, \bfA^\Vert$ given by 

\begin{equation}
\nonumber
\begin{aligned}
& A^\bot_{ij}:= \tfrac{A_{i1}A_{1j}}{A_{11}},
\qquad  A^\Vert_{ij} := (\tfrac{A_{i1}A_{1j}}{A_{11}}+ A_{ij}) (1-\delta_{i1})(1-\delta_{j1})  .
\end{aligned}
\end{equation}

We we recover some of the results obtained by G. Bouchitt\'e and C. Picard in \cite{BoPi}. However,  the method employed in \cite{BoPi} only applies, in the linear case, to a diagonal conductivity matrices.

%
%
%
%
%
%
%
%
%
%
%
%
%
%
%
%

\end{remark}


%
 
 \section{Technical preliminaries and a priori estimates.}\label{secapriori}

 This section is dedicated, essentially,  to the  analysis of the leading-order asymptotic behaviour of the  solution $(\bfu_\ep)$ to \eqref{Pe} and its stress $\bfsigma(\bfu_\ep)$ in the limit $\ep \rightarrow 0$.  The following notion  of convergence will take a crucial part in this study.
\begin{definition}\label{defcvtet}
Let $\theta_\e, \theta$  be   positive
Radon measures on a compact  set $K\subset\RR^N$  and    let $f_\e, f$  be    
Borel  functions on $K$. We  say that  
  $(f_\e)$ weakly converges to $f$ with respect to the pair $(\theta_\e, \theta)$ if 

\begin{equation}
\label{fetetoft}
\begin{aligned}
&
\sup_\e \int_K |f_\e|^2 d\theta_\e<\infty, \quad f\in L^2_\theta(K)
\\&  \theta_\e \buildrel\star \over \rightharpoonup  \theta \ \hbox{ and } f_\e\theta_\e \buildrel\star \over \rightharpoonup f\theta \quad \hbox{ weakly* in } \ \M(K),
\\&  ( \hbox{notation:} \ f_\e \buildrel {\theta_\e, \theta}\over \rightharpoonup f).
\end{aligned}
\end{equation}

We  say that  
  $(f_\e)$ strongly  converges to $f$ with respect to the pair $(\theta_\e, \theta)$ if 

\begin{equation}
\label{fetetoftstrong}
\begin{aligned}
&
f_\e \buildrel {\theta_\e, \theta}\over  {\relbar\!\!\relbar\!\!\rightharpoonup} f \  \hbox{ and } \  
\limsup_{\e\to0} \int_K |f_\e|^2 d\theta_\e\le\int_K |f|^2 d\theta  \quad ( \hbox{notation:} \ f_\e \buildrel {\theta_\e, \theta}\over  {\relbar\!\!\relbar\!\!\rightarrow}  f).
\end{aligned}
\end{equation}

\end{definition}
 
  
 We now present the main statement of the section. For notational simplicity, the measures  $(\nu_\e\otimes \L^2)_{\lfloor \ov\Omega}$ and $(m_\e\otimes \L^2)_{\lfloor \ov\Omega}$ are denoted   by 
 $\nu_\e\otimes \L^2$ and $m_\e\otimes \L^2$.
 For  simplicity, the extension  by $0$ to $\ov\O$ of the measure $\bfE\bfu$ is still denoted by  $\bfE\bfu$.
 

\begin{proposition}\label{propaprioriu}
	
	\ni   Let   $(\bfu_\e)$  be the sequence of   solutions to \eqref{Pe}.  Then $\bfu_\e$ is bounded in $BD(\Omega)$ and
	
	\begin{equation} 
	\label{supFeuefini}
	\begin{aligned}
	&\sup_{\e>0}  \int_\Omega |\bfu_\e'|^2 dm_\e\otimes \L^2 + \int_\Omega |\bfu_\e| dx+  \int_\Omega \mu_\ep \left\vert \bfe(\bfu_\e) \right\vert^2 dx
	<\infty.
	\end{aligned}
	\end{equation}
	
	Up to   a subsequence, the following hold:

	\begin{equation} 
	\label{cvu}
	\begin{aligned}
	&\bfu_\e \ \buildrel\star\over \rightharpoonup  \bfu \quad 
	\hbox{weakly*  in }  BD(\Omega),\quad \bfE \bfu_\e \ \buildrel\star\over \rightharpoonup  \bfE\bfu \quad 
	\hbox{weakly*  in }  \M(\ov\Omega;\SS^3),
	\\&\mu_\e  \bfe(\bfu_\e) \buildrel{\nu_\e\otimes\L^2, \nu \otimes\L^2}\over \rrrrightharpoonup 
\tfrac{\bfE\bfu}{\nu\otimes\L^2},
	\qquad \bfsigma_\e(\bfu_\e) 
\buildrel {\nu_\e\otimes\L^2, \nu\otimes\L^2}\over \rrrrightharpoonup  
\bfsigma^\nu(\bfu),
\\&   \bfe_{x'}(\bfu_\e') \buildrel{m_\e\otimes\L^2, m \otimes\L^2}\over \rrrrightharpoonup  \bfe_{x'}((\bfu^\star)'),
\quad 
 	\bfu\in BD^{\nu,m}_0(\Omega).
	\end{aligned}
	\end{equation}

\end{proposition}

 \ni 
 Before presenting the proof of Proposition \ref{propaprioriu}, we introduce and prove some auxiliary results.   
 The next lemma states some fundamental properties of convergence  with respect to the pair $(\theta_\e, \theta)$,
  established in  \cite[Theorem 4.4.2]{Hu} in a more general context  (see also \cite{BeBo},  \cite{Bo},   \cite{BoBuFr}, \cite[Section 2.1]{Zh}).


\begin{lemma} \label{lemfeps} Let $(\theta_\e)$  be a  sequence of positive
Radon measures on a compact  set $K\subset\RR^N$  weakly* 
converging  in $\M(K)$  to some positive Radon measure $\theta$. Then, 

(i)  any   sequence $(f_\e)$  of 
Borel  functions on $K$  such
that 

\begin{equation}
\label{supfethetaefini}
\sup_\e \int  |f_\e|^2 d\theta_\e<\infty,
\end{equation}

has a weakly converging subsequence with respect to the pair $(\theta_\e, \theta)$.

 (ii)
If $f_\e \buildrel {\theta_\e, \theta}\over {\relbar\!\!\relbar\!\!\rightharpoonup} f$ (resp. $f_\e  \buildrel {\theta_\e, \theta}\over {\relbar\!\!\relbar\!\!\rightarrow} f$), 
then

 \begin{equation} 
  \label{linfeps}
   \begin{aligned}
 \liminf_{\e\to 0} \int   f_\e^2d\theta_\e \ge \int  f^2d\theta\quad \lp \hbox{resp. } \  \lim_{\e\to 0} \int   f_\e^2d\theta_\e = \int  f^2d\theta\rp.
   \end{aligned}
\end{equation}

(iii) If $f_\e \buildrel {\theta_\e, \theta}\over {\relbar\!\!\relbar\!\!\rightharpoonup} f$ and  $g_\e \buildrel {\theta_\e, \theta}\over  {\relbar\!\!\relbar\!\!\rightarrow} g$, then
 
 \begin{equation} 
 \nonumber
   \begin{aligned}
   \disp\lim_{\e\to 0} \int  f_\e g_\e d\theta_\e = \int  fg d\theta.
   \end{aligned}
\end{equation} 

\end{lemma}

 As a first application of Lemma \ref{lemfeps},  we obtain  some relations between the measures $\nu$, $m$, and $\L^1_{\lfloor[0,L]}$:

\begin{lemma} \label{lemL1nu}  Under \eqref{mueto}, the following holds 
	
	\begin{equation} 
	  \label{DL1num}
	\begin{aligned}
	&\L^1_{\lfloor [0,L]}\ll\nu;  \quad   
\tfrac{\L^1}{\nu}	  \in L^2_\nu([0,L]);\quad \L^1_{\lfloor [0,L]}\ll  m;\quad  \tfrac{\L^1}{m}   \in L^2_m([0,L]); 
	\\& 
	\int_{[0,L]} |\tfrac{\L^1}{\nu} |^2 d\nu \le  m([0,L]); \quad \int_{[0,L]} |\tfrac{\L^1}{m} |^2 dm \le  \nu ([0,L]).
	\end{aligned}
	\end{equation}

\end{lemma}

\begin{proof}  Noticing that, by  \eqref{defmenue} and \eqref{mueto},   $\sup_\e\int_{[0,L]} |\mu_\e|^2 d\nu_\e= \sup_\e m_\e([0,L])<\infty$ (resp. 
$\sup_\e\int_{[0,L]} |\mu_\e|^{-2} dm_\e= \sup_\e\nu_\e([0,L])<\infty$),
%
%
we deduce from  Lemma \ref{lemfeps} that  the sequence $(\mu_\e)$ (resp. $(\mu_\e^{-1})$) has a converging subsequence with respect to the pair $(\nu_\e,\nu)$ (resp. $(m_\e, m)$), and 	
	\begin{equation} 
	\label{D1}	
	\begin{aligned}
	&\mu_\e \nu_\e   \buildrel\star \over \rightharpoonup g \nu , \quad \mu_\e^{-1} m_\e  \buildrel\star \over \rightharpoonup h m, \quad 
	g\in L^2_\nu, \ h\in L^2_m,
	\\& \liminf_{\e\to 0}   \int |\mu_\e|^2 d\nu_\e\ge \int |g|^2 d\nu;  \quad \liminf_{\e\to0} \int|\mu_\e|^{-2} dm_\e \ge \int |h|^2 dm.
	\end{aligned}
	\end{equation}  

On the other hand, by  \eqref{defmenue} and \eqref{mueto}, the following holds: 

\begin{equation} 
\label{D2}
	\begin{aligned}
	&\mu_\e \nu_\e=\mu_\e^{-1}m_\e= \L^1_{\lfloor[0,L]},\quad |\mu_\e|^2\nu_\e=m_\e,\quad |\mu_\e|^{-2} m_\e=  \nu_\e,\quad
	\\&\limsup_{\e\to0} m_\e([0,L])\le m([0,L]), \quad \limsup_{\e\to0} \nu_\e([0,L])\le \nu([0,L]).
	\end{aligned}
	\end{equation}  
	
Assertion  \eqref{DL1num} follows from \eqref{D1} and \eqref{D2}.
\end{proof}

The following   statement, proved in 
\cite[Lemma 3.1]{BoPi} (see also \cite[Lemma 6.2]{GuHeMo} for a more general version), 
provides a sufficient condition for the product of a sequence strongly converging in $L^1(0,L)$
by a sequence of functions weakly*  converging to a measure in $\M([0,L])$,
to   weakly*  converge in  $\M([0,L])$ to 
the product of the individual limits. 




\begin{lemma} \label{lemBoPi}
	Let $(b_\e)$  be a bounded sequence in $L^1(0,L)$ that  weakly*  converges in   $\M([0,L])$ to some Radon measure $\theta$ satisfying 
	\begin{equation} 
	\label{theta0L}
	\begin{aligned}
	\theta(\{0\})=\theta(\{L\})=0.
	\end{aligned}
	\end{equation} 
	
	Let $(w_\e)$ be a bounded sequence in $W^{1,1}(0,L)$  weakly* converging   in $BV(0,L)$ to  some $w$.
	Assume that
	
	
	\begin{equation} 
	\label{tzno}
	\begin{aligned}
	\theta(\{t\})Dw(\{t\})=0 \quad \forall t\in (0,L).
	\end{aligned}
	\end{equation}
	
	\ni Then 
 
	\begin{equation} 
	\nonumber 
	\begin{aligned}
	\lime \int_0^L \psi  b_\e w_\e dx = \int_{(0,L)} \psi w^{(r)} d\theta = \int_{(0,L)} \psi w^{(l)} d\theta  \qquad \forall \psi\in C([0,L]), 
	\end{aligned}
	\end{equation}

	\ni where $ w^{(r)}$  (resp. $w^{(l)}$) denotes  the right-continuous (resp.    left-continuous)  representative of $w$.
	
\end{lemma}

For any $\bfvarphi\in BD(\O)$, we denote by $ \bfgamma^\pm_{\Sigma_{x_1}}(\bfvarphi) $ the trace of $\bfvarphi$ on both sides of 
$\Sigma_{x_1}$ (see \eqref{defSigmax1}). In  the next lemma, we show that the mapping $x\to  \bfgamma^\pm_{\Sigma_{x_1}}(\bfvarphi)$ can be identified with $\bfvarphi^\pm$ defined by \eqref{defhalfprecise}.

\begin{lemma}\label{lemhalf}
Let  $\bfvarphi\in BD(\O)$ and let $\bfvarphi^\star$, $\bfvarphi^\pm$ be defined by 
\eqref{defprecise}, \eqref{defhalfprecise}. Then

 \begin{equation}\label{traceball}
 \begin{aligned}
    \bfgamma^\pm_{\Sigma_{x_1}}(\bfvarphi) (x)=\bfvarphi^\pm(x)=\lim_{r\rightarrow 0}
     &\intb_{\!\!\!B^\pm_r(x)}\!\! \bfvarphi(y)  dy  \quad 
 \H^2\text{\!-a.e. } x\in  \Sigma_{x_1},  \  \forall x_1\!\in\!(0,L),
\end{aligned}\end{equation}

%

 \begin{equation}\label{phistarphipm}
 \begin{aligned}
 \bfvarphi^\star   = \frac{1}{2}(  \bfvarphi^++\bfvarphi^-)  \quad 
 \H^2\text{\!-a.e. on}  \  \Sigma_{x_1},  \  \forall x_1\!\in\!(0,L),
\end{aligned}\end{equation}


 \begin{equation}\label{phistarphipminL1}
 \begin{aligned}
 \bfvarphi^\star, \ \bfvarphi^\pm  \in L^1_{\H^2}(\Sigma_{x_1}) \qquad   \forall x_1\!\in\!(0,L),
 \end{aligned}\end{equation}

	\begin{equation}
 	\label{phipm=phistar}
	\bfvarphi^+=\bfvarphi^-=\bfvarphi^\star= \lim_{r\rightarrow 0}
 \intb_{\!\!\!B^\pm_r(x)}\!\! \bfvarphi(y)  dy \ \H^2 \hbox{-a.e. in } \ \Sigma_{x_1} \quad \hbox{ if } \  |\bfE\bfvarphi|(\Sigma_{x_1}) =0,
	\end{equation}

	\begin{equation}
 	\label{varphi+=varphimae}
\bfE\bfvarphi\ll \nu\otimes\L^2 \qquad \Longrightarrow 	\qquad \bfvarphi^+=\bfvarphi^-=\bfvarphi^\star
	\quad   m\otimes\L^2\text{-a.e.}.
	\end{equation}
\end{lemma}

\begin{proof}
By   \cite[p. 84,  Trace Theorem; p. 91, Proposition 2.2]{Ko}  (see also  \cite[p. 209 (ii)-(iii)]{AmCoDa})    
the    traces  of   $\bfvarphi$ 
on both side of
 any  $\C^1$ hypersurface   $M$  contained in 
$\O$   are $\H^2$-a.e. equal to its one side Lebesgue limits  on both sides of $M$. Applying this to    $M=\Sigma_{x_1}$ for all $x_1\in (0,L)$,
we obtain  \eqref{traceball}.   
Assertion   \eqref{traceball} ensures that the two limits in  the first line of \eqref{defhalfprecise} exist and are finite
for  $\H^2-$a.e.   $x\in \Sigma_{x_1}$,  for all $x_1\in (0,L)$.  
When they do,
   the limit in the first line of \eqref{defprecise} also exists,  and 

 \begin{equation}
 \nonumber
 \begin{aligned}
\frac{1}{2}(  \bfvarphi^+(x)+\bfvarphi^-(x))&=
\frac{1}{2}\lp \lim_{r\to0} \intb_{B^+_r(x)} \bfvarphi(y) \ dy+\intb_{B^-_r(x)} \bfvarphi(y) \ dy
\rp
\\&  = \lim_{r\to0} \intb_{B_r(x)} \bfvarphi(y) \ dy=  \bfvarphi^\star(x),
\end{aligned}\end{equation}

therefore  \eqref{phistarphipm} holds.
 Assertion \eqref{phistarphipminL1} results from  \eqref{traceball},  \eqref{phistarphipm} and the fact that the traces of $\bfvarphi$ on each side of $\Sigma_{x_1}$ belong to $L^1_{\H^2}(\Sigma_{x_1})$.
Noticing that by  \eqref{EphilfloorM} we have
 
 \begin{equation}
 \nonumber
 \bfE\bfvarphi\lfloor_{\Sigma_{x_1}} = \left( \bfvarphi^+-\bfvarphi^-\right)\odot \bfe_1 \H^2 \lfloor_{\Sigma_{x_1}}
  \quad  \forall x_1\!\in\!(0,L),
     \end{equation}

we deduce from    the elementary inequality 

\begin{equation}\label{Elem}
|\bfa|\le \sqrt2 |\bfa\odot \bfn| \quad \hbox{ if } \ ||\bfn||=1  ,
 \end{equation}
 
that
	$\bfvarphi^+=\bfvarphi^-$ $\H^2$-a.e. in $\Sigma_{x_1}$ whenever $|\bfE\bfvarphi|(\Sigma_{x_1}) =0$. Assertion \eqref{phipm=phistar}  then follows from  \eqref{traceball}  and \eqref{phistarphipm}.
	We have 	$\nu\otimes\L^2(\Sigma_{x_1})=\nu(\{x_1\})\L^2(\O')$, therefore $\nu\otimes\L^2(\Sigma_{x_1})=0$ if $x_1\not\in\A_\nu$
(see \eqref{defAnu}). 
Accordingly, 	
if  $\bfE\bfvarphi\ll \nu\otimes\L^2$, then  $|\bfE\bfvarphi|(\Sigma_{x_1})=0$ for all $x_1\in (0,L)\setminus \A_\nu$.
Assertion \eqref{varphi+=varphimae}   then results  from \eqref{phipm=phistar}
and the fact that, by \eqref{nocommonatom}, $m(\A_\nu)=0$. 
\end{proof}

\begin{lemma} \label{lemubar} 
Let $\bfvarphi\in BD(\Omega)$ such that $\bfvarphi =0$ on $\partial \Omega$,
 and let  $\ov \bfvarphi\in L^1(0,L;\RR^3)$ be the Borel function  defined by 
 
\begin{equation} 
\label{defubar}
   \begin{aligned} 
 \ov\bfvarphi(x_1):= \int_{\Sigma_{x_1}} \bfvarphi^\star d\H^2\quad  \forall x_1\in (0,L). 
\end{aligned}
\end{equation}

Then 

\begin{equation} 
\label{ovuborneBV}
   \begin{aligned}
&   \ov \bfvarphi\!\in\! BV (0,\!L;\RR^3), \quad   \!\!\! || \ov \bfvarphi||_{_{ L^1 (0,L;\RR^3)}}\!= \! ||  \bfvarphi||_{ L^1(\O)}, \quad\! \!\!|| \ov \bfvarphi||_{_{ BV (0,L;\RR^3)}}\!\le\! \sqrt2  ||  \bfvarphi||_{ BD (\O)},
\\&D\ov\bfvarphi \ll  |\bfE\bfvarphi | (.\times \Omega'), \quad 
 |D\ov\bfvarphi|(B )\le \sqrt2 |\bfE\bfvarphi |(B\times\Omega') \quad  \forall B\in \B((0,L)).
   \end{aligned}
\end{equation}


Moreover,  the left  (resp. right)
 limit   $\ov\bfvarphi^{(l)}$ (resp. $\ov\bfvarphi^{(r)}$)
of $\ov\bfvarphi$ at $x_1$ satisfies 

\begin{equation} 
\label{ovu=}
   \begin{aligned}
&\ov\bfvarphi^{(l)}(x_1)= \int_{\Sigma_{x_1}}\bfvarphi^-  d\H^2 \quad \forall x_1\in (0,L).
\\&\lp \hbox{resp. } \quad \ov\bfvarphi^{(r)}(x_1)= \int_{\Sigma_{x_1}}\bfvarphi^+  d\H^2 \quad \forall x_1\in (0,L)\rp.
\end{aligned} 
\end{equation}

\end{lemma}

\begin{proof} Let  $\hbox{\rm eV}(\ov\bfvarphi, (0,L))$ denote the    essential variation  of $\ov\bfvarphi$ on $(0,L)$, that is

\begin{equation} 
\label{eV}
   \begin{aligned}
 \hbox{\rm eV}(\ov\bfvarphi, (a,b))\!:=\!\!\!\! \inf_{\L^1(N)=0}\!\sup\la \sum_{i=1}^n  \lb \ov\bfvarphi(t_{i+1})-\ov\bfvarphi(t_{i })\rb,  
 \lb\begin{aligned} & t_1,\ldots,t_n\in (a,b)\setminus N
 \\ &a<t_1<\ldots<t_n<b
 \end{aligned}\rpt
 \ra.
   \end{aligned} 
\end{equation}

By  \cite[Proposition 3.6 and Theorem 3.27] {AmFuPa}, the field  $\ov\bfvarphi$ belongs to $BV(0,L; \RR^3)$ if and only if 
  $ \hbox{\rm eV}(\ov\bfvarphi, (0,L))<\infty$    and in this case  $ \hbox{\rm eV}(\ov\bfvarphi, (0,L))= |\bfD\ov\bfvarphi |((0,L))$.
 Let $a,b$ be two real numbers such that $0\le a<b\le L$, 
 $D:= \{ t \in (0,L), |\bfE\bfvarphi|(\Sigma_{t})>0\}$
and let $t_1,\ldots,t_n  \subset (a,b)\setminus D$ such that  $0<t_1<\ldots<t_n<L$. By \eqref{Elem},    \eqref{phipm=phistar},  and  Green's formula in $BD( \O_i)$, where $\O_i\!:=\!(t_i,t_{i+1})\!\times\O'$, we have, since  $\bfvarphi=0$ on $\partial \O$,

\begin{equation} 
\label{IPt}
   \begin{aligned}
 \lb \ov\bfvarphi(t_{i+1})-\ov\bfvarphi(t_{i })\rb
&=  \lb \int_{\Sigma_{t_{i+1}}} \bfvarphi^- d\H^2-  \int_{\Sigma_{t_{i}}} \bfvarphi^+ d\H^2 \rb
\\&\le \sqrt2  \lb \lp \int_{\Sigma_{t_{i+1}}} \bfvarphi^- d\H^2-  \int_{\Sigma_{t_{i}}} \bfvarphi^+ d\H^2 \rp \odot \bfe_1  \rb 
\\&=  \sqrt2  \lb \int_{\partial \O_i}\hskip-0,3cm \bfgamma_i( \bfvarphi) \odot \bfn d\H^2   \rb 
 =  \sqrt2  \lb \bfE\bfvarphi  \lp \O_i\rp   \rb \le   \sqrt2  \lb \bfE\bfvarphi  \rb \lp \O_i\rp,
 \end{aligned} 
\end{equation}

 where $\bfgamma_i( \bfvarphi)$ denotes the trace on $\partial\O_i$  of  the restriction of  $\bfvarphi$ to $\O_i$, therefore

\begin{equation} 
\nonumber
   \begin{aligned}
\sum_{i=1}^n  \lb \ov\bfvarphi(t_{i+1})-\ov\bfvarphi(t_{i })\rb
  \le  \sum_{i=1}^n  \sqrt2  \lb \bfE\bfvarphi  \rb \lp \O_i\rp\le    \sqrt2  \lb \bfE\bfvarphi  \rb \lp (a,b)\times\O'\rp.
 \end{aligned} 
\end{equation}

By the arbitrary choice of $t_1,\ldots,t_n $, noticing that $D$ is at most countable countable thus $\L^1$-negligible,   we infer 

\begin{equation}
\label{ovphiborne}
|\bfD\ov\bfvarphi|((a,b) ) =\hbox{\rm eV}(\ov\bfvarphi, (a,b)) \le  \sqrt2  \lb \bfE\bfvarphi  \rb \lp (a,b)\times\O'\rp,
\end{equation}
 
yielding, by the arbitrariness  of $a,b$,  the second line of  (\ref{ovuborneBV}). The first line  easily follows.
An argument  analogous to   \eqref{IPt} implies
 
 \begin{equation} 
\nonumber
   \begin{aligned}
 \lim_{t\to x_1^\pm, t\not\in D} \lb \ov\bfvarphi(t)-\int_{\Sigma_{x_1}} \bfvarphi^\pm d\H^2\rb  \le  \lim_{t\to x_1^\pm, t\not\in D}  \sqrt2  \lb \bfE\bfvarphi  \rb \lp (x_1,t)\times\O'\rp=0,%
 \end{aligned} 
\end{equation}

yielding   \eqref{ovu=}.
\end{proof}


In the   next proposition, we study the asymptotic behavior of a sequence $(\bfvarphi_\e)$ satisfying some suitable estimate (see \eqref{supFefiefini}).
This study  will  be applied to the  solution  to \eqref{Pe}  and also  to the sequence of test fields defined  in Section \ref{secproofth} (see Proposition \ref{propvarphie}).
These  fields  do  not necessarily  vanish on $\partial \O$. Accordingly,   we introduce the normed space

 \begin{equation}
 \label{defBDnum}
\hskip-0,2cm  \begin{aligned} 
 &
BD^{\nu,m}(\Omega) = \la  \bfvarphi \in BD(\Omega)\lb
\begin{aligned}
&  \bfE\bfvarphi \ll \nu\otimes\L^2, \  \tfrac{\bfE\bfvarphi}{\nu\otimes\L^2}
  \in L^2_{\nu\otimes\L^2}(\Omega;\RR^3)
\\& (\bfvarphi^\star)'\in L^2_{m }(0,L; H^1(\Omega';\RR^3))
\end{aligned}\rpt\ra,
\\& \left|\left|\bfvarphi \right|\right|_{_{BD^{\nu,m}(\Omega)}}\hskip-0,1cm := \hskip-0,1cm \int_\O\hskip-0,1cm |\bfvarphi |dx+ 
 \lp  \int_{\Omega} \hskip-0,1cm \lb \tfrac{\bfE\bfvarphi}{\nu\otimes\L^2}   \rb^2 \hskip-0,1cm d\nu\otimes\L^2\rp^{\frac{1}{2}}\hskip-0,2cm +  \lp\int_\Omega |\bfe_{x'}(\bfvarphi ^\star)|^2 dm\otimes\L^2\rp^{\frac{1}{2}}\hskip-0,2cm .
 \end{aligned} 
   \end{equation}


\begin{proposition}\label{propapriori}

\ni   Let $(\bfvarphi_\e)$ be a sequence in $W^{1,1}(\Omega;\RR^3)$ such that 

\begin{equation} 
  \label{supFefiefini}
   \begin{aligned}
&\sup_{\e>0}  \int_\Omega |\bfvarphi_\e| dx+  \int_\Omega \mu_\ep \left\vert \bfe(\bfvarphi_\e) \right\vert^2 dx
<\infty.
\end{aligned}
\end{equation}

\ni Then $(\bfvarphi_\e)$ is bounded in $BD(\Omega)$ and, 
up to a subsequence,

 \begin{equation} 
  \label{cvfi}
   \begin{aligned}
&\bfvarphi_\e \to \bfvarphi & & \hbox{strongly in } \ L^p(\Omega;\RR^3) \ \  \forall p\in \left[ 1,\tfrac{3}{2}\right),
\\ &\bfE\bfvarphi_\e \buildrel\star\over \rightharpoonup \bfE\bfvarphi  \qquad & & \hbox{weakly*  in } \ \M(  \ov\Omega;\SS^3),  
\end{aligned}
\end{equation}
 
 for some $\bfvarphi\in BD(\Omega)$, where the measure $\bfE\bfvarphi$ is extended by $0$  to $\ov\O$.
 Moreover
 
 \begin{equation}\label{cvfinuenu}
 \begin{aligned}
   &\!\bfE\bfvarphi \!\ll  \!\nu \!\otimes\! \L^2\!, 
      \quad    \tfrac{\bfE\bfvarphi}{\nu\otimes\L^2}   \! \in \!L^2_{\nu \otimes \L^2 }(\Omega; \SS^3),
\ \ \mu_\e  \bfe(\bfvarphi_\e) \!\buildrel{\nu_\e\otimes\L^2, \nu \otimes\L^2}\over \rrrrightharpoonup \!  \tfrac{\bfE\bfvarphi}{\nu\otimes\L^2}.
   \end{aligned}
\end{equation}

Assume in addition  

\begin{equation} 
  \label{supfiemefini}
   \begin{aligned}
&\sup_{\e>0}  \int_\Omega |\bfvarphi_\e'|^2 dm_\e\otimes \L^2  
<\infty,
\end{aligned}
\end{equation}

then 

\begin{equation}\label{cvfimem}
 \begin{aligned}
   &(\bfvarphi^\star)'\in L^2_m(0,L; H^1(\Omega';\RR^3)), \quad & & \bfvarphi\in BD^{\nu,m}(\Omega),
   \\&  \bfvarphi_\e'  \buildrel{m_\e\otimes\L^2, m \otimes\L^2}\over \rrrrightharpoonup   (\bfvarphi^\star)', \quad & &  \bfe_{x'}(\bfvarphi_\e') \buildrel{m_\e\otimes\L^2, m \otimes\L^2}\over \rrrrightharpoonup  \bfe_{x'}((\bfvarphi^\star)').
   \end{aligned}
\end{equation}



 \end{proposition}
 

\begin{proof}
By  \eqref{mueto}  and(\ref{supFefiefini}), we have

\begin{equation} 
  \label{fi2}
   \begin{aligned}
\int_\Omega | \bfvarphi_\e| dx + \int_\Omega |\bfe(\bfvarphi_\e)| dx & \le \int_\Omega | \bfvarphi_\e| dx+   \lp \int_\Omega\tfrac{1}{\mu_\e} dx  \rp^{\frac{1}{2}}  \lp \int_{\Omega} \mu_\e \lb \bfe(\bfvarphi_\e)  \rb^2 dx  \rp^{\frac{1}{2}}
\\&  \le C,
 \end{aligned}
\end{equation}

hence  $(\bfvarphi_\e)$ is bounded in $BD(\Omega)$ and weakly* converges, up to a subsequence, to some $\bfvarphi\in BD(\Omega)$.
Taking   the compactness of the injection  of $BD(\Omega)$  into $L^p(\Omega;\RR^3)$     for $p\in \lc1,\frac{3}{2}\rp$  into account (see \cite[Theorem 2.4, p. 153]{Te}), we deduce   
 \begin{equation} 
  \label{cvfi0}
   \begin{aligned}
&\bfvarphi_\e \to \bfvarphi & & \hbox{strongly in } \ L^p(\Omega;\RR^3) \ \  \forall p\in \left[ 1,\tfrac{3}{2}\right),
\\ &\bfE\bfvarphi_\e \buildrel\star\over \rightharpoonup \bfE\bfvarphi  \qquad & & \hbox{weakly*  in } \ \M(   \Omega;\SS^3).
\end{aligned}
\end{equation}

On the other hand, assumptions  (\ref{defmenue}) and (\ref{supFefiefini})   imply 

\begin{equation} 
   \label{fi22}
   \begin{aligned}
\sup_{\e>0}  \int_{\Omega}  \lb\mu_\e \bfe(\bfvarphi_\e)  \rb^2 d\nu_\e\otimes\L^2 <\infty,
 \end{aligned}
\end{equation}
thus, by Lemma \ref{lemfeps},  the sequence $ \mu_\e  \bfe(\bfvarphi_\e))$ weakly converges with respect to the pair $(\nu_\e\otimes\L^2, \nu \otimes\L^2)$,
  up to a subsequence, to some $\bfXi\in L^2_{\nu\otimes\L^2}(\Omega;\SS^3)$.
 Noticing that $\mu_\e  \bfe(\bfvarphi_\e) \nu_\e\otimes\L^2= \bfE \bfvarphi_\e$,  
  the definition    \eqref{fetetoft}   indicates  that 
\begin{equation} 
   \label{cvfi1}
   \begin{aligned}
\bfE \bfvarphi_\e \buildrel{\star}\over\rightharpoonup 
 \bfXi \nu\otimes \L^2\quad \hbox{ weakly* in } \ \M(\ov\O;\SS^3).
 \end{aligned}
\end{equation}

Since $\nu\otimes \L^2(\partial\O)=0$, the convergences \eqref{cvfi0} and \eqref{cvfi1} imply that the extension 
by $0$ of $\bfE\bfvarphi$ to $\ov\O$ is equal to $ \bfXi \nu\otimes \L^2$  and that  $\bfXi=  \tfrac{\bfE\bfvarphi}{\nu\otimes\L^2}$.
Assertions \eqref{cvfi} and  (\ref{cvfinuenu}) are  proved. 
To prove \eqref{cvfimem}, we first notice that  \eqref{fi22} implies     $\sup_{\e>0}  \int_{\Omega}   \lb \bfe_{x'}(\bfvarphi_\e')  \rb^2 dm_\e\otimes\L^2<+\infty$. Therefore,  if \eqref{supfiemefini}
is satisfied, then by Lemma \ref{lemfeps} and \eqref{fetetoft}  the following convergences hold, up to a subsequence,
%
%
%
%

 \begin{equation}\label{mueefietoGamma}
 \begin{aligned}
 &   \bfvarphi_\e'     \buildrel{m_\e\otimes\L^2, m \otimes\L^2}\over \rrrrightharpoonup    \bfh' ,  
\quad & & \mu_\e \bfvarphi_\e'   \buildrel \star \over \rightharpoonup    \bfh' m\otimes \L^2 \quad \hbox{weakly* in } \ \M(\ov\Omega;\RR^3),
  \\& \bfe_{x'}(\bfvarphi_\e')   \buildrel{m_\e\otimes\L^2, m \otimes\L^2}\over \rrrrightharpoonup  \bfGamma , 
   \quad & & \mu_\e \bfe_{x'}(\bfvarphi_\e' )   \buildrel \star \over \rightharpoonup   \bfGamma m\otimes \L^2 \quad \hbox{weakly* in } \ \M(\ov\Omega;\SS^3),
      \end{aligned}
\end{equation}


 for some $\bfh'\in L^2_{m\otimes\L^2}(\Omega)$, $\bfGamma\in L^2_{m\otimes\L^2}(\Omega;\SS^3)$.
%
%
   Assertion \eqref{cvfimem} is proved  provided we show  that 

 \begin{equation} \label{h'=fistar'}  \begin{aligned} 
  \bfh'= (\bfvarphi^\star)'
    \quad m\otimes\L^2\hbox{-a.e.  in } \, \Omega,
     \quad
\end{aligned}\end{equation}
  
  \begin{equation} 
\label{Gamma=ex'fistar'}
   \begin{aligned}
\lp  \bfvarphi^\star \rp' \in L^2_m(0,L; H^1 (\Omega';\RR^3)), \quad    \bfGamma= \bfe_{x'} \lp \lp  \bfvarphi^\star \rp' \rp\quad m\otimes\L^2\hbox{-a.e.  in } \, \Omega.  \end{aligned}
\end{equation}

  {\it Proof of \eqref{h'=fistar'}.} Let us fix $\psi \in \D(\Omega)$. By (\ref{cvfi}),   $(\psi \bfvarphi_\e)$ 
  weakly* converges in $BD(\Omega)$  to $\psi \bfvarphi$,
  hence, 
  by  Lemma \ref{lemubar}, 
  the sequence $(\ov{ \psi \bfvarphi_\e})$ defined by (\ref{defubar})  weakly* converges   in $BV(0,L;\RR^3)$  to 
  $\ov{\psi \bfvarphi}$.  By   \eqref{DL1num}, (\ref{ovuborneBV}) and  (\ref{cvfinuenu})  we have,   
  
  $$
  |D\ov{\psi \bfvarphi}|\ll |\bfE(\psi \bfvarphi)|(.\times \Omega')=  |\psi \bfE(\bfvarphi)+ \bfnabla\psi\odot\bfvarphi \L^3|(.\times \Omega')\ll\nu\otimes\L^2,
  $$
  
therefore, by    (\ref{nocommonatom}), the assumptions of Lemma \ref{lemBoPi} 
%
%
are satisfied by $(b_\e, w_\e):= (\mu_\e, \ov{\psi \bfvarphi_\e})$ and $(\theta, w):= (m,\ov{ \psi \bfvarphi})$. Taking  \eqref{varphi+=varphimae}, \eqref{mueefietoGamma} and  \eqref{h'=fistar'} into account and applying Fubini's theorem, we deduce
  
  \begin{equation} 
  \nonumber
   \begin{aligned} 
\int_\Omega  \psi  \bfh'   dm\otimes\L^2&=  \lime \int_\Omega \mu_\e \psi  \bfvarphi_\e' dx  = \lime\int_0^L \mu_\e \ov{\psi \bfvarphi_\e'} dx_1
 = \int_{(0,L)} (\ov{\psi \bfvarphi'})^+dm 
  \\&= \int_{(0,L)} \ov{  \psi (\bfvarphi')^+}dm
 =\int_\Omega  \psi  ( \bfvarphi')^+  dm\otimes\L^2 =\int_\Omega  \psi  ( \bfvarphi')^\star  dm\otimes\L^2.
  \end{aligned}
\end{equation}

By the arbitrary choice of $\psi$,
  Assertion  \eqref{h'=fistar'} is proved.  
\qed

{\it Proof of \eqref{Gamma=ex'fistar'}.} 
     Let us fix   $\bfPsi\in \D(\Omega;\SS^3)$.  By (\ref{mueefietoGamma}) and \eqref{h'=fistar'}, we have

 \begin{equation} 
\label{g=2}
   \begin{aligned}
   \int_\Omega \bfGamma:\bfPsi \,dm\otimes \L^2&=\lime 
   \int_\Omega  \mu_\e \bfe(\bfvarphi_\e) :\bfPsi  dx
 =  \lime -   \int_\Omega \mu_\e \bfvarphi_\e' \cdot \bfdiv\bfPsi dx 
\\& =  -   \int_\Omega (\bfvarphi^\star)' \cdot \bfdiv\bfPsi    dm\otimes \L^2
.
  \end{aligned}
\end{equation}
 

Choosing   a test field of the type  $\bfPsi(x)= \eta(x_1)\bfUpsilon(x')$, with $\eta\in \D(0,L)$ and $\bfUpsilon\in \D(\Omega';\SS^3)$, and letting $\eta$ vary in $\D(0,L)$,  
we infer, noticing that $h'_1=0$, 

\begin{equation} 
\label{g=22}
   \begin{aligned}
   \int_{\O'} \!\!\!\bfGamma(x_1,x')\! :\!\bfUpsilon(x')  dx'&\!=\! -  \!\! \int_{\O'} \!\!(\bfvarphi^\star)' (x_1,x') \!\cdot\! \bfdiv\bfUpsilon (x')   dx' \quad \forall  \ x_1\!\in \!(0,L)\!\setminus \!N_{\bfUpsilon},\!\!
  \end{aligned}
\end{equation}
 

for some $m$-negligible subset $N_{\bfUpsilon}$ of $(0,L)$.
Letting ${\bfUpsilon}$ vary in a countable   subset $\C$  of $\D(\Omega';\SS^3)$ dense in $H^1_0(\Omega';\SS^3)$, 
and denoting by $N_0$ some $m$-negligible subset of $(0,L)$ such that 
$\bfGamma(x_1, x')\in L^2(\Omega';\SS^3) \ \forall x_1\in (0,L)\setminus N_0$, 
we obtain the following equalities in the sense of distributions in  $\D'(\Omega';\SS^3)$:

\begin{equation} 
\label{g=222}
   \begin{aligned}
  \bfe_{x'}\lp \lp  \bfvarphi^\star \rp' \rp(x_1,\cdot)  =   \bfGamma (x_1,\cdot)   \qquad  \forall  \ x_1\in (0,L)\setminus \bigcup_{{\bfUpsilon}\in \C}  N_0\cup N_{\bfUpsilon}, \end{aligned}
\end{equation}

and deduce 
that  $\bfe_{x'}\lp \lp  \bfvarphi^\star \rp' \rp \in L^2_m(0,L; L^2(\Omega'; \SS^3))$.
This, along  with \eqref{h'=fistar'} and the two-dimensional  second Korn inequality in $ H^1(\Omega';\RR^2)$, implies that 
$ \lp \bfvarphi^\star \rp'  \in L^2_m(0,L; H^1(\Omega'; \SS^3))$.
 Assertion \eqref{Gamma=ex'fistar'} is proved. 
\end{proof}

  We are now in a position to prove the main result of  Section \ref{secapriori}.


\begin{proof}[\bf Proof of Proposition \ref{propaprioriu}]
By multiplying \eqref{Pe} by $\bfu_\e$ and by integrating it  by parts over $\O$, we obtain 
$\int_\Omega \bfsigma_\e(\bfu_\e) :\bfe(\bfu_\e)  dx = \int_\Omega\bff \cdot\bfu_\e dx $, and deduce 
%
\begin{equation} 
  \label{tt}
   \begin{aligned}
\int_\Omega  \mu_\e |\bfe(\bfu_\e)|^2 dx\le \int_\Omega  \bfsigma_\e(\bfu_\e):\bfe(\bfu_\e) dx \le C ||\bff||_{L^\infty(\O;\RR^3)}\int_\Omega |\bfu_\e| dx.
\end{aligned}
\end{equation}

  The assumptions  \eqref{mueto}, Poincar\'e   and Cauchy-Schwarz inequalities, imply 

\begin{equation} 
\label{ttt}
   \begin{aligned}
\int_\Omega |u_{\e1}| dx & \le C \int_\Omega \lb \frac{\partial u_{\e1}}{\partial x_1}\rb dx
 \le C \lp \int_\Omega \frac{1}{\mu_\e} dx \rp^{\frac{1}{2}}  
 \lp \int_\Omega  \mu_\e    \lb \frac{\partial u_{\e1}}{\partial x_1}\rb^2 dx\rp^{\frac{1}{2}}
 \\&  \le C \lp\int_\Omega  \mu_\e |\bfe(\bfu_\e)|^2 dx \rp^{\frac{1}{2}}.
\end{aligned}
\end{equation}

\ni  
By  Fubini's Theorem, 
Poincar\'e's inequality in $W_0^{1,1}(\Omega';\RR^2)$, 
Assertion \eqref{mueto},   Cauchy-Schwarz   and 
Jensen's  inequalities, and   Korn's  inequality in $H_0^1(\Omega';\RR^2)$, we have

\begin{equation} 
\label{tttt}
   \begin{aligned}
\int_\Omega\!\! |\bfu_\e'| dx & \!\le C \!\!   \int_{\Omega} \!\! \lb\bfnabla_{x'}\bfu_\e' \rb\! dx 
 \! \le \!C \lp\! \int_0^L \!\!\!\!\tfrac{1}{\mu_\e} dx_1\!\! \rp^{\!\!\frac{1}{2}}  \hskip-0,2cm \lp\! \int_0^L  \!\!\!{\mu_\e} \lp \int_{\Omega'} \!\! \lb\bfnabla_{x'}\bfu_\e' \rb dx'\!\rp^2\!\! dx_1  \!\rp^{\!\!\frac{1}{2}}  
 \\&\!\le \!C \! \lp \!\int_0^L \hskip-0,2cm {\mu_\e} \!\!  \! \int_{\Omega'} \!\! \lb\bfnabla_{x'}\bfu_\e' \rb^2 \!\!dx' \! dx_1 \!\! \rp^{\!\!\frac{1}{2}}  
 \hskip-0,2cm  \le \!C \!
 \lp \int_0^L\!\!  {\mu_\e}   \int_{\Omega'}  \lb\bfe_{x'}(\bfu_\e') \rb^2 dx'  dx_1\rp^{\!\!\frac{1}{2}}  \hskip-0,2cm.  
\end{aligned}
\end{equation}

%
%
%

We deduce from \eqref{tt}, \eqref{ttt}, \eqref{tttt} that $\int_\Omega |\bfu_\e| dx   \le C\lp \int_\Omega |\bfu_\e| dx   \rp^{\tfrac{1}{2}}$, yielding

\begin{equation} 
\label{u3}
   \begin{aligned}
\int_\Omega |\bfu_\e| dx & \le C .
\end{aligned}
\end{equation}

On the other hand, by   Korn's inequality  in $H^1_0(\Omega';\RR^2)$, we  have
 
   \begin{equation} \label{ue'meborne}   \begin{aligned}   \int_\Omega   |\bfu_\e'|^2 dm_\e \otimes\L^2&
    = \int_0^L  \hskip-0,2cm {\mu_\e } \lp  \int_{\Omega'} |\bfu_\e'|^2 dx'\rp dx_1  
   \\&  \le
  C  \int_0^L  {\mu_\e } \lp  \int_{\Omega'} |\bfe_{x'}(\bfu'_\e)|^2 dx'\rp dx_1   
  \le   C \int_\Omega  \mu_\e |\bfe(\bfu_\e)|^2 dx .
  \end{aligned}\end{equation}

By  (\ref{tt}),  (\ref{u3}), and \eqref{ue'meborne}, the estimate \eqref{supFeuefini} is proved. This means that   $\bfvarphi_\e=\bfu_\e$
satisfies \eqref{supFefiefini} and \eqref{supfiemefini}. 
Therefore, by Proposition \ref{propapriori},  the convergences  stated in \eqref{cvu} hold for some $\bfu\in BD^{\nu,m}(\Omega)$. 
The proof  of Proposition \ref{propaprioriu} is achieved provided we show that 

   \begin{equation} \label{u=0partialO}   \begin{aligned} 
   \bfu=0 \quad \hbox{ on } \ \partial \O,
  \end{aligned}\end{equation}

  (recall  that the trace    is  not weakly* continuous on $BD(\O)$) 
  and

   \begin{equation} \label{u'=0partialO}   \begin{aligned} 
   (\bfu^\star)' =0 \quad   m\otimes  \H^1\hbox{-a.e. on }
  \partial \O'\times  (0,L).
  \end{aligned}\end{equation}

{\it Proof of \eqref{u=0partialO}.}  
Let us fix $\bfPsi\in C^\infty(\ov\O;\SS^3)$. By passing to the limit as $\e\to 0$ in the integration by parts formula $\int_\Omega \bfe(\bfu_\e):\bfPsi dx = - \int_\Omega \bfu_\e\cdot\bfdiv\bfPsi  dx, $
%
taking the strong convergence of $\bfu_\e$ to $\bfu$ in $L^1(\Omega;\RR^3)$ and the weak* convergence 
of $(\bfe(\bfu_\e))$ to $\bfE(\bfu)$ in $\M(\ov\O;\SS^3)$  into account,    we obtain  $\int_\Omega  \bfPsi :d\bfE\bfu  = - \int_\Omega \bfu\cdot\bfdiv\bfPsi  dx$.
%
We then deduce from  the Green Formula in $BD(\Omega)$ (see \cite[Theorem  II-2.1]{Te}) 
$$
\int_\Omega \bfPsi : d \bfE\bfu =- \int_\Omega \bfu\cdot \bfdiv\bfPsi  dx
+ \int_{\partial \Omega}  \bfPsi:\bfu\odot \bfn d\H^2 ,
$$


that $  \int_{\partial \Omega}  \bfPsi:\bfu\odot \bfn d\H^2(x) =0$.
By the arbitrariness of $\bfPsi$,  taking \eqref{Elem} into account, Assertion  \eqref{u=0partialO} is proved.\qed

{\it Proof of \eqref{u'=0partialO}.} Let us fix $\bfPsi\in C^\infty(\ov\Omega;\SS^3)$.
 Since $\bfu_\e=0$ on $\partial \O$, 
  \eqref{g=2}   holds for  such a  $\bfPsi$ when  $\bfvarphi_\e=\bfu_\e$. We infer
   
    \begin{equation} 
\label{holds}
   \begin{aligned}
   \int_\Omega \bfe_{x'}(\bfu'):\bfPsi \,dm\otimes \L^2&
   =    -  \int_{(0,L)} \lp \int_{\O'} (\bfu^\star)' \cdot \bfdiv\bfPsi  dx'   \rp dm(x_1)
.
  \end{aligned}
\end{equation}

By \eqref{cvfimem} applied to $\bfvarphi_\e:=\bfu_\e$, the field  $(\bfu^\star)'$ belongs to $L^2_m(0,L; H^1(\O';\RR^3))$,  
hence there exists an $m$-negligible subset $N$ of $(0,L)$ such that $(\bfu^\star)'(x_1,.)\in H^1(\O';\RR^3)$ for all $x_1\in (0,L)\setminus N$. 
By integration by parts, taking the symmetry of $\bfPsi$ into account, we infer

   \begin{equation} 
\label{mae}
   \begin{aligned}
  \int_{\O'}\!\! (\bfu^\star)'\! \!  \cdot\!  \bfdiv\bfPsi  dx' \! =\! \!  \int_{\partial \O'}\! \!  (\bfu^\star)' \! \! \cdot \! \bfPsi\bfn  d\H^1(x') - \! \! \int_{\O'} \! \!  \bfe_{x'}((\bfu^\star)' ):\bfPsi dx'
  \quad m\hbox{-a.e. } x_1.
  \end{aligned}
\end{equation}

It follows  from \eqref{holds} and \eqref{mae} that $\int_{(0,L)\times \partial \O'} (\bfu^\star)' \cdot \bfPsi\bfn  dm\otimes \H^1 =0$.
%
%
%
%
%
%
%
%
By the arbitrary choice of $\bfPsi$,  
Assertion \eqref{u'=0partialO} is proved.
 \end{proof}

{\color{black}
  \section{Partial mollification in  $BD^{\nu,m}(\Omega)$}\label{secsetH}

For any two Borel functions $f,g: \O\to \ov\RR$, we denote by $f\ast_{x'} g$ the partial convolution of $g$ and $f$ with respect to the variable $x'$,  defined by 

 \begin{equation}
 \label{defconv} 
	 f \ast_{x'} g(x) := 
	 \la 	\begin{aligned} &\int_{\RR^2} \widetilde f(x_1, x'-y')\widetilde g(y') dy'  & & \hbox{if } \  \widetilde f(x_1, x'-.)\widetilde g(.) \in L^1(\RR^2),
	 \\& 0 & & \hbox{otherwise},
   \end{aligned} \rpt	\end{equation}

and by $f^\delta$ the  ``partial mollification" of $f$ with respect to  $x'$ given  by

\begin{equation}
\label{deffdelta}
\begin{aligned}  
f^\delta := f\ast_{x'} \eta_\delta,
\end{aligned}	\end{equation}

where  $\eta_\delta \in \D(\RR^2)$ denotes the standard mollifier defined by

 \begin{equation*}
	\begin{aligned}  
\eta(x'):=\la \begin{aligned} & C \exp\lp \frac{1}{|x'|^2-1}\rp \quad & & \hbox{if } \quad |x'|<1,
\\& 0 & & \hbox{otherwise},
\end{aligned}\rpt \qquad \eta_\delta(x'):=\frac{1}{\delta^2} \eta\lp\frac{x'}{\delta}\rp,
   \end{aligned}
\end{equation*}
  the constant $C$ being  chosen  so  that $\int_{\RR^2} \eta dx'=1$.
  Some basic properties are stated  in the next lemma.

  \begin{lemma} 
    Let $f: \O\to \ov\RR$ be  a Borel function,  $\theta$  a positive Radon measure on $[0,L]$, and $\delta>0$. Then, 
the function $f^\delta$ is Borel measurable. If  $f\in L^p_{\theta\otimes \L^2}(\O)$    $(p\in[1,+\infty))$,  
  then 


%
%

 \begin{equation}
\label{fdeltainLpO'}
\begin{aligned}  
&
\int_{\O'} |f^\delta(x_1,x')|^p dx' \le \int_{\O'} |f (x_1,x')|^p dx' \qquad  \forall x_1\in (0,L),  
\end{aligned}	\end{equation}

 \begin{equation}
\label{fdeltainLp}
\begin{aligned}  
& f^\delta \in L^p_{\theta\otimes \L^2}(\Omega),\quad 
|| f^\delta ||_{ L^p_{\theta\otimes \L^2}(\O)} \le || f  ||_{ L^p_{\theta\otimes \L^2}(\O)} ,
\end{aligned}	\end{equation}

  \begin{equation}
\label{fdeltatof}
	\begin{aligned}  
  \lim_{\delta\to 0} 	 \int_\Omega | f-f^\delta|^p d\theta\otimes \L^2=0,
   \end{aligned}	\end{equation}

  \begin{equation}
\label{fdeltareg}
	\begin{aligned}  
&f^\delta(x_1,.)\in C^\infty(\ov\O') \quad \forall x_1\in (0,L), \quad 
\\[2pt]& \tfrac{\partial^{n+m}}{\partial x_2^n x_3^m}f^\delta = f \ast_{x'}  \tfrac{\partial^{n+m}}{\partial x_2^n x_3^m} \eta_\delta
 \in L^p_{\theta\otimes \L^2}(\O), \quad \forall n,m\in \NN
 \\[2pt]& \lb\lb \tfrac{\partial^{n+m}}{\partial x_2^n x_3^m}f^\delta\rb\rb_{ L^p_{\theta\otimes \L^2} }\hskip-0,5cm \le \tfrac{C}{\delta^{n+m}} \lb\lb  f \rb\rb_{ L^p_{\theta\otimes \L^2} }
\quad \forall n,m\in \NN.
   \end{aligned}	\end{equation}

If  $f\in L^p_{\theta\otimes \L^2}(\O)$  
 and $h\in L^{p'}_{\theta\otimes \L^2}(\O)$ $(\tfrac{1}{p}+\tfrac{1}{p'}=1)$, then 

\begin{equation}
\label{fdg=fgd}
  	\begin{aligned}
 &\int_{\Omega} f^\delta  h d\theta \otimes \L^2 = \int_{\Omega}  f h^\delta  d\theta \otimes \L^2. 
\end{aligned}
\end{equation}

If $\psi\in C^1_c( \Omega)$, then $\psi^\delta\in C(\ov\O)$ and 




\begin{equation}
\label{psid,k=psi,kd}
  	\begin{aligned}
 &  \frac{\partial \lp \psi^\delta\rp}{\partial x_k}= \lp \frac{\partial \psi}{\partial x_k}\rp^\delta
 , \ \forall k\in \{1,2,3\}.
\end{aligned}
\end{equation}

  \end{lemma}

\begin{proof}
By Fubini's theorem, 
  the mappings $h^\pm(x):=  \int_{\RR^2} (\widetilde f (x_1, x'-y')  \eta_\delta (y'))^\pm dy'$  (where $l^+(x):= \sup \{l(x),0\}$)
  are Borel measurables
  and so is the set  $A:=\{ x \in \O,   \int_{\RR^2} \lb \widetilde f (x_1, x'-y') \eta_\delta (y') \rb dy'<+\infty\} $, 
therefore, $f\ast_{x'} \eta_\delta =  (h^+-h^-)\mathds{1}_{A}$ is Borel measurable. 
Assertion \eqref{fdeltainLpO'}   follows from
 the classical  properties of convolution in $\RR^2$
(notice that  $\int_{\RR^2} \eta_\delta dx'=1$).
Assertion \eqref{fdeltainLp} is a straightforward consequence of \eqref{fdeltainLpO'}.
We have 

  \begin{equation}
\nonumber
	\begin{aligned}  
  \int_\Omega | f-f^\delta|^p d\theta\otimes \L^2=\int_{[0,L]} d\theta (x_1) \int_{\O'} | f-f^\delta|^p (x_1, x') dx'.
   \end{aligned}	\end{equation}

By  \eqref{fdeltainLpO'}, the following holds  $ \int_{\O'} | f-f^\delta|^p (., x') dx' \le 2^{p-1}  \int_{\O'} | f |^p (., x') dx'   \in L^1_\theta$, and  by the properties of mollification in $L^p(\O')$, 
 for all $x_1$ such that $f(x_1,.)\in L^p(\O')$, thus for $\theta$-a.e. $ x_1\in [0,L]$, 
 $ \int_{\O'} | f-f^\delta|^p (x_1, x') dx'$
converges to  $0$.
Assertion   \eqref{fdeltatof} then results  from the dominated convergence theorem. 
  Assertion \eqref{fdeltareg} follows from well known properties of mollification. 
  Assertion \eqref{fdg=fgd} is proved  by applying   Fubini's theorem several times.
 Assertion \eqref{psid,k=psi,kd}  is obtained by noticing that  $\widetilde\psi\in C^1_c(\RR^2)$  and by differentiating under the integral sign.
  \end{proof}

  The next proposition specifies some properties of partial mollification when applied to elements of $BD^{\nu,m}_0(\O)$. 
  We set
  
  \begin{equation} 
\label{defsigmanu}
\begin{aligned}
  & \bfsigma^\nu(\bfvarphi):= l \,{\rm tr}\,\lp \tfrac{\bfE\bfvarphi}{\nu\otimes\L^2}\rp \bfI + 2\,\tfrac{\bfE\bfvarphi}{\nu\otimes\L^2}.
  \end{aligned}
\end{equation}

  \begin{proposition}
  	\label{propmol}
	Let $\bfv\in BD^{\nu,m}_0(\O)$ and $\delta>0$. Then,

\begin{equation}
\label{vdeltaBD}
\begin{aligned}
&\bfv^\delta \in  
  	BD(\Omega), 
\quad 	\bfE\bfv^\delta \ll \nu\otimes\L^2, \quad \tfrac{ \bfE (\bfv^\delta)}{\nu\otimes \L^2}=  \lp \tfrac{ \bfE \bfv}{\nu\otimes \L^2} \rp^\delta,
  \end{aligned}
  \end{equation}
  
  \begin{equation}
\label{vd+=v+d}
\begin{aligned}
&(\bfv^\delta)^\pm   =  (\bfv^\pm)^\delta \quad  \H^2\hbox{-a.e. on }\ \Sigma_{x_1}, \quad & &\forall x_1\in (0,L), 
\\&(\bfv^\delta)^\star  =  (\bfv^\star)^\delta   \quad  \H^2\hbox{-a.e. on }\ \Sigma_{x_1}, \quad & &\forall x_1\in (0,L),
  \end{aligned}
  \end{equation}
  
  \begin{equation}
\label{vdeltaL2m}
\begin{aligned}
&\lp\lp\bfv^\delta\rp^\star\rp' \in  
  	L^2_m(0,L; H^1(\O';\RR^3)), 
\quad 	\bfe_{x'}\lp\lp\bfv^\delta\rp^\star\rp= \lp\bfe_{x'}(\bfv^\star)\rp^\delta,
  \end{aligned}
  \end{equation}
  
  \begin{equation}
\label{vdeltaBDnum}
\begin{aligned}
&\bfv^\delta\in BD^{\nu,m}(\O), \quad  \lim_{\delta\to0} \lb\lb \bfv-\bfv^\delta\rb\rb_{BD^{\nu,m}(\O)}=0,
  \end{aligned}
  \end{equation}
  
and the following holds for all $x\in\O$, $\a\in\{2,3\}$: 
  		
\begin{equation}   \label{610} 	\begin{aligned}
  	&\lim_{\kappa \rightarrow 0^+ }  (\bfv^\delta)^\mp (x_1\pm\kappa,x') =  (\bfv^\delta)^\pm(x),
	\\&\lim_{\kappa \rightarrow 0^+ }  \tfrac{\partial}{\partial x_\a}(\bfv^\delta)^\mp (x_1\pm\kappa,x') =  \tfrac{\partial}{\partial x_\a}(\bfv^\delta)^\pm(x),
\end{aligned}\end{equation}

  	\begin{equation} 
	\label{varphi+=}
	\begin{aligned}
	&\begin{aligned} 
   \lp  v_1^\delta\rp^+  (x)  =
	\frac{1}{l+2} \int_{(0,x_1]} &(\bfsigma^\nu)_{11}(\bfv^\delta)(s_1,x') \ d\nu(s_1)
	\\& \hskip1cm -\sum_{\b=2}^3 \frac{l}{l+2} \int_0^{x_1} \frac{\partial v_\b^\delta}{\partial x_\b}(s_1,x') ds_1,
	\end{aligned}	  
	\\ &\begin{aligned}
     \lp  v_\a^\delta\rp^+ (x) = \int_{(0,x_1]} (\bfsigma^\nu)_{1\alpha}(\bfv^\delta)(s_1,x') \ d\nu(s_1) - \int_{0}^{x_1} \pd{v^\delta_1}{x_\alpha}(s_1,x') \ ds_1   . \end{aligned}
	\end{aligned}
  	\end{equation}

\end{proposition} 

\begin{proof} By \eqref{fdeltainLp} we have  $\bfv^\delta\in L^1(\O;\RR^3)$ and  $\int_\O|\bfv^\delta|dx\le \int_\O|\bfv|dx$. Let us fix $\bfPsi\in \D(\O;\SS^3)$. Then 
$\bfPsi^\delta \in C^\infty(\ov\O;\SS^3)$, thus using   \eqref{fdg=fgd},  \eqref{psid,k=psi,kd},    Green's formula in $BD(\O)$, and the fact that $\bfv\in BD^{\nu,m}_0(\O)$, we obtain

 \begin{equation}
\nonumber
\begin{aligned}
  	 \int_{\Omega}\bfv^\delta\cdot\bfdiv\bfPsi \ dx & =   \int_{\Omega}\bfv \cdot (\bfdiv\bfPsi)^\delta \ dx = \int_{\Omega}\bfv \cdot \bfdiv(\bfPsi^\delta) \ dx 
	 = -\int_{\O} \bfPsi^\delta:d\bfE\bfv
	  \\&=  -\int_{\O} \bfPsi^\delta:\tfrac{\bfE\bfv}{\nu\otimes\L^2} d\nu\otimes\L^2
  =  -\int_{\O} \bfPsi:\lp\tfrac{\bfE\bfv}{\nu\otimes\L^2}\rp^\delta d\nu\otimes\L^2 .
  \end{aligned}
  \end{equation} 

By the arbitrary choice of $\bfPsi$, the assertion \eqref{vdeltaBD} is proved. Similarly, applying  Green's formula in $BD(\O)$ and  using  \eqref{fdg=fgd},  \eqref{psid,k=psi,kd},  
\eqref{vdeltaBD}, we infer, for all $x_1\in (0,L)$, 

 \begin{equation}
\nonumber
\begin{aligned}
  	  \int_{\Sigma_{x_1}} \hskip-0,4cm  \bfPsi\hskip-0,1cm:\hskip-0,1cm \lp\bfv^\delta\rp^-\hskip-0,3cm\odot\bfe_1   dx  
	&= \int_{\partial([0,x_1]\times\O')} \hskip-1,4cm \bfPsi: \bfv^\delta\odot\bfn \ dx
  =  \int_{[0,x_1]\times\O'}  \hskip-1 cm\bfPsi: d\bfE\bfv^\delta + \int_{[0,x_1]\times\O'}\hskip-1cm \bfdiv \bfPsi\cdot \bfv^\delta dx 
	 \\& =  \int_{[0,x_1]\times\O'}  \bfPsi: \lp\tfrac{\bfE\bfv}{\nu\otimes\L^2}\rp^\delta d\nu\otimes\L^2  + \int_{[0,x_1]\times\O'} (\bfdiv \bfPsi)^\delta  \cdot \bfv dx  
	  \\& =  \int_{[0,x_1]\times\O'}  \bfPsi^\delta :  \tfrac{\bfE\bfv}{\nu\otimes\L^2} d\nu\otimes\L^2  + \int_{[0,x_1]\times\O'}  \bfdiv (\bfPsi^\delta )   \cdot \bfv dx
	  \\& =\int_{\Sigma_{x_1}}  \bfPsi^\delta: \bfv^- \odot\bfe_1   dx =\int_{\Sigma_{x_1}}  \bfPsi: \lp\bfv^-\rp^\delta \odot\bfe_1   dx.
	   \end{aligned}
  \end{equation} 

By the arbitrary nature of $\bfPsi$,
(arguing in the same manner for $\lp\bfv^\delta\rp^+$), the first line of \eqref{vd+=v+d} is proved. By  \eqref{phistarphipm} and the first line of  \eqref{vd+=v+d}, 
for all $x_1\in(0,L)$  the following equalities hold $\H^2$-a.e. on $\Sigma_{x_1}$:

 \begin{equation}
\nonumber
\begin{aligned}
  	\lp\bfv^\star\rp^\delta= \tfrac{1}{2} \lp \bfv^++\bfv^-\rp^\delta=  \tfrac{1}{2} \lp\lp \bfv^+\rp^\delta+\lp\bfv^-\rp^\delta\rp
	=  \tfrac{1}{2} \lp\lp \bfv^\delta\rp^++\lp\bfv^\delta\rp^-\rp=\lp\bfv^\delta\rp^\star.
	   \end{aligned}
  \end{equation} 
  
  Assertion \eqref{vd+=v+d} is proved. 
  By \eqref{defBDnum0},   \eqref{fdeltainLp} and \eqref{vd+=v+d}, we have  $\lp\lp\bfv^\delta\rp^\star\rp' \in L^2_m(0,L;$ $L^2(\O';\RR^3))$.
  Taking   \eqref{fdg=fgd},  \eqref{psid,k=psi,kd}, \eqref{vd+=v+d} into account and integrating by parts with respect to $x'$  in $L^2_m(0,L;H^1_0(\O';\RR^3))$, we find 
  
   \begin{equation}
\nonumber
\begin{aligned}
  	 \int_{\O} \hskip-0,1cm \lp\hskip-0,1cm \lp\bfv^\delta\rp^\star\rp'\hskip-0,15cm\cdot\hskip-0,07cm\bfdiv \bfPsi   dm\otimes\L^2 \hskip-0,1cm
 	 & = \hskip-0,1cm\int_{\O}  \hskip-0,1cm\lp\lp\bfv^\star\rp'\rp^\delta\hskip-0,2cm\cdot\bfdiv \bfPsi   dm\otimes\L^2 
   = \hskip-0,1cm\int_{\O} \hskip-0,1cm \lp\bfv^\star \rp'\hskip-0,1cm\cdot\bfdiv \lp \bfPsi^\delta \rp   dm\otimes\L^2 
	  \\& =\hskip-0,1cm- \hskip-0,1cm\int_{\O}  \hskip-0,1cm\bfe_{x'}\lp\bfv^\star \rp\hskip-0,1cm  :   \hskip-0,1cm\bfPsi^\delta     dm\otimes\L^2 =-\hskip-0,1cm \int_{\O} \hskip-0,1cm \lp \bfe_{x'}\lp\bfv^\star \rp\rp^\delta\hskip-0,1cm :   \bfPsi      dm\otimes\L^2 ,
	   \end{aligned}
  \end{equation} 
 
yielding  \eqref{vdeltaL2m}.
Assertion \eqref{vdeltaBDnum} is a consequence of \eqref{defBDnum}, \eqref{vdeltaBD},  \eqref{vdeltaL2m}, and \eqref{fdeltatof} 
 applied  for $f\in \la  \tfrac{\bfE\bfv}{\nu\otimes\L^2}, \bfe_{x'}(\bfv), \bfv \ra$ and $\theta\in \{\nu, m\}$. Let us fix $x\in \O$: by \eqref{Elem}, \eqref{vd+=v+d}  and Green's formula,
 denoting by $\bfgamma$ the trace application on $BD((x_1,x_1+\kappa)\times\O')$,  we have
 
   \begin{equation}
\nonumber
\begin{aligned}
  &	\lb  (\bfv^\delta)^- (x_1+\kappa,x') -  (\bfv^\delta)^+(x)\rb \le \sqrt2 \lb \lp  (\bfv^-)^\delta (x_1+\kappa,x') -  (\bfv^+)^\delta (x)\rp\odot\bfe_1\rb
 \\&=\sqrt2 \lb \int_{\partial\lp(x_1,x_1+\kappa)\times\O'\rp } \hskip-0 cm  \eta_\delta(x'-y') \bfgamma(\bfv)(s_1,y') \odot\bfn d\H^2(s_1,y') \rb
 \\&=\sqrt2 \lb \int_{ (x_1,x_1+\kappa)\times\O'  } \hskip-1,5cm  \eta_\delta(x'-y') d\bfE\bfv (s_1,y') 
 +  \int_{ (x_1,x_1+\kappa)\times\O'  } \hskip-1,5cm  \bfv\odot \bfnabla_{x'}\eta_\delta(x'-y') d s_1d y'   \rb
  \\&\le C \lp  |\bfE\bfv| \lp(x_1,x_1+\kappa)\times\O' \rp + \int_{(x_1,x_1+\kappa)\times\O' }|\bfv| dx\rp,
	 	   \end{aligned}
  \end{equation} 

therefore $\limsup_{\kappa\to0^+} \lb  (\bfv^\delta)^- (x_1+\kappa,x') -  (\bfv^\delta)^+(x)\rb =0$. We   likewise find that  $\limsup_{\kappa\to0^+} \Big|  (\bfv^\delta)^+ (x_1-\kappa,x') -  (\bfv^\delta)^-(x)\Big| =0$.
The second line of \eqref{610} is obtained by using the second line of 
  \eqref{fdeltareg} and  by  substituting the partial derivatives  $\tfrac{\partial \eta_\delta}{\partial x_\a}$ for $\eta_\delta$ in the above computations.
To prove   \eqref{varphi+=},  let us  fix $(x_1,x')\in \Omega$, $\kappa>0$: by \eqref{vdeltaBD} and Green's formula,  we have
 
\begin{equation}
\nonumber
\begin{aligned}
 \int_{_{(0,x_1+\kappa)}} &\hskip-0,5cm \tfrac{E_{11} \bfv^\delta}{\nu\otimes\L^2}(s_1,x') d\nu(s_1)
   =    \int_{_{(0,x_1+\kappa)\times\Omega'}} \hskip-0,5cm \tfrac{E_{11} \bfv }{\nu\otimes\L^2}(s_1,y') \eta_\delta(x'-y')  d\nu\otimes\L^2(s_1,y')
  \\&
   =    \int_{(0,x_1+\kappa)\times\Omega'}   \eta_\delta(x'-y')  d E_{11} \bfv (s_1,y')
   \\& = \int_{\Sigma_{x_1+\kappa}}\eta_\delta(x'-y') v_1^-(s_1,y') d\H^2(s_1,y') = \lp v_1^-\rp^\delta(x_1+\kappa ,y') .
\end{aligned}
\end{equation}

Likewise, the following holds  for $\b\in\{2,3\}$, 
\begin{equation}
\nonumber
 \begin{aligned}
& \int_{(0,x_1+\kappa)}\tfrac{E_{\b\b} \bfv^\delta}{\nu\otimes\L^2}(s_1,x') d\nu(s_1)
  =    \     \int_{(0,x_1+\kappa)\times\Omega'}   \eta_\delta(x'-y')  d E_{\b\b} \bfv (s_1,y')
  \\& =  -\int_{(0,x_1+\kappa)\times\Omega'} \frac{\partial}{\partial y_\b} \lp \eta_\delta(x'-y') \rp v_\b (s_1,y') d\L^3(s_1,y')
   \\& =  \int_0^{x_1+\kappa} \lp  \frac{\partial}{\partial x_\b} \int_{ \Omega'}   \eta_\delta(x'-y')   v_\b (s_1,y') dy'\rp ds_1
   = \int_0^{x_1+\kappa} \frac{\partial v_\b^\delta}{\partial x_\b}(s_1,x') ds_1. 
\end{aligned}
\end{equation}

Passing to the limit as $\kappa\to 0^+$, taking \eqref{vd+=v+d} into account, we infer 

\begin{equation}
\label{eq22} 
\begin{aligned}
 &\int_{_{(0,x_1 ]}}   \tfrac{E_{11} \bfv^\delta}{\nu\otimes\L^2}(s_1,x') d\nu(s_1)
   = \lp v_1^\delta\rp^+ (x_1  ,y') ,
\\
& \int_{(0,x_1]}\tfrac{E_{\b\b} \bfv^\delta}{\nu\otimes\L^2}(s_1,x') d\nu(s_1)
    = \int_0^{x_1 } \frac{\partial v_\b^\delta}{\partial x_\b}(s_1,x') ds_1,
\end{aligned}
\end{equation}

which, joined with   \eqref{defsigmanu}, yields the first  equation in  \eqref{varphi+=}. 
Similarly, we have 

\begin{equation}
\label{eq11}
\begin{aligned}
&  
   \int_{(0,x_1+\kappa)} \hskip-0,7cm 2 \tfrac{E_{1\a} \bfv^\delta}{\nu\otimes\L^2} (s_1,x') d\nu(s_1)
 =    \int_{(0,x_1+\kappa)\times\Omega'}
 2   \eta_\delta(x'-y')  d E_{1\a} \bfv (s_1,y')
    \\& =\int_{\Sigma_{x_1+\kappa}}\hskip-0,7 cm \eta_\delta(x'-y')  v_\a^-(s_1,y') d\H^2(s_1,y')
    +\int_{(0,x_1+\kappa)\times\Omega'} \hskip-1,4cm v_1(s_1,y') \frac{\partial \eta_\delta}{\partial x_\a} (x'-y') ds_1dy'
  \\&= \lp v^\delta_\a\rp^-(x_1+\kappa,x') + \int_0^{x_1+\kappa} \frac{\partial v_1^\delta}{\partial x_\a} (s_1,x') ds_1,
\end{aligned}
\end{equation}

yields by the same argument the second  equation in  \eqref{varphi+=}.
 \end{proof}

%
%
 
   \begin{proposition}
  	\label{propmol2}
For all  $\bfv\in BD^{\nu,m}_0(\Omega)$ and   $\delta>0$, the following holds for some constant $C$ independent of $\delta$

  \begin{equation}
 \label{Xiborne}
 \begin{aligned}
&\int_\Omega \lb\tfrac{\bfE \bfv^\delta}{\nu\otimes\L^2} \rb^2 
 d\nu\otimes\L^2\le \int_\Omega \lb\tfrac{\bfE \bfv }{\nu\otimes\L^2} \rb^2 d\nu\otimes\L^2<\infty,
\\&\int_\Omega \Big\vert
\tfrac{\partial }{\partial x_\alpha}\tfrac{\bfE \bfv^\delta}{\nu\otimes\L^2}   \Big\vert^2 d\nu\otimes\L^2\le 
\frac{C}{\delta} \int_\Omega \lb\tfrac{\bfE \bfv }{\nu\otimes\L^2} \rb^2 d\nu\otimes\L^2<\infty, 
\end{aligned}
\end{equation}

\begin{equation}
\bfv^\delta, \ \frac{\partial \bfv^\delta}{\partial x_\a},  \ \frac{\partial^2 \bfv^\delta}{\partial x_\a\partial x_\b}
\in L^2(\Omega;\RR^3), \qquad  \forall\  a,\b \in \{2,3\} .
\label{varphiborne}
\end{equation}

  \end{proposition}
  

    \begin{proof}

  Assertion \eqref{Xiborne} follows from  \eqref{fdeltainLp}  and \eqref{vdeltaBD}.
By Lemma \ref{lemL1nu}, the Lebesgue measure on $\O$ is absolutely  continuous with respect to   $m\otimes\L^2$, thus,  by   
\eqref{fistar=fiae} and 
\eqref{varphi+=varphimae},

\begin{equation}\label{fi+=fi}
 \lp \bfv^\delta\rp^+  = \lp \bfv^\delta\rp^- =\lp \bfv^\delta\rp^\star = \bfv^\delta   \quad \L^3\hbox{-a.e. in  }  \Omega.
\end{equation}

By   \eqref{eq22}, \eqref{Xiborne}, \eqref{fi+=fi}, Cauchy-Schwarz inequality and Fubini Theorem,  we have

\begin{equation}
\nonumber
\begin{aligned}
   \int_{\Omega} \lvert v^\delta_1 \rvert^2  dx
   &= \int_{\Omega} \lvert (v^\delta_1)^+    \rvert^2  dx
   = \int_\Omega \lb  \int_{(0,x_1]} \tfrac{E_{11}\bfv^\delta}{\nu\otimes\L^2}(s_1,x') \ d\nu(s_1)\rb^2 dx
   \\& \le C \int_\Omega \lvert \tfrac{E_{11}\bfv^\delta}{\nu\otimes\L^2} \rvert^2 d\nu\otimes\L^2  \le C \int_\Omega \lvert \tfrac{E_{11}\bfv }{\nu\otimes\L^2}\rvert^2 d\nu\otimes\L^2 <\infty,
    \end{aligned} 
\end{equation}

yielding, by \eqref{fdeltainLp},   

\begin{equation}
\nonumber
   \int_{\Omega} \lb \frac{\partial  v^\delta_1}{\partial x_\a}\rb^2  dx \le \frac{C}{\delta}  \int_{\Omega} \lb  v^\delta_1 \rb^2  dx \le \frac{C}{\delta} \int_\Omega \lvert  \tfrac{E_{11}\bfv }{\nu\otimes\L^2} \rvert^2 d\nu\otimes\L^2 <\infty.
\end{equation}
  
We deduce from \eqref{vdeltaBD},  \eqref{eq11}, \eqref{fi+=fi} and  the last   inequalities  that, for $\a\in \{2,3\}$,   
  \begin{equation}
  \nonumber
  \begin{aligned}
   \int_{\Omega}& \lvert v^\delta_\a \rvert^2  dx  
   \le  C  \hskip-0,1cm \int_{\Omega}\lb \int_{(0,x_1]} \hskip-0 cm  \tfrac{E_{1\a}\bfv^\delta}{\nu\otimes\L^2}(s_1,x')  d\nu(s_1) \rb^2  \hskip-0,1cm dx
   + C   \hskip-0,1cm\int_{\Omega}\lb \int_{0}^{x_1} \hskip-0,1cm \pd{v^\delta_1}{x_\alpha}(s_1,x') \ ds_1\rb^2 \hskip-0,1cmdx
   \\&\le  C  \int_\Omega \lvert  \tfrac{E_{1\a}\bfv }{\nu\otimes\L^2}\rvert^2 d\nu\otimes\L^2
   + C   \int_{\Omega} \lb \frac{\partial  v^\delta_1}{\partial x_\a}\rb^2  dx
 \le \frac{C}{\delta} \int_\Omega \lvert  \tfrac{\bfE\bfv}{\nu\otimes\L^2}\rvert^2 d\nu\otimes\L^2 <\infty,
   \end{aligned}
\end{equation}

and then from \eqref{fdeltainLp} that, for $\a,\b\in \{2,3\}$, 

 \begin{equation}
  \nonumber
  \begin{aligned}
&   \int_{\Omega}  \lb \frac{\partial \bfv^\delta}{\partial x_\a}\rb^2  dx  \le \frac{C}{\delta}  \int_{\Omega} \lvert \bfv^\delta_\a \rvert^2  dx
 \le \frac{C}{\delta^2} \int_\Omega \lvert  \tfrac{\bfE\bfv}{\nu\otimes\L^2}\rvert^2 d\nu\otimes\L^2 <\infty,
 \\&
    \int_{\Omega}  \lb \frac{\partial^2 \bfv^\delta}{\partial x_\a\partial x_\b }\rb^2  dx 
     \le \frac{C}{\delta} \int_\O  \lb \frac{\partial \bfv^\delta}{\partial x_\a}\rb^2  dx 
 \le \frac{C}{\delta^3} \int_\Omega \lvert  \tfrac{\bfE\bfv}{\nu\otimes\L^2} \rvert^2 d\nu\otimes\L^2 <\infty.
   \end{aligned}
\end{equation}


Assertion \eqref{varphiborne} is proved.   	   \end{proof}
}

    \section{Proof of Theorem \ref{th}}\label{secproofth}

The proof of Theorem \ref{th} rests on the choice of an appropriate sequence of
 test fields
$(\bfvarphi_\e)$, which   will be constructed from  an arbitrarily chosen partially mollified element    of $BD^{\nu,m}_0(\Omega)$, that is a  field $\bfvarphi$ 
  of the type 

\begin{equation}
\label{defvarphi}
\bfvarphi= \bfv^\delta, \quad \bfv\in BD_0^{\nu,m}(\Omega), \quad \delta>0. 
\end{equation}

 Let us briefly outline  our approach.       In the spirit of  Tartar's   method  \cite{Ta},     we  will   multiply  (\ref{Pe}) by $\bfvarphi_\e$ and 
 integrate by parts  to obtain  

 \begin{equation}
  \label{IPintro}
  \int_\Omega \bfsigma_\e (\bfu_\ep) \cdot \bfe(\bfvarphi_\e) \ dx = \int_\Omega \bff \cdot \bfvarphi_\e dx.
  \end{equation}


By passing to the limit     as   $\e\to 0$ in accordance with the convergences established in propositions \ref{propaprioriu} and \ref{propvarphie}, we    will find   $a(\bfu,\bfv^\delta) = \int_\Omega \bff \cdot \bfv^\delta dx$, where $a(\cdot,\cdot)$ is the  
 symmetric bilinear form on $BD^{\nu,m} (\Omega)$ defined by \eqref{defa}.  
Then, sending  $\delta$ to $0$,   we will   infer from   Proposition \ref{propmol}  that  $a(\bfu,\bfv) = \int_\Omega \bff \cdot \bfv dx$.
 From  Proposition \ref{propaprioriu}, we will deduce that $\bfu$ belongs to $BD_0^{\nu,m}(\Omega)$,  hence   is a solution to  \eqref{Peff}.  Next, we will prove that    $BD^{\nu,m}_0(\Omega)$ is a Hilbert space and 
  $a(\cdot,\cdot)$  is coercive and continuous on  it,  hence the solution to  \eqref{Peff}  is unique and 
 the convergences established in  Proposition \ref{propaprioriu}  for subsequences,  hold for the complete sequences. 
 
\quad The sequence  $(\bfvarphi_\e)$  will be deduced from a family of sequences $((\bfvarphi_\e^k)_\e)_{k\in \NN}$  by    a diagonalization argument. 
 Given $k\in \NN$, the construction of   $(\bfvarphi_\e^k)_\e$ is based on the choice of an appropriate 
finite partition $( I_j^k)_{ j\in \{1,\ldots,n_k\}}$
  of $(0,L]$ defined as follows: 
  since    the set of the atoms of the measures $\nu$ and $m$ is  at most countable, we   can   fix  a sequence 
  $(A_k)_{k\in \NN }$   of finite subsets  of $[0,L]$  satisfying 
  
  \begin{equation}
  \label{defAk}
  \la \begin{aligned}
  & A_k= \la t^k_0, \, t^k_1,\ldots, t^k_{n_k}\ra, \quad A_k\subset A_{k+1} \quad  \forall k\in \NN, 
  \\& 0  = t^k_0  < t^k_1 < t^k_2 < \ldots <t^k_{n_{k-1} }<  t^k_{n_k} = L,
  \\& \nu\lp\la t_j^k\ra\rp= m\lp\la t_j^k\ra\rp=0 \quad \forall  k\in \NN, \quad \forall j \in \{0,\ldots,n_k\},  
  \\& \lim_{k\to\infty}\sup_{j\in \{1,\ldots,n_k\}}\lb t^k_j-t^k_{j-1}\rb =0.
  \end{aligned}\rpt
  \end{equation}

 Setting 
    \begin{equation}
  \label{defIjk}
  \begin{aligned}
  \\& I^k_j:=\lp t^k_{j -1}, t^k_{j }\right] & & \forall \ k\in \NN, \ \forall j\in \{1,\ldots,n_k\},
  \end{aligned} 
  \end{equation}

 we  introduce the function $\phi^k_\ep:(0,L) \to \RR$ defined by 

  
  \begin{equation}
  \label{defphiek}
  \begin{aligned}
 \phi^k_\ep(x_1) : =   \sum_{j=1}^{n_k} \frac{\nu_\e((t_{j-1}^k, x_1))}{\nu_\e(I^k_{j})} \mathds{1}_{I_j^k}(x_1).
  \end{aligned}
  \end{equation} 
  

Note that the restriction of   $\phi^k_\ep$   to each $I_j^k$ is absolutely continuous, and 
 
 \begin{equation}
\label{phidphi}
 \begin{aligned}
 &\frac{d \phi^k_\e}{d x_1}(x_1)  =\frac{\mu^{-1}_\e(x_1)}{\nu_\ep(I^k_j)}\quad \text{ in }\ \mathring{I_j^k}; \quad  0\le  \phi^k_\e    \le 1\ \hbox{ in } \  (0,L),
 \\&  \phi^k_\e((t_j^k)^-)=1\ \hbox{ and }  \  \phi^k_\e((t_{j-1}^k)^+)=0\quad \forall j\in \{1,\ldots,n_k\}.
  \end{aligned}\end{equation}

   For all $j\in \{0,\ldots,n_k\}$, 
$x\in I_j^k\times \Omega'$, $ \alpha \in \{2,3\}$,  we set (see \eqref{defsigmanu})
 
   \begin{equation}
  \label{defvarphiek}
  \begin{aligned}
 \hspace{-0,1cm} \varphi^k_{\ep 1} (x) &:= \frac{\phi^k_\ep(x_1)}{l+2}   \int_{I^k_{j}} 
 \sigma^\nu_{11}(\bfvarphi) (s_1,x')
d\nu(s_1) 
 \\
  & \hspace{3cm}   - \frac{l}{l+2}\sum_{\a=2}^3  \int_{t^k_{j-1}}^{x_1} \pd{\varphi_\a}{x_\a}(s_1,x')   ds_1  + \varphi^+_1 (t^k_{j-1},x'), 
  \\\hspace{-0,1cm}
  \varphi^k_{\ep \alpha} (x) &\! : =   \!  \phi^k_\ep(x_1) \! \hspace{-0,1cm}\int_{I^k_{j}} \! \hspace{-0,1cm}
    \sigma^\nu_{1\a}(\bfvarphi) (s_1,x')
 d\nu(s_1)    
  - \hspace{-0,1cm}\int_{t^k_{j-1}}^{x_1} \hspace{-0,1cm}\pd{\varphi_1}{x_\alpha}(s_1,x')   ds_1 \!  +\!  \varphi^+_\alpha (t^k_{j-1},x'). 
  \end{aligned}
  \end{equation}
   
%
%

 The  sequence of test fields    $(\bfvarphi_\e)$  is  determined by  the next proposition.


\begin{proposition}\label{propvarphie}  Let $\bfv\in BD_0^{\nu,m}(\Omega)$, $\delta>0$, and $\bfvarphi$, $\bfvarphi_\e^k$   respectively given by \eqref{defvarphi}, \eqref{defvarphiek}.
There exists an increasing sequence $(k_\e)$  of  positive integers  converging to $\infty$  such 
that  $\bfvarphi_\e$ defined by 

\begin{equation} 
  	\label{defvarphie}
  	\begin{aligned}
  	& \bfvarphi_\e:= \bfvarphi_\e^{k_\e},
  	\end{aligned}
  	\end{equation}

 strongly converges to $\bfvarphi$ in $L^1(\O;\RR^3)$ and satisfies the assumptions \eqref{supFefiefini} and \eqref{supfiemefini} of Proposition \ref{propapriori}. 
In particular, the convergences  and relations \eqref{cvfi},  \eqref{cvfinuenu} and \eqref{cvfimem} are satisfied.
 In addition, the following strong convergences in the sense of  \eqref{fetetoftstrong} hold:

		\begin{align}
 	\label{cvfistrong} \bfsigma_{\e}(\bfvarphi_\e)\bfe_1
    \buildrel {\nu_\e\otimes\L^2, \nu\otimes\L^2}\over \rrrrightarrow \bfsigma^\nu(\bfvarphi)\bfe_1,
\quad
   \bfe_{x'}(\bfvarphi_\e)  
    \buildrel {m_\e\otimes\L^2, m\otimes\L^2}\over \rrrrightarrow    \bfe_{x'}((\bfvarphi^\star)'),
  	\end{align}
      
   where $\bfsigma^\nu$ is given by \eqref{defsigmanu}.

 \end{proposition}

\mn Proposition \ref{propvarphie} will  be proved   in Section \ref{secproofpropvarphike}. 
The next  step 
consists in passing to the limit as $\e\to 0$ in \eqref{IPintro}. 
  Expressing in  \eqref{IPintro}, for $\bfg\in \{\bfu_\e, \bfvarphi_\e\}$,     the scalar fields   $e_{11}(\bfg)$, $\sigma_{\e22}(\bfg)$,  $\sigma_{\e33}(\bfg)$      in terms of  the components of 
 $\bfsigma_\e(\bfg)\bfe_1$   and $\bfe_{x'}(\bfg)$    (the details of this computation are situated at the end of the section), leads to the following equation:

  \begin{equation}
  \label{rearrangement}
  \begin{aligned} 
&
  \int_{\Omega} \tfrac{1}{l + 2}\sigma_{\e11}(\bfu_\e)\sigma_{\e11}(\bfvarphi_\e)
  + \sum_{\a=2}^3   \sigma_{\e1\a}(\bfu_\e)\sigma_{\e1\a}(\bfvarphi_\e)
 \ d\nu_\e\otimes\L^2  
\\&   + \int_\Omega  4 e_{23}(\bfu_\e) e_{23}(\bfvarphi_\e) +  \tfrac{4(l+1)}{l+2} \sum_{\a=2}^{3} e_{\a\a}(\bfu_\e) e_{\a\a}(\bfvarphi_\e)   \, dm_\e\otimes\L^2
\\& \quad +\int_\Omega \tfrac{2l}{l+2}  \Big( e_{22}(\bfu_\e) e_{33}(\bfvarphi_\e) + e_{33}(\bfu_\e) e_{22}(\bfvarphi_\e)   \Big) dm_\e\otimes\L^2=\int_\Omega \bff\cdot \bfvarphi_\e dx.
 \end{aligned} \end{equation}

By \eqref{cvu}, the {\color{black} next}   weak convergences    in the sense of   \eqref{fetetoft}  hold 

	\begin{equation}
  	\label{cvuweak}
	\hskip-0,5cm \begin{aligned}
  &   	\bfsigma_{\e}(\bfu_\e)\bfe_1
 \buildrel {\nu_\e\otimes\L^2, \nu\otimes\L^2}\over \rrrrightharpoonup
  \bfsigma^\nu(\bfu)\bfe_1,
\quad
   \bfe_{x'}(\bfu_\e')  
    \buildrel {m_\e\otimes\L^2, m\otimes\L^2}\over  \rrrrightharpoonup     \bfe_{x'}((\bfu^\star)').
       \end{aligned}	\end{equation}
 
 By passing to the limit as $\e\to 0$ in \eqref{rearrangement}, 
 by virtue of \eqref{cvfistrong} ,  \eqref{cvuweak} and  Lemma \ref{lemfeps} (iii), 
 we obtain

 \begin{equation}
\label{abrut}
\begin{aligned}
&\int_{\Omega} \tfrac{1}{l+2} \sigma^\nu_{11}(\bfu)  \sigma_{11}^\nu(\bfvarphi) 
+   \sum_{\a=2}^{3} 
\sigma^\nu_{1\a}(\bfu)    \sigma_{1\a}^\nu(\bfvarphi)  \   d\nu\otimes\L^2 
\\&
+  \int_{\Omega} 4\, e_{23}(\bfu^\star) e_{23}(\bfvarphi^\star)  +  \tfrac{4(l+1)}{l+2} \sum_{\a=2}^3 e_{\a\a}(\bfu^\star) e_{\a\a}(\bfvarphi^\star)   dm\otimes\L^2
\\&+\int_{\Omega}   \tfrac{2l}{l+2}\big( e_{22}(\bfu^\star)e_{33}(\bfvarphi^\star)+ e_{33}(\bfu^\star)e_{22}(\bfvarphi^\star) \big)  dm\otimes\L^2
=\int_\Omega \bff\cdot\bfvarphi \ dx.
 \end{aligned}
\end{equation}

An elementary computation yields 

 \begin{equation}
\label{reorg}
\begin{aligned}
&\int_{\Omega}\hskip-0,1cm  \tfrac{1}{l+2} \sigma^\nu_{11}(\bfu)  \sigma_{11}^\nu(\bfvarphi) 
\hskip-0,1cm +\hskip-0,1cm    \sum_{\a=2}^{3} 
\sigma^\nu_{1\a}(\bfu)    \sigma_{1\a}^\nu(\bfvarphi)     d\nu\hskip-0,1cm \otimes\hskip-0,1cm \L^2 \hskip-0,1cm 
=\hskip-0,1cm  \int_\O\hskip-0,1cm  \bfa^\bot \tfrac{\bfE\bfu}{\nu\otimes\L^2}: \tfrac{\bfE\bfvarphi}{\nu\otimes\L^2} d\nu\hskip-0,1cm \otimes\hskip-0,1cm \L^2,
\\
&  \int_{\Omega}\hskip-0,1cm  4\, e_{23}(\bfu^\star) e_{23}(\bfvarphi^\star)  +\hskip-0,1cm \int_{\Omega} \hskip-0,1cm   \tfrac{2l}{l+2}\big( e_{22}(\bfu^\star)e_{33}(\bfvarphi^\star)\hskip-0,1cm + e_{33}(\bfu^\star)e_{22}(\bfvarphi^\star) \big)  dm\hskip-0,1cm \otimes\hskip-0,1cm \L^2
\\&\hskip 1cm +  \tfrac{4(l+1)}{l+2}\hskip-0,1cm  \sum_{\a=2}^3 e_{\a\a}(\bfu^\star) e_{\a\a}(\bfvarphi^\star)   dm\otimes\L^2\hskip-0,1cm 
=\hskip-0,1cm \int_\Omega\hskip-0,1cm  \bfa^\Vert\bfe_{x'}(\bfu^\star): \bfe_{x'}(\bfvarphi^\star) \ dm\otimes\L^2,
 \end{aligned}
\end{equation}

where $\bfa^\bot$ and $\bfa^\Vert$ are given by \eqref{defabotaVert}. We infer  from \eqref{abrut} and \eqref{reorg} that 

\begin{equation}
\nonumber
a(\bfu,\bfvarphi)=\int_\Omega \bff\cdot\bfvarphi \ dx, 
\end{equation}

where $a(\cdot,\cdot)$ is the  continuous   symmetric  bilinear  form on $BD^{\nu,m}(\Omega)$ defined  by  \eqref{defa}. 
Substituting $\bfv^\delta$ for $\bfvarphi$ (see  \eqref{defvarphi}) and letting $\delta$ converge to $0$,  we deduce  from the strong convergence  in $BD^{\nu,m}(\O)$ of 
$\bfv^\delta$ to $\bfv$   stated in   \eqref{vdeltaBDnum}  that 
 
\begin{equation}
\nonumber
a(\bfu,\bfv )=\int_\Omega \bff\cdot\bfv \ dx \quad \forall \bfv\in BD^{\nu,m}_0(\Omega).  
\end{equation}

Since, by Proposition \ref{propaprioriu}, the field $\bfu$ belongs to $BD^{\nu,m}_0(\Omega)$, we conclude that $\bfu$ is a solution to \eqref{Peff}.

  \quad 
 Let us prove that $BD_0^{\nu,m}(\Omega)$  is a Hilbert space.
By   the  Poincar\'e inequality in $\la \bfv\in BD(\Omega), \ \bfv=0 \ \hbox{ on } \ \partial \Omega\ra$ (see \cite[Remark 2.5 (ii) p. 156]{Te}),  we have
 
\begin{equation}
\label{norm11}
\begin{aligned}
\int_\Omega |\bfv| dx &\le C \int_\Omega d |\bfE\bfv|=  C \int_\Omega |\tfrac{\bfE\bfv}{\nu\otimes\L^2}| d\nu\otimes\L^2
\\&\le C
\lp  \int_{\Omega}  | \tfrac{\bfE\bfv}{\nu\otimes\L^2}   |^2 \ d\nu\otimes\L^2\rp^{\frac{1}{2}} \le C||\bfv||_{BD_0^{\nu,m}(\Omega)}\quad \forall\ \bfv\in BD_0^{\nu,m}(\Omega),
\end{aligned}
\end{equation}

hence   the semi-norm $||.||_{BD_0^{\nu,m}(\Omega)}$ defined by \eqref{defnormBDnum0}  is a norm on $BD^{\nu,m}_0(\Omega)$.  On the other hand, 
Fubini's Theorem and 
 Korn's inequality in $H^1_0(\Omega';\RR^2)$ imply

\begin{equation}
\label{norm22}
\begin{aligned}
\int_\Omega |(\bfv')^\star |^2 dm\otimes \L^2 &
= \int_0^L dm(x_1) \int_{\Omega'}  |(\bfv')^\star |^2 dx' 
\\&\hskip-2,5cm  \le C \int_0^L dm(x_1) \int_{\Omega'}  |\bfe_{x'}(\bfv^\star) |^2 dx'   \le C ||\bfv||_{BD_0^{\nu,m}(\Omega)}^2\quad \forall\ \bfv\in BD_0^{\nu,m}(\Omega).
\end{aligned}
\end{equation}

 Let  $(\bfv_n)$ be a Cauchy sequence in $BD_0^{\nu,m}(\Omega)$. By  \eqref{norm11} and \eqref{norm22},  the sequences
$(\bfv_n)$,  $((\bfv_n')^\star)$, $(\tfrac{\bfE\bfv_n}{\nu\otimes\L^2})$
are Cauchy sequences in $BD(\Omega)$,  $L^2_{m}(0,L; H^1_0(\Omega;\RR^3)),$
 $L^2_{\nu\otimes\L^2}(\Omega;\SS^3)$  respectively,
hence the following convergences hold

 \begin{equation}
\label{cvCauchy}
\begin{aligned}
&\bfv_n   \to \bfv & & \hbox{strongly in}  \quad BD(\O), 
\\& (\bfv_n')^\star \to \bfw'  & & \hbox{strongly in}  \quad L^2_{m}(0,L; H^1_0(\Omega';\RR^3)),
\\&    \tfrac{\bfE\bfv_n}{\nu\otimes\L^2}    \to \bfXi    & & \hbox{strongly in}  \quad  L^2_{\nu\otimes\L^2}(\Omega;\SS^3),
\end{aligned}
\end{equation}

 for some 
$\bfv$,  $\bfw'$, $\bfXi$. We prove below that 
 
\begin{align}
 &\bfE(\bfv) \ll \nu\otimes \L^2, \quad   \bfXi= \tfrac{\bfE\bfv}{\nu\otimes\L^2}, \quad \bfv=0 \ \hbox{ on } \ \partial \O, 
 \label{vBDnu0}
\\& \bfw'=  (\bfv')^\star  \quad
 m\otimes\L^2 
 \hbox{-a.e..} 
 \label{vstarH1}
\end{align}

It follows from \eqref{cvCauchy}-\eqref{vstarH1} that   $\bfv\in BD^{\nu,m}_0(\Omega)$ and  $(\bfv_n)$ strongly converges to 
$\bfv$ in $BD^{\nu,m}_0(\Omega)$, hence 
$BD^{\nu,m}_0(\Omega)$
 is a Hilbert space. 
   The proof of Theorem \ref{th} is achieved provided we establish that  
 the  form $a(\cdot,\cdot)$ is continous and coercive on $BD^{\nu,m}_0(\Omega)$. The  continuity is straightforward. The coercivity of $a(\cdot,\cdot)$ 
results from Lemma \ref{lemacoercive} stated below.  \qed 

 {\bf Proof of \eqref{vBDnu0}.}   
As $\bfv_n=0$ on $\partial\O$, 
by 
   \eqref{cvCauchy} and Green's formula  we have,   for $\bfPsi \in C^1(\ov\Omega;\SS^3)$,

 \begin{equation} 
\nonumber 
   \begin{aligned}
    \int_\O \bfv \cdot\bfdiv\bfPsi dx & = \lim_{n\to\infty} \int_\O \bfv_n\cdot\bfdiv\bfPsi dx = - \lim_{n\to\infty} \int_\O  \bfPsi d \bfE\bfv_n
 \\&   =  -\lim_{n\to\infty}  \int_\O \tfrac{\bfE\bfv_n}{\nu\otimes\L^2}:\bfpsi d\nu\otimes\L^2= -\int_\O \bfXi :\bfpsi d\nu\otimes\L^2. 
\end{aligned}
\end{equation}

We deduce from Green's formula that 
 \begin{equation} 
\nonumber 
   \begin{aligned}
   - \int_\O \bfPsi : d\bfE(\bfv)  + \int_{\partial \O} \bfv\odot\bfn:\bfPsi d\H^2 = -\int_\O \bfXi :\bfpsi d\nu\otimes\L^2. 
\end{aligned}
\end{equation}

By the arbitrary choice of $\bfpsi$, we infer \eqref{vBDnu0}. 
 \qed


 {\bf Proof of \eqref{vstarH1}.} 
  By \eqref{cvCauchy}, $\lim_{n\to+\infty} \int_\O |(\bfv_n')^\star-\bfw'|^2 dm\otimes \L^2=0$,  hence there exists a $m$-negligible subset 
$N$ of $(0,L)$  such that 

\begin{align}
&  \lim_{n\to+\infty} \int_{\Sigma_{x_1}} |(\bfv_n')^\star-\bfw'|^2 d\H^2=0
\qquad \forall x_1\in (0,L)\setminus N.
\label{Fatou}
\end{align}

  On the other hand, since $(\bfv_n)$ strongly converges to $\bfv$ in $BD(\O)$, 
the traces $\bfgamma_{\Sigma_{x_1}}^\pm(\bfv_n)$
 on both side of  $\Sigma_{x_1}$ strongly converges to  $\bfgamma_{\Sigma_{x_1}}^\pm(\bfv)$  in $L^1_{\H^2}(\Sigma_{x_1})$ for all $x_1\in (0,L)$. 
By  \eqref{traceball}, \eqref{varphi+=varphimae}, and  \eqref{vBDnu0},
 $\bfv^\star(x_1,.)=\bfgamma_{\Sigma_{x_1}}^+(\bfv)=\bfgamma_{\Sigma_{x_1}}^-(\bfv)$ $\H^2$-a.e. on $\Sigma_{x_1}$ for $m$-a.e. $x_1\in(0,L)$. 
Accordingly, there exists a $m$-negligible subset 
$N_1$ of $(0,L)$  such that 
  
   \begin{align}
&   \lim_{n\to+\infty} \int_{\Sigma_{x_1}} |(\bfv_n)^\star-\bfv^\star|  d\H^2 =0 \qquad \forall x_1\in (0,L)\setminus N_1.
\label{Fatou2}
\end{align}

Let us fix $x_1\in (0,L)\setminus (N\cup N_1)$.  By \eqref{Fatou} there exists a subsequence of  $(\bfv_n')^\star$  converging $\H^2$-a.e. on $\Sigma_{x_1}$
to $\bfw' $.  By \eqref{Fatou2}, there exists a further subsequence  converging $\H^2$-a.e. on $\Sigma_{x_1}$
to $(\bfv')^\star$.  Hence $\bfw'=(\bfv')^\star$ $\H^2$-a.e. on $\Sigma_{x_1}$ for  $m$-a.e. $x_1\in (0,L)$. 
Setting  $A:= \{x\in\O, \ \bfw'(x)\not =(\bfv')^\star(x)\}$, $A_{x_1}:= A\cap \Sigma_{x_1}$,   we infer that $\H^2(A_{x_1})=0$ for all $x_1\in (0,L)\setminus (N\cup N_1)$.
It then follows from Fubini's theorem that  $ m\otimes\L^2(A)= \int_{(0,L)} \H^2(A_{x_1}) dm(x_1)=0$.
  \qed


\begin{lemma}\label{lemacoercive} For all $\bfv \in BD^{\nu,m}_0(\Omega)$, $ \a,\b\in \{2,3\}$, we have 

\begin{equation}
  	\label{ubyxi}
	\begin{aligned}
  	\int_{\Omega} \lb \tfrac{E_{\a\b}\bfv}{\nu\otimes\L^2} \rb^2  \ d\nu\otimes\L^2 \le \int_{\Omega} &\lb e_{\a\b}((\bfv^\star)') \rb^2  \ dm\otimes\L^2.
  \end{aligned} 	\end{equation}
 
 \end{lemma}
 
 \begin{proof}
Let  $\bfv\in BD^{\nu,m}_0(\Omega)$, 
$\delta>0$, and $ \bfvarphi_\e$  defined by \eqref{defvarphi},  \eqref{defvarphie}. 
By Proposition \ref{propvarphie}, the convergence \eqref{cvfinuenu} holds, hence 
by  Lemma \ref{lemfeps} (ii), we have 
   for  $\a,\b \in \{2,3\}$,  
  	$$
   \int_{\Omega}  \lb  \tfrac{E_{\a\b}\bfvarphi}{\nu\otimes\L^2}\rb^2 \ d\nu\otimes\L^2 \le \liminf_{\ep\rightarrow 0}\int_{\Omega} \mu_\ep \lb \bfe_{\a\b}(\bfvarphi_\ep) \rb^2 \ dx .
  	$$ 
As, on the other hand, by \eqref{fetetoftstrong} and \eqref{cvfistrong}, the following holds 
  	$$
  	\lim_{\ep\rightarrow 0}\int_{\Omega} \mu_\ep \lb \bfe_{\a\b}(\bfvarphi_\ep) \rb^2 \ dx = \int_{\Omega} \lb \bfe_{\a\b}((\bfvarphi^\star)') \rb^2 \ dm\otimes\L^2,
  	$$
we deduce that 
$$
\int_{\Omega}  \lb  \tfrac{E_{\a\b}\bfvarphi}{\nu\otimes\L^2} \rb^2 \ d\nu\otimes\L^2 \le \int_{\Omega} \lb \bfe_{\a\b}((\bfvarphi^\star)') \rb^2 \ dm\otimes\L^2 .
$$
Substituting $\bfv^\delta$ for $\bfvarphi$ and 
passing  to the limit  as $\delta \rightarrow 0$,  taking  \eqref{defBDnum}, \eqref{vdeltaBDnum} into account, we obtain   \eqref{ubyxi}. 
\end{proof}

  {\it Justification of \eqref{rearrangement}.} We fix $\bfe, \widetilde\bfe\in \SS^3$ and set $\bfsigma:= l(\tr\bfe) \bfI + 2\bfe$, $\widetilde\bfsigma:= l(\tr\widetilde\bfe) \bfI+ 2\widetilde\bfe$. 
  We have 
  
  \begin{equation}
\label{j1}	\begin{aligned}
  \bfsigma:\widetilde \bfe&
= \sum_{i=1}^3\sigma_{ii}\widetilde e_{ii}+ \sigma_{12}\widetilde \sigma_{12}+  \sigma_{13}\widetilde\sigma_{13} + 4e_{23}\widetilde e_{23}. 
  \end{aligned} 	\end{equation}

Noticing that

   \begin{equation}
  	\nonumber	\begin{aligned}
  &e_{11}=\tfrac{1}{l+2} (\sigma_{11}-l e_{22}-l e_{33}), \quad \widetilde e_{11}=\tfrac{1}{l+2} (\widetilde \sigma_{11}-l \widetilde e_{22}-l\widetilde  e_{33}),
  \\& \sigma_{22}= l e_{11}+ (l+2) e_{22}+ l e_{33}=   \tfrac{l}{l+2} (\sigma_{11}-l e_{22}-l e_{33})+ (l+2) e_{22}+ l e_{33},
    \\& \sigma_{33}= l e_{11}+  l  e_{22}+ (l+2) e_{33}=  \tfrac{l}{l+2} (\sigma_{11}-l e_{22}-l e_{33})+  l  e_{22}+ (l+2) e_{33},
  \end{aligned} 	\end{equation}

we obtain, by substitution,

  \begin{equation}
  	\nonumber	\begin{aligned}
  \sum_{i=1}^3\!\!\sigma_{ii}\widetilde e_{ii}\!&= \!
  \sigma_{11}\tfrac{1}{l+2}(\widetilde\sigma_{11}\!-\!l\widetilde e_{22}\!-\!l\widetilde e_{33}) 
  \!+\! \lp\tfrac{l}{l+2}(\sigma_{11}\!-\!l e_{22}\!-\!l e_{33}) \!+\! (l\!+\!2) e_{22}\! +\! l e_{33}\! \rp \! \widetilde e_{22} 
  \\&\hskip 3cm + \lp \tfrac{l}{l+2}(\sigma_{11}-l e_{22}-l e_{33}) + l e_{22} + (l+2)  e_{33}\rp \widetilde e_{33}
  \\&= \tfrac{1}{l+2} \sigma_{11}\widetilde\sigma_{11} + \tfrac{4(l+1)}{l+2}(e_{22} \widetilde e_{22}+ e_{33} \widetilde e_{33}) + \tfrac{2l}{l+2} (e_{22}\widetilde e_{33}+ e_{33} \widetilde e_{22}),
  \end{aligned} 	\end{equation}

yielding, by \eqref{j1}, 
  
 \begin{equation}
  	\nonumber	\begin{aligned}
  \bfsigma:\widetilde \bfe = \tfrac{1}{l+2}& \sigma_{11}\widetilde\sigma_{11} + + 2\sigma_{12}\widetilde e_{12}+ 2\sigma_{13}\widetilde e_{13}
  \\& + 4e_{23}\widetilde e_{23}   
 + \tfrac{4(l+1)}{l+2}(e_{22} \widetilde e_{22}+ e_{33} \widetilde e_{33}) + \tfrac{2l}{l+2} (e_{22}\widetilde e_{33}+ e_{33} \widetilde e_{22}).
  \end{aligned} 	\end{equation}
  
Substituting   $\bfe(\bfu_\e)$, $\bfe(\bfvarphi_\e)$, $ \tfrac{1}{\mu_\e} \bfsigma_\e(\bfu_\e)$, $ \tfrac{1}{\mu_\e} \bfsigma_\e(\bfvarphi_\e)$, respectively, for 
$\bfe$, $\widetilde \bfe $,   $  \bfsigma$, $  \widetilde\bfsigma$, we infer \eqref{rearrangement}. \qed

  \subsection{Proof of Proposition \ref{propvarphie}}\label{secproofpropvarphike}

The  proof  of Proposition \ref{propvarphie} lies in the  asymptotic analysis of the family of sequences $\lp\lp\bfvarphi_\e^k\rp_\e\rp_{k\in \NN}$,
 the results of which are presented in the next proposition whose proof is located
   in Section \ref{subsecprop5}.


\begin{proposition}\label{propvarphike} 
Let  $\bfv\in BD^{\nu,m}_0(\Omega)$, $\delta>0$,  $\bfsigma^\nu$ defined  by \eqref{defsigmanu}, and  $ \bfvarphi$, $\bfvarphi_\e^k$ 
respectively given   by \eqref{defvarphi}, \eqref{defvarphiek}. 
 Then   $\bfvarphi_\e^k$ belongs to $ H^1(\Omega;\RR^3)$ and satisfies 
 
		\begin{equation}
		\label{supkefikefini}	
		\sup_{k\in \NN; \ \e>0}  \int_\Omega | \bfvarphi^k_\e |^2 dm_\e\otimes \L^2  <\infty,
		\end{equation}  
\begin{equation}
  	\label{varphiektovarphi}
 \lim_{k\to\infty}	\sup_{\e>0}\int_{\Omega} \vert \bfvarphi_\e^k- \bfvarphi \vert  \ dx =0, 
  	\end{equation}
 
\begin{equation}
  	\label{lse11varphiek}
  	\begin{aligned}
  \limsup_{k\to\infty}	 \limsup_{\ep \rightarrow 0}  
	 \int_\O  \lb \bfsigma_\e(\bfvarphi_\e^k)\bfe_1\rb^2 d\nu_\e\otimes\L^2 
	 \le \int_\O \lb \bfsigma^\nu(\bfvarphi)\bfe_1 \rb^2 d\nu\otimes\L^2,
  	  	\end{aligned}
  	\end{equation}
 
\begin{equation}
  	\label{lseabvarphiek}
   \limsup_{k\to\infty} 	 \limsup_{\ep \rightarrow 0} \int_{\Omega}  \left\vert \bfe_{x'}(\lp \bfvarphi_\e^k\rp') \right\vert^2   dm_\e\otimes\L^2 \le	\int_{\Omega} \left\vert \bfe_{x'} ((\bfvarphi^\star)') \right\vert^2  dm\otimes\L^2.
  	\end{equation}

\end{proposition}

Let us  fix 
a  decreasing sequence of positive reals $(\a_k)_{k\in \NN}$  converging to $0$. 
By Proposition \ref{propvarphike}, there exists  of  a decreasing sequence of positive reals $(\e_k)_{k\in \NN}$ converging to $0$ as $k\to\infty$  and such that, for all $\e<\e_k$, 

\begin{equation}
\label{upstep2}
\hspace{-12 pt}\begin{aligned}
&\int_{\Omega} \vert \bfvarphi_\e^k- \bfvarphi \vert  \ dx \le  \a_k,
 \\&  \int_\O  \lb \bfsigma_\e(\bfvarphi_\e^k)\bfe_1\rb^2 d\nu_\e\otimes\L^2 
	 \le \int_\O \lb \bfsigma^\nu(\bfvarphi)\bfe_1 \rb^2 d\nu\otimes\L^2+  \a_k,
  \\&  \int_{\Omega}  \left\vert \bfe_{x'}( \bfvarphi_\e^k) \right\vert^2   dm_\e\otimes\L^2 \le	\int_{\Omega} \left\vert \bfe_{x'} (\bfvarphi^\star) \right\vert^2  dm\otimes\L^2+  \a_k.
\end{aligned}\end{equation}

%
  	
 Let 
  	$k_\e$ be   the unique integer such that $\e_{k_\e +1}\le \e<\e_{k_\e}$ (notice that $k_\e\to\infty$). We  set

\begin{equation} 
  	\label{defvarphie2}
	\bfvarphi_\e = \bfvarphi^{k_\ep}_\ep.
	\end{equation}
	
	By \eqref{Pe}, \eqref{defmenue}, \eqref{supkefikefini}, 
  \eqref{upstep2} and \eqref{defvarphie2},   the sequence $(\bfvarphi_\e)$ 
 %
%
strongly converges to $\bfvarphi$ in $L^1(\Omega;\RR^3)$ and satisfies the assumptions   \eqref{supFefiefini} and  \eqref{supfiemefini} of  Proposition \ref{propapriori}. Therefore,   the convergences (\ref{cvfinuenu}) and \eqref{cvfimem} hold. We deduce that 
 
  \begin{equation}
\nonumber
	\begin{aligned}
		\bfsigma_\e(\bfvarphi_\e)\bfe_1   \buildrel {\nu_\e\otimes\L^2, \nu\otimes\L^2}\over \rrrrightharpoonup \bfsigma^\nu(\bfvarphi)\bfe_1, 
	\qquad \bfe_{x'} ((\bfvarphi_\e)')  
   \buildrel {m_\e\otimes\L^2, m\otimes\L^2}\over \rrrrightharpoonup   \bfe_{x'}((\bfvarphi^\star )'). 
      \end{aligned}	\end{equation}

On the other hand,   \eqref{upstep2} and   \eqref{defvarphie2}  imply (since $k_\e\to\infty$)

 \begin{equation}
\nonumber
\hspace{-12 pt}\begin{aligned}
& \limsup_{\e\to0}  \int_\O  \lb \bfsigma_\e(\bfvarphi_\e)\bfe_1\rb^2 d\nu_\e\otimes\L^2 
	 \le \int_\O \lb \bfsigma^\nu(\bfvarphi)\bfe_1 \rb^2 d\nu\otimes\L^2,
  \\& \limsup_{\e\to0} \int_{\Omega}  \left\vert \bfe_{x'}\lp\lp \bfvarphi_\e\rp'\rp \right\vert^2   dm_\e\otimes\L^2 \le	\int_{\Omega} \left\vert \bfe_{x'} ((\bfvarphi^\star)') \right\vert^2  dm\otimes\L^2,
\end{aligned}\end{equation}

yielding   \eqref{cvfistrong}. 
 Proposition \ref{propvarphie} is proved provided we establish Proposition  \ref{propvarphike}.  \qed

\mn

 \subsection{Proof of Proposition \ref{propvarphike} } 
 \label{subsecprop5}

 The field $\bfvarphi_\e^k$ belongs to  $H^1(I_j^k\times\O';\RR^3)$  for all $j\in \{1,\ldots,n_k-1\}$,
 hence 
to prove that  $\bfvarphi_\e^k\in H^1(\Omega;\RR^3)$, it suffices to show that 
the traces of  $\bfvarphi_\e^k$ coincide on the common boundaries of 
$I_j^k\times\O'$ and $I_{j+1}^k\times\O'$, that is

\begin{equation}
\label{tracesegales}
\bfgamma_{\Sigma_{t_j^k}}^-(\bfvarphi_\e^k)= \bfgamma_{\Sigma_{t_j^k}}^+(\bfvarphi_\e^k)\quad 
\forall j\in \{1,\ldots,n_k-1\}.
\end{equation}

 By \eqref{varphi+=}, \eqref{defvarphi}, \eqref{defIjk}, the following holds
 
   \begin{equation}
\nonumber
  \begin{aligned}
   \varphi^+_1\! (t^k_{j},x')\!- \! \varphi^+_1 \!(t^k_{j \!-\!1},x')\!=\!
\tfrac{1}{l+2} \hskip-0,1cm\int_{I_j^k} \hskip-0,1cm (\bfsigma^\nu)_{11}(\bfvarphi)(s_1,x')    d\nu&(s_1)
  -\!\sum_{\a=2}^3 \tfrac{l}{l+2} \hskip-0,1cm\int_{I_j^k} \frac{\partial \varphi_\a}{\partial x_\a}(s_1,x') ds_1.
  \end{aligned}
  \end{equation}
 
On the other hand, noticing that  $\phi_\e^k( (t_j^k)^+)=0$ and $\phi_\e^k( (t_j^k)^-)=1$, 
we deduce from  \eqref{defvarphiek}  that, for all $x=(t_j^k,x') \in \Sigma_{t_j^k}$,  
  $(\bfgamma_{\Sigma_{t_j^k}}^+(\bfvarphi_\e^k))_1  (x)
=    \varphi^+_1 (t^k_j,x') $ and 

   \begin{equation}
\nonumber
  \begin{aligned}
&  
(\bfgamma_{\Sigma_{t_j^k}}^-(\bfvarphi_\e^k))_1  (x) =  \tfrac{1}{l+2}   \int_{I^k_{j}} 
 (\bfsigma^\nu)_{11}(\bfvarphi) (s_1,x')
  \ d\nu(s_1)   \\& \hspace{4.8cm} -\sum_{\a=2}^3  \tfrac{l}{l+2} \int_{t^k_{j-1}}^{t_j^k} \pd{\varphi_\a}{x_\a}(s_1,x')   ds_1  + \varphi^+_1 (t^k_{j-1},x').
  \end{aligned}
  \end{equation}
  
  The last  equations imply $(\bfgamma_{\Sigma_{t_j^k}}^+(\bfvarphi_\e^k))_1 = (\bfgamma_{\Sigma_{t_j^k}}^-(\bfvarphi_\e^k))_1$.
Similarly, by \eqref{varphi+=} and  \eqref{defIjk}  we have,   for $\a\in\{2,3\}$, 
 
   \begin{equation}
\nonumber
  \begin{aligned}
   \varphi^+_\a (t^k_{j},x')-  \varphi^+_\a (t^k_{j -1},x')=
 \int_{I_j^k} (\bfsigma^\nu)_{1\alpha}(\bfvarphi )(s_1,x') \ d\nu(s_1) - \int_{I_j^k} \pd{\varphi_1}{x_\alpha}(s_1,x') \ ds_1,
  \end{aligned}
  \end{equation}
  
  and,  by \eqref{defvarphiek}, $(\bfgamma_{\Sigma_{t_j^k}}^+(\bfvarphi_\e^k))_\a
 (t_{j}^k,x')  =    \varphi^+_\a (t^k_j,x')$ and 
    
     \begin{equation}
\nonumber
  \begin{aligned}
&  
(\bfgamma_{\Sigma_{t_j^k}}^-\!(\bfvarphi_\e^k)\!)_\a (t_{j}^k\!,\!x') \!=   \!\!  \int_{I^k_{j}} \!\!\!
 (\bfsigma^\nu( \bfvarphi ) )_{1\a}  (s_1\!,\!x' )
 d\nu(s_1)  \! -\!\! \int_{I_j^k} \! \pd{\varphi_1}{x_\alpha} ( s_1 , x' )  ds_1 \! + \!\varphi^+_\alpha  ( t^k_{j-1}\!,\!x' ),    \end{aligned}
  \end{equation}
  
yielding  $(\bfgamma_{\Sigma_{t_j^k}}^+(\bfvarphi_\e^k))_\a =(\bfgamma_{\Sigma_{t_j^k}}^-(\bfvarphi_\e^k))_\a $. Assertion \eqref{tracesegales} 
is proved.

In what follows, for all $x_1\in(0,L)$, we denote   by $j_{x_1}$ the unique integer satisfying 
  
  \begin{equation}
  \label{defjx1}
  \begin{aligned}
& x_1\in \left( t^k_{j_{x_1}-1}, t^k_{j_{x_1}}\right].
  \end{aligned} 
  \end{equation} 
  
  The next lemma plays a  crucial  role in the  proof of Proposition \ref{propvarphike}.

    \begin{lemma}\label{lemIxx1}
 We have 
  
  \begin{equation}
  \label{nueItonuI}
  \begin{aligned}
  &\lime \nu_\e(  I^k_j)=\nu(  I^k_j)   \hbox{ and }  
  \lim_{\ep \rightarrow 0} m_\ep (I^k_j)  = m(I^k_j) & &\forall k\!\in\! \NN, \  \forall j\!\in\! \{1,\ldots,n_k\}
  .
  \end{aligned}
  \end{equation}

 For all $k\in\NN$, the mapping $ x_1 \in (0,L] \to \nu(I^k_{j_{x_1}})$ defined by \eqref{defIjk}, \eqref{defjx1}  is Borel measurable and satisfies, for all $p\in (0,\infty)$,

  	\begin{equation}
  	\label{limintnuIk=0}
  	\begin{aligned}
	& \lime\int  \nu(I_{j_{x_1}}^k)  dm_\e(x_1) =\int  \nu(I_{j_{x_1}}^k)  dm (x_1),
	\\
  	&  \lim_{k\to \infty}\int_{0}^L \nu(I^k_{j_{x_1}})^p \ d\L^1(x_1)=  0,
  \quad  \lim_{k\to \infty}\int_{[0,L]} \nu(I^k_{j_{x_1}})^p \ dm(x_1)=  0 .
  	\end{aligned}
  	\end{equation}  	
  \end{lemma}

  \begin{proof}
  Since  
  $\nu(\partial I^k_j)=m(\partial I^k_j)=0$    for all $k\in \NN$,   $j\in \{1,\ldots,n_k\}$ (see \eqref{defAk}),
  the convergences   \eqref{nueItonuI} result   from (\ref{mueto}).
	By  \eqref{defAk} and (\ref{defjx1}), we have
  	
  	\begin{equation}
  	\label{Ikx1step}
  	\begin{aligned}
  	&  \nu(I^k_{j_{x_1}}) = \sum_{j=1}^{n_k} \nu \lp I^k_{j}\rp
  	\mathds{1}_{I^k_j} (x_1),
  	\end{aligned} 
  	\end{equation}
  	
  	hence the mapping $ x_1 \in (0,L] \to \nu(I^k_{j_{x_1}})$   is Borel-measurable and, by \eqref{nueItonuI},  
	
	\begin{equation}
  	\nonumber
  	\begin{aligned}
	 \lime\int  \nu(I_{j_{x_1}}^k)  dm_\e(x_1) &=\lime \sum_{j=1}^{n_k}  \nu(I_{j }^k)   m_\e (I_{j }^k)= \sum_{j=1}^{n_k}  \nu(I_{j }^k)   m (I_{j }^k)
	\\&=\int  \nu(I_{j_{x_1}}^k)  dm (x_1)
	 .
  	\end{aligned}
  	\end{equation}

%
%
 The measure $\nu$ is bounded and by \eqref{defAk}, for each fixed  $x_1\in (0,L]$, the sequence of sets    $(I^k_{j_{x_1}})_{k \in \NN}$  
is decreasing and 
satisfies 
$\bigcap_{k\in \NN} \downarrow I^k_{j_{x_1}} = \{{x_1}\},
$
therefore   $\lim_{k\to\infty}\nu(I^k_{j_{x_1}})  =\nu(\{{x_1}\})$. 
  Applying the Dominated Convergence Theorem, 
noticing that, by  (\ref{nocommonatom}),  $\L^1(\A_\nu)=m(\A_\nu)=0$  (see \eqref{defAnu}),  
we infer 
 	\begin{equation}
  	\nonumber
  	\begin{aligned}
  	&\lim_{k\to \infty}\int_{0}^L \nu(I^k_{j_{x_1}})^p \ d\L^1(x_1) 
  =  \int_{\A_\nu} \nu(\{x_1\})^p d\L^1(x_1)=0, 
   \\&\lim_{k\to \infty}\int_{[0,L]} \nu(I^k_{j_{x_1}})^p \ dm(x_1) 
	=   \int_{\A_\nu} \nu(\{x_1\})^p dm(x_1)
=0. \end{aligned}
  	\end{equation}
   \end{proof}


  {\bf Proof of \eqref{supkefikefini}.}
By  \eqref{defsigmanu}, \eqref{varphi+=},   \eqref{varphiborne},  \eqref{defvarphi}, 
we have, forall $x_1\in (0,L)$,  
\begin{flalign*}
\int_{\Omega'} \lb\bfvarphi^+ (t^k_{j_{x_1}-1},x') \rb^2 \ dx' & \le C   \int_{\Omega}  \lb \bfsigma^\nu (\bfvarphi) \rb^2 \ d\nu\otimes\L^2  +  C\int_{\Omega} \lb \pd{\bfvarphi }{x_\alpha} \rb^2 \ dx \le C ,
\end{flalign*}
therefore, by   \eqref{phidphi},   \eqref{defvarphiek}, and \eqref{defjx1}, 

\begin{equation}
\nonumber 
  	\begin{aligned}
&\sup_{x_1\in(0,L)}  \int_{\Omega'}  | \bfvarphi^k_{\ep}|^2 (x_1,x')  dx'   
\\&\quad \le  C \lp  \int_\O \!  \lb \bfsigma^\nu (\bfvarphi)\rb^2   \!d\nu\otimes\L^2  \!  +   \! \int_\O\! |  \pd{\bfvarphi }{x_\alpha}|^2  dx +   \!  \int_{\Omega'}    |\bfvarphi^+  (t^k_{j_{x_1}-1},x') |^2 dx'\rp
 \!\le\! C.
\end{aligned} 
  	\end{equation}
	
By integrating    over $(0,L)$ with respect to   $m_\e$,   we obtain  \eqref{supkefikefini}. \qed
   
  {\bf Proof of \eqref{varphiektovarphi}.} 
By  \eqref{varphi+=},  \eqref{varphiborne},  \eqref{defvarphi},
the following estimate holds for $x_1\in I_j^k $ (or equivalently for $j=j_{x_1}$):
 
  \begin{equation}
\label{varphi(x)-varphi(t)}
  	\begin{aligned}
	 \int_{\O'} \hskip-0,2cm | \bfvarphi^+&(x_1,x') -\bfvarphi^+(t_j,x')| dx'
	 \le C \int_{I_j^k\times\O'} \hskip-0,8cm |\bfsigma^\nu(\bfvarphi)| d\nu\otimes\L^2  + C\sum_{\a=2}^3 \int_{I_j^k\times\O'} \hskip-0,2cm | \pd{\bfvarphi}{x_\a}|  d\L^3
	\\& \le C \nu(I_j^k)^{\frac{1}{2}}  
	||\bfsigma^\nu(\bfvarphi)||_{L^2_{\nu\otimes\L^2}
	}^{\frac{1}{2}}
	 + C \lp  \sup_{j\in\{1,\ldots,n_k\}} \L^1(I_j^k)\rp^{\frac{1}{2}}\sum_{\a=2}^3 || \pd{\bfvarphi}{x_\a} ||_{L^2(\O)}^{\frac{1}{2}}
\\&	\le C \nu(I_j^k)^{\frac{1}{2}}  +  C\lp  \sup_{j\in\{1,\ldots,n_k\}} \L^1(I_j^k)\rp^{\frac{1}{2}}.
 \end{aligned} 
  	\end{equation}

 By integration over $(0,L)$ with respect to $\L^1$, taking  \eqref{defAk}, \eqref{defjx1}, \eqref{limintnuIk=0} into account,
 we infer

 \begin{equation}
\label{est1}
  	\begin{aligned}
\lim_{k\to\infty} 	\int_{\O} | \bfvarphi_\e^+(x_1,x')-\bfvarphi_\e^+(t_{j_{x_1}},x')| dx =0.
 \end{aligned} 
  	\end{equation}

By the same argument, we deduce from \eqref{phidphi},  \eqref{defvarphiek} that 

 \begin{equation}
\label{est2}
  	\begin{aligned}
\lim_{k\to\infty} 	\int_{\O} | \bfvarphi_\e^k(x_1,x')-\bfvarphi_\e^+(t_{j_{x_1}},x')| dx =0.
 \end{aligned} 
  	\end{equation}

Assertion \eqref{varphiektovarphi} results from \eqref{est1} and \eqref{est2}. \qed

  {\bf Proof of \eqref{lse11varphiek}.} Taking  \eqref{Pe}, \eqref{defsigmanu}, \eqref{phidphi} and  \eqref{defvarphiek} into account,  an elementary computation yields, for
  all $j\in \{1,\ldots,n_k\}$ and for $\L^3$-a.e.  $x\in I_j^k\times \O'$,

  \begin{equation}
  \label{sigma1}
\begin{aligned}
&
\begin{aligned}
 \bfsigma_\e(\bfvarphi_\e^k)(x)\bfe_1&=\mu_\e\lp l{\rm tr} (\bfe(\bfvarphi_\e^k)) \bfI+ 2\bfe(\bfvarphi_\e^k)\rp\bfe_1
 \\&= \frac{1}{\nu_\e(I_j^k)} \int_{I_j^k} \bfsigma^\nu(\bfvarphi)(s_1,x')  \bfe_1 d\nu(s_1) + \bfr_\e^k(x),
 \end{aligned}
 \end{aligned}
\end{equation}

where for $\a\in \{2,3\}$, 
  
  \begin{equation}
	\label{defrek}
	\begin{aligned}
\hspace{-0,5cm}	&\begin{aligned}
		&\frac{r_{\e1}^k}{\mu_\e}(x)  : = l \sum_{\a=2}^{3} \lp \pd{\varphi^+_{\a}}{x_\a}(t^k_{j-1},x') - \pd{\varphi^+_{\a}}{x_\a}(x_1,x') \rp  \\ 
		& \hspace{0.3cm}+ 2l \phi^k_\e(x_1)\sum_{\a=2}^{3}   \int_{I^k_{j}}  \pd{(\bfsigma^\nu(\bfvarphi))_{1\a}}{x_\a}(s_1,x') d\nu(s_1)   - l  \sum_{\a=2}^{3}  \int_{t^k_{j-1}}^{x_1}  \pd{^2\varphi_{1}}{x^2_{\a}} (s_1,x') ds_1,
	\end{aligned}
\\&\begin{aligned}
  	&   \frac{ r^k_{\e\a}}{\mu_\e}(x)  : =  \frac{1}{l+2} \phi^k_\ep(x_1) \int_{I^k_{j}}   \pd{(\bfsigma^\nu(\bfvarphi))_{11}}{x_\a} (s_1,x') \ d\nu(s_1) 
  \\& \hspace{0.3cm} - \frac{l}{l+2} \sum_{\b=2}^3\int_{t^k_{j-1}}^{x_1} \!\pd{^2\varphi_\b}{x_\b \partial x_\a}(s_1,x')   ds_1 
	 +  \pd{\varphi_1^+}{x_\a}\big(t^{k}_{j-1}, x' \big)  - \pd{\varphi_1}{x_\a}\big(x_1, x' \big). 
\end{aligned}
\end{aligned}
\end{equation}

We prove below that

\begin{equation}
\label{reto0}
	\begin{aligned}
  \limsup_{k\to\infty}	 \limsup_{\ep \rightarrow 0}   \int_\O   |\bfr_\e^k|^2 d \nu_\e\otimes\L^2=0 .
	\end{aligned}
\end{equation}

By   \eqref{defmenue},   \eqref{defIjk} and \eqref{sigma1}, we have
\begin{equation}
\label{eqe1}
	\begin{aligned}
		  \int_\O  \lb \bfsigma_\e(\bfvarphi_\e^k)\bfe_1-\bfr_\e^k\rb^2 \hskip-1,5cm &\hskip 1,5cm d\nu_\e\otimes\L^2  
		\\&= \sum_{j=1}^{n_k} \frac{  \int_{I^k_j}  \mu^{-1}_\ep(x_1)  dx_1  }{(\nu_\ep(I^k_j))^2}  \int_{\Omega'}   
		 \left\vert \int_{I^k_j}  \bfsigma^\nu(\bfvarphi)\bfe_1 (s_1,x')  d\nu(s_1) \right\vert^2  dx'  
		\\& \le \sum_{j=1}^{n_k}   \frac{\nu(I^k_j)}{\nu_\ep(I^k_j)}  \int_{I^k_j\times \Omega'}  \left\vert   \bfsigma^\nu(\bfvarphi)\bfe_1 \right\vert^2 \  d\nu\otimes\L^2.
	\end{aligned}
\end{equation}

Assertion \eqref{lse11varphiek} follows from \eqref{nueItonuI}, \eqref{reto0},  \eqref{eqe1}. \qed

{\it Proof of \eqref{reto0}.} A computation analogous to \eqref{varphi(x)-varphi(t)} yields for $x_1\in I_j^k$, taking \eqref{varphiborne} into account, 

 \begin{equation}
 \label{uno}
  	\begin{aligned}
	 \int_{\O'} \hskip-0,1cm \lb \pd{\bfvarphi^+}{x_\a} (x_1,x') -\pd{\bfvarphi^+}{x_\a}(t_j,x')\rb^2 dx'
	 \le C \nu(I_j^k)  +  C   \sup_{j\in\{1,\ldots,n_k\}}\hskip-0,3cm  \L^1(I_j^k) .
 \end{aligned} 
  	\end{equation}
	
		Similarly, by \eqref{fdeltainLp}, 
	
 \begin{equation}
 \label{due}
  	\begin{aligned}
 \int_{\O'}\lb 	 \int_{I^k_{j}}  \pd{\bfsigma^\nu }{x_\a}(s_1,x') d\nu(s_1) \rb^2 dx'&\le  C\nu(I_j^k)  \lb\lb\pd{\bfsigma^\nu}{x_\a}\rb\rb^2_{L^2_{\nu\otimes\L^2}(\O)} 	 
 \le C \nu(I_j^k)  ,
 \end{aligned} 
  	\end{equation}
	
	 \begin{equation}
 \label{tre}
  	\begin{aligned}
 \int_{\O'}\lb 	\int_{t^k_{j-1}}^{x_1} \!\pd{^2\varphi_\b}{x_\b \partial x_\a}(s_1,x')   ds_1 \rb^2 dx'&\le  C \sup_{j\in\{1,\ldots,n_k\}} \hskip-0,3cm \L^1(I_j^k)  
 \lb\lb\pd{^2\varphi_\b}{x_\b \partial x_\a}\rb\rb^2_{L^2 (\O)} 	
 \\& \le C\sup_{j\in\{1,\ldots,n_k\}} \hskip-0,3cm \L^1(I_j^k)   .
 \end{aligned} 
  	\end{equation}

Collecting \eqref{phidphi},  \eqref{defrek}, \eqref{uno}, \eqref{due}, \eqref{tre}, noticing that $\mu_\e^2\nu_\e=m_\e$, we infer 

 \begin{equation}
 \label{quatro}
  	\begin{aligned}
 \int_\O |\bfr_\e^k|^2 d\nu_\e\otimes\L^2 \le   C\int  \nu(I_{j_{x_1}}^k)  dm_\e(x_1)+  C\sup_{j\in\{1,\ldots,n_k\}}\hskip-0,3cm  \L^1(I_j^k)  m_\e ((0,L)) .
 \end{aligned} 
  	\end{equation}

Assertion \eqref{reto0} results from 	\eqref{defAk}, \eqref{limintnuIk=0}, \eqref{quatro}. \qed

 {\bf Proof of \eqref{lseabvarphiek}.}
By  (\ref{defvarphiek}) we have, for $x_1\in I_j^k$,

  	\begin{equation}
  	\label{ex'=}
  	\begin{aligned}
  	&\bfe_{x'}(\bfvarphi_\e^k)(x)= \bfe_{x'}(\bfvarphi^+)(t_{j-1}^k,x')+ \bfR_\e^k(x),
	\\& \bfR_\e^k(x):= \phi_\e^k(x_1)\int_{I_j^k} \bfe_{x'}\lp \bfsigma^\nu(\bfvarphi)\bfe_1\rp(s_1,x') d\nu(s_1) 
	\\&\hskip4cm - \sum_{\a,\b=2}^3 \int_{t_{j-1}^k}^{x_1} \tfrac{\partial^2\varphi_1}{\partial x_\a\partial x_\a} (s_1,x')ds_1 \bfe_\a\odot\bfe_\b.
	  	\end{aligned}
  	\end{equation}

We deduce from  \eqref{phidphi}, \eqref{due}, \eqref{tre}, \eqref{ex'=},  that $ \int_\O  \lb\bfR_\e^k\rb^2(x)dm_\e$ is bounded from above by the left-hand side of \eqref{quatro}, hence,
by \eqref{defAk}, \eqref{limintnuIk=0}, 

\begin{equation}
  	\label{Reto0}
  	\begin{aligned}
  	\lim_{k\to\infty}  \sup_{\e>0} \int_\O  \lb\bfR_\e^k\rb^2 dm_\e\otimes \L^2 =0.
	  	\end{aligned}
  	\end{equation}

 By   (\ref{defmenue}) and (\ref{defIjk}) we have 
  	
  	\begin{equation}
  	\nonumber
  	\begin{aligned}
  	\int_{\Omega}  \lb \bfe_{x'} (\bfvarphi^+) \rb^2  (t^k_{j_{x_1}-1},x')   \ dm_\e\otimes\L^2 
  	& = \sum_{j=1}^{n_k}  m_\ep(I^k_j)    \int_{\Omega'}   \lb \bfe_{x'} (\bfvarphi^+) \rb^2(t^k_{j-1},x')  \ dx',
  	\end{aligned}
  	\end{equation}

yielding, by    \eqref{nueItonuI},

  	\begin{equation}
  	\label{Sg}
  	\begin{aligned}
  	\lim_{\ep \rightarrow 0}\!\int_{\Omega} \! \lb \bfe_{x'}\! (\bfvarphi^+\!) \!\rb^2  (t^k_{j_{x_1}\!-1},x') &  dm_\e\otimes\L^2 \!
  	  = \!\int_{\Omega}\!\lb  \bfe_{x'} (\bfvarphi^+) \rb^2\!(t^k_{j_{x_1}\!-1},x')  dm\otimes\L^2\!.
  	\end{aligned}
  	\end{equation}

By (\ref{defAk}) and (\ref{defjx1}), for all $x_1\in (0,L)$,  the sequence $(t_{j_{x_1-1}}^k)_{k\in\NN}$ converges to $x_1$ from below as $k \rightarrow \infty$. Therefore, by \eqref{610},  for each $x \in \Omega$ the following holds
\begin{equation}
\label{ptcv}
 \lim_{k \rightarrow \infty} \lb
\bfe_{x'}(\bfvarphi^+)
 \rb^2(t^k_{j_{x_1}-1},x') =
\lb
 \bfe_{x'}(\bfvarphi^-)
 \rb^2(x).
\end{equation}  	

On the other hand, by   \eqref{varphi+=}, $\big\vert \bfe_{x'}(\bfvarphi^+)  \big\vert^2(t^k_{j_{x_1}-1},x')   \le g(x)$, where

\begin{equation}
\begin{aligned}
g(x)\!:=   \!  \!\int_{(0,L)} \hskip-0,5cm \lb  \bfe_{x'}(\bfsigma^\nu\!(\bfvarphi) \bfe_1)
 \rb^2 (s_1,x')  d\nu(s_1) 
\!+  \!\!\sum_{\a,\b=2}^3 \!\int_{0}^{L} \lb\pd{^2\varphi_1}{x_\alpha\partial x_\b}\rb^2\!(s_1,x') ds_1 .
\end{aligned} 
\nonumber
\end{equation}

We deduce from   \eqref{Xiborne} and \eqref{varphiborne} that 
$g\in L^1_{m\otimes\L^2}(\Omega)$, and then from  \eqref{Sg}, \eqref{ptcv} and 
 the Dominated Convergence Theorem, 
 that 

\begin{equation}
\label{limex'varphi}
  	\begin{aligned}
  	\lim_{k \rightarrow \infty}  \int_{\Omega}\big\vert  \bfe_{x'}(\bfvarphi^+)  \big\vert^2(t^k_{j_{x_1}-1},x')  \ dm\otimes\L^2
 &=   \int_{\Omega} \big\vert   \bfe_{x'}(\bfvarphi^-)   \big\vert^2  \ dm\otimes\L^2.
  	\end{aligned}
  	\end{equation}
	
By \eqref{nocommonatom} and \eqref{decEfi} we have   $|\bfE\bfvarphi|(\Sigma_{x_1})=0$ for $m$-a.e. $x_1\in (0,L)$, 
therefore 	Assertion \eqref{phipm=phistar} implies that 
 $ \bfe_{x'}(\bfvarphi^-)= \bfe_{x'}(\bfvarphi^\star)$ $m\otimes\L^2$-a.e..
 Collecting \eqref{ex'=},  \eqref{Reto0}, \eqref{Sg}, \eqref{limex'varphi}, and the last equation,  the  assertion \eqref{lseabvarphiek}  is proved.\qed

\subsection{Sketch proof of Proposition \ref{thsystem} }\label{secsketchproofthsystem}
Repeating the argument of the proof of  Proposition \ref{propaprioriu}, we establish the apriori estimates 
	\begin{equation*} 
	\begin{aligned}
	&\sup_{\e>0}  \int_\Omega |\bfu_\e|^2 dm_\e\otimes \L^2 + \int_\Omega |\bfu_\e| dx+  \int_\Omega \mu_\ep \left\vert \bfnabla\bfu_\e \right\vert^2 dx
	<\infty, 
	\end{aligned}
	\end{equation*}
	
and deduce, up to a subsequence,  
the following convergences (analogous to \eqref{cvu})
\begin{equation}
\label{cvuweakheat}
\begin{aligned}
 &\bfu_\e \ \buildrel\star\over \rightharpoonup  \bfu  \quad   \hbox{weakly*  in }  BV(\Omega;\RR^n)    \hbox{ for some } 
	\bfu\in BV^{\nu,m}_0(\Omega),
\\& \mu_\e  (\hbox{\rm\bfC}\bfnabla \bfu_\e)\bfe_1
\buildrel {\nu_\e\otimes\L^2, \nu\otimes\L^2}\over \rrrrightharpoonup
 (\hbox{\rm\bfC}\tfrac{\bfD\bfu \ \ }{\nu\otimes\L^{d-1}})\bfe_1,
\quad  \bfnabla_{x'}\bfu_\e  
\buildrel {m_\e\otimes\L^2, m\otimes\L^2}\over \rrrrightharpoonup    \bfnabla_{x'}\bfu^\star,
\end{aligned}	\end{equation}

where $BV^{\nu,m}_0(\Omega)$ and 
$\bfnabla_{x'} \bfv$  are defined by  \eqref{defBVnum0} and  \eqref{defnablax'Xi}.
Fixing  $\bfv \in BV^{\nu,m}_0(\Omega)$, $\delta>0$, $k\in \NN^*$,  we set    $\bfvarphi = \bfv^\delta$ and  

\begin{equation}
\nonumber
\begin{aligned}
\bfvarphi^k_{\ep } (x) &:=  \left(  \int_{I^k_{j_{x_1}}}(\bfT^{-1} \hbox{\rm\bfC\,}\tfrac{\bfD\bfvarphi \ \ }{\nu\otimes\L^{d-1}})\bfe_1 (s_1,x') \ d\nu(s_1) \right)  \phi^k_\ep(x_1) \\
& \hspace{1.8cm} -   \int_{t^k_{j_{x_1}-1}}^{x_1} (\bfT^{-1} \hbox{\rm\bfC\,}\bfnabla_{x'}\bfvarphi)\, \bfe_1 (s_1,x')   ds_1  + \bfvarphi^+ (t^k_{j_{x_1}-1},x'). 
\end{aligned}
\end{equation}

Mimicking   propositions \ref{propvarphie} and \ref{propvarphike}, we exhibit  a sequence   $\bfvarphi_\e(=\bfvarphi_\e^{k_\e})$ 
satisfying 

\begin{equation}
\label{cvfistrongheat}
\begin{aligned}
& \hspace{3.5cm}\lim_{\ep \rightarrow 0} \int_{\Omega} \vert \bfvarphi_\e- \bfvarphi \vert  \ dx=0, \hspace{2.2cm} \\
& \mu_\e  (\hbox{\rm\bfC\,} \bfnabla \bfvarphi_\e )\bfe_1 \buildrel{\nu_\e\otimes\L^2, \nu \otimes\L^2}\over  \rrrrightarrow   (\hbox{\rm\bfC} \tfrac{\bfD\bfvarphi \ \ }{\nu\otimes\L^{d-1}})\bfe_1, \qquad  \bfnabla_{x'} \bfvarphi_\e \buildrel{m_\e\otimes\L^2, m \otimes\L^2}\over  \rrrrightarrow \bfnabla_{x'}\bfvarphi^\star.
\end{aligned}
\end{equation}

Multiplying  \eqref {PeSyst}  by $\bfvarphi_\e$,  integrating  by parts, and applying     the formula 
\begin{equation}
\label{rearrangementsystem}
\begin{aligned}
\hbox{\rm\bfC}\bfnabla\bfu_\e : \bfnabla \bfvarphi_\e \hskip-0,1cm =  \hskip-0,1cm(\bfT^{-1} \hbox{\rm\bfC\,}\bfnabla\bfu_\e) \bfe_1 \! \cdot \!(\hbox{\rm\bfC}\bfnabla\bfvarphi_\e) \bfe_1  - (\bfT^{-1}\hbox{\rm\bfC\,}\bfnabla_{x'}\bfu_\e)\, \bfe_1 \cdot(\hbox{\rm\bfC}\bfnabla_{x'}\bfvarphi_\e)\, \bfe_1  \\
&\hspace{-90pt} +\hbox{\rm\bfC\,}\bfnabla_{x'}\bfu_\e : \bfnabla_{x'}\bfvarphi_\e,
\end{aligned}
\end{equation}

proved below, we obtain 

\begin{equation}
\nonumber
\begin{aligned}
 \hskip-1,5cm \int_\Omega \bff \cdot \bfvarphi_\e dx =   
&\int_\Omega \mu_\e (\bfT^{-1}\hbox{\rm\bfC\,} \bfnabla \bfu_\e) \bfe_1  \cdot \mu_\e (\hbox{\rm\bfC\,}\bfnabla\bfvarphi_\e)\bfe_1 \ d\nu_\e\otimes\L^2
\\&\hskip-1,5cm +  \int_\Omega -(\bfT^{-1}\hbox{\rm\bfC\,} \bfnabla_{x'} \bfu_\e ) \bfe_1 \cdot   (\hbox{\rm\bfC} \bfnabla_{x'} \bfvarphi_\e ) \bfe_1 +\hbox{\rm\bfC\,} \bfnabla_{x'} \bfu_\e  \cdot    \bfnabla_{x'} \bfvarphi_\e \ dm_\e\otimes \L^2.
\end{aligned}
\end{equation}

\ni Passing to the limit as  $\e\to 0$ in accordance with   \eqref{cvuweakheat} and  \eqref{cvfistrongheat}, we find
$$
a( \bfu,\bfvarphi)=\int_\Omega  \bfu \cdot \bfvarphi\ dx,
$$ 
where

 \begin{equation}
\nonumber
 \begin{aligned}
\hskip-0,3cm  a(\bfu,\bfvarphi) : =   
 &  \int_\Omega    ({\bfT^{-1} } \hbox{\rm\bfC\,} \tfrac{\bfD\bfu \ \ }{\nu\otimes\L^{d-1}}) \bfe_1    \cdot (\hbox{\rm\bfC\,}\tfrac{\bfD\bfvarphi) \ \ }{\nu\otimes\L^{d-1}})\bfe_1   \ d\nu \otimes \L^{d-1} 
 \\&-   \int_\Omega      \hskip-0,2cm  ({ \bfT^{-1} } \hbox{\rm\bfC\,} \bfnabla_{x'} \bfu^\star) \bfe_1    \!\cdot\!
(\hbox{\rm\bfC} \bfnabla_{x'} \bfvarphi)^\star ) \bfe_1   + \hbox{\rm\bfC\,} \bfnabla_{x'} \bfu^\star \!   :   \! \bfnabla_{x'} \bfvarphi)^\star \ dm \otimes  \L^{d-1}.
 \end{aligned}
 \end{equation} 

An elementary computation  shows that $a(\cdot,\cdot)$ is also given by \eqref{defaSyst}.
The rest of the proof is similar to that of Theorem \ref{th}.

{\bf Proof of \eqref{rearrangementsystem}.}
\def\tC {\hbox{\rm\bfC}}
 Noticing that  $\bfT$ defined by \eqref{defT} 
  satisfies
$$
(\bfT\bfnabla\bfv)\bfe_1
= 
  (\hbox{\rm\bfC}\bfnabla\bfv )\,\bfe_1 - (\hbox{\rm\bfC}\bfnabla_{x'}\bfv)\,\bfe_1  ,
$$
and taking the invertibility of $\bfT$ and   the symmetry  of $\bfT^{-1}$ and $\tC $ into account,  we obtain  

\begin{equation}
\nonumber
\begin{aligned}
\tC  \bfnabla\bfu \!:\!  \bfnabla \bfv &=  (\tC  \bfnabla \bfu)\bfe_1\!\cdot\!(\bfnabla\bfv)\bfe_1
+ \tC \bfnabla \bfu \!:\! \bfnabla_{x'}\bfv=\!   (\tC  \! \bfnabla \bfu)\bfe_1\!\cdot\! (\bfnabla\bfv)\bfe_1
\! + \!  \bfnabla \bfu \!:\! \tC \! \bfnabla_{x'}\bfv
\\& = 
 (\tC  \bfnabla \bfu)\bfe_1\!\cdot\! (\bfnabla\bfv)\bfe_1
+   (\bfnabla \bfu)\bfe_1  \!\cdot\!  (\tC \bfnabla_{x'}\bfv)\bfe_1 + \bfnabla_{x'}\bfu \!:\! \tC \bfnabla_{x'}\bfv
\\& = 
(\tC  \bfnabla \bfu)\bfe_1\!\cdot\! \bfT^{-1}((\tC \bfnabla\bfv)\bfe_1-(\tC \bfnabla_{x'}\bfv)\bfe_1)
\\&\hskip 1,3cm +  \bfT^{-1}((\tC \bfnabla\bfu)\bfe_1-(\tC \bfnabla_{x'}\bfu)\bfe_1)  \!\cdot\!  (\tC \bfnabla_{x'}\bfv)\bfe_1 + \bfnabla_{x'}\bfu \!:\! \tC \bfnabla_{x'}\bfv
\\& = \! 
(\bfT^{-1}\tC \bfnabla \bfu)\bfe_1\!\cdot\! ( \tC  \bfnabla\bfv)\bfe_1 
 \! -(\bfT^{-1}\!  \tC \bfnabla_{x'}\bfu)\bfe_1  \!\cdot\!  (\tC  \bfnabla_{x'}\bfv)\bfe_1\!  +\!  \bfnabla_{x'}\bfu \!:\! \tC  \bfnabla_{x'}\bfv.
\end{aligned}
\end{equation}
\qed

\section*{Acknowledgements}
 The work of Michel Bellieud and Shane Cooper  was financed by the French National Research Agency (ANR)(Project ``Blanc", ANR-13-BS03-0009-03). The work of Shane Cooper was also financed by the Engineering and Physical Sciences Research Council (EPSRC)(Project ``Operator asymptotics, a new approach to length-scale interactions in metamaterials", grant EP/M017281/1).  

\end{document}